\documentclass[12pt,reqno]{amsart}
\expandafter\ifx\csname ProvidesPackage\endcsname\relax\toks0=\expandafter{%
\fi\ProvidesPackage{diagrams}[2024/11/20 v3.97 Paul Taylor's commutative
diagrams]
\toks0=\bgroup}
\ifx\diagram\isundefined\else\expandafter\endinput
\fi

\edef\cdrestoreat{
\noexpand\catcode`\noexpand\@=\the\catcode`\@
\noexpand\catcode`\noexpand\#=\the\catcode`\#
\noexpand\catcode`\noexpand\$=\the\catcode`\$
\noexpand\catcode`\noexpand\<=\the\catcode`\<
\noexpand\catcode`\noexpand\>=\the\catcode`\>
\noexpand\catcode`\noexpand\:=\the\catcode`\:
\noexpand\catcode`\noexpand\;=\the\catcode`\;
\noexpand\catcode`\noexpand\!=\the\catcode`\!
\noexpand\catcode`\noexpand\?=\the\catcode`\?
\noexpand\catcode`\noexpand\+=\the\catcode'53
}\catcode`\@=11 \catcode`\#=6 \catcode`\<=12 \catcode`\>=12 \catcode'53=12
\catcode`\:=12 \catcode`\;=12 \catcode`\!=12 \catcode`\?=12

\ifx\diagram@help@messages\CD@qK\let\diagram@help@messages y\fi

\def\cdps@Rokicki#1{\special{ps:#1}}\let\cdps@dvips\cdps@Rokicki\let
\cdps@RadicalEye\cdps@Rokicki\let\CD@HB\cdps@Rokicki\let\CD@IK\cdps@Rokicki
\let\CD@HB\cdps@Rokicki
\def\cdps@Bechtolsheim#1{\special{dvitps: Literal "#1"}}%
\let\cdps@dvitps\cdps@Bechtolsheim\let\cdps@IntegratedComputerSystems
\cdps@Bechtolsheim
\def\cdps@Clark#1{\special{dvitops: inline #1}}
\let\cdps@dvitops\cdps@Clark
\let\cdps@OzTeX\empty\let\cdps@oztex\empty\let\cdps@Trevorrow\empty
\def\cdps@Coombes#1{\special{ps-string #1}}

\count@=\year\multiply\count@12 \advance\count@\month
\ifnum\count@>24382 
\ifnum\count@>24385 
\errhelp{You may press RETURN and carry on for the time being.}\errmessage{! remain
compatible with your files. Please get it NOW.}\fi\fi

\def\CD@DE{\global\let}\def\CD@RH{\outer\def}

{\escapechar\m@ne\xdef\CD@o{\string\{}\xdef\CD@yC{\string\}}
\catcode`\&=4 \CD@DE\CD@Q=&\xdef\CD@S{\string\&}
\catcode`\$=3 \CD@DE\CD@k=$\CD@DE\CD@ND=$
\xdef\CD@nC{\string\$}\gdef\CD@LG{$$}
\catcode`\_=8 \CD@DE\CD@lJ=_
\obeylines\catcode`\^=7 \CD@DE\@super=^
\ifnum\newlinechar=10 \gdef\CD@uG{^^J}
\else\ifnum\newlinechar=13 \gdef\CD@uG{^^M}
\else\ifnum\newlinechar=-1 \gdef\CD@uG{^^J}
\else\CD@DE\CD@uG\space\expandafter\message{! input error: \noexpand
\newlinechar\space is ASCII \the\newlinechar, not LF=10 or CR=13.}
\fi\fi\fi}

\mathchardef\lessthan='30474 \mathchardef\greaterthan='30476

\ifx\tenln\CD@qK
\font\tenln=line10\relax
\fi\ifx\tenlnw\CD@qK\ifx\tenln\nullfont\let\tenlnw\nullfont\else
\font\tenlnw=linew10\relax
\fi\fi

\ifx\inputlineno\CD@qK\csname newcount\endcsname\inputlineno\inputlineno\m@ne
\fi

\def\cd@shouldnt#1{\CD@KB{* THIS (#1) SHOULD NEVER HAPPEN! *}}

\def\get@round@pair#1(#2,#3){#1{#2}{#3}}
\def\get@square@arg#1[#2]{#1{#2}}
\def\CD@AE#1{\CD@PK\let\CD@DH\CD@@E\CD@@E#1,],}
\def\CD@m{[}\def\CD@RD{]}\def\commdiag#1{{\let\enddiagram\relax\diagram[]#1%
\enddiagram}}

\def\CD@BF{{\ifx\CD@EH[\aftergroup\get@square@arg\aftergroup\CD@YH\else
\aftergroup\CD@JH\fi}}
\def\CD@CF#1#2{\def\CD@YH{#1}\def\CD@JH{#2}\futurelet\CD@EH\CD@BF}

\def\CD@KK{|}

\def\CD@PB{
\tokcase\CD@DD:\CD@y\break@args;\catcase\@super:\upper@label;\catcase\CD@lJ:%
\lower@label;\tokcase{~}:\middle@label;
\tokcase<:\CD@iF;
\tokcase>:\CD@iI;
\tokcase(:\CD@BC;
\tokcase[:\optional@;
\tokcase.:\CD@JJ;
\catcase\space:\eat@space;\catcase\bgroup:\positional@;\default:\CD@@A
\break@args;\endswitch}

\def\switch@arg{
\catcase\@super:\upper@label;\catcase\CD@lJ:\lower@label;\tokcase[:\optional@
;
\tokcase.:\CD@JJ;
\catcase\space:\eat@space;\catcase\bgroup:\positional@;\tokcase{~}:%
\middle@label;
\default:\CD@y\break@args;\endswitch}

\let\CD@tJ\relax\ifx\protect\CD@qK\let\protect\relax\fi\ifx\AtEndDocument
\CD@qK\def\CD@PG{\CD@gB}\def\CD@GF#1#2{}\else\def\CD@PG#1{\edef\CD@CH{#1}%
\expandafter\CD@oC\CD@CH\CD@OD}\def\CD@oC#1\CD@OD{\AtEndDocument{\typeout{%
\CD@tA: #1}}}\def\CD@GF#1#2{\gdef#1{#2}\AtEndDocument{#1}}\fi\def\CD@ZA#1#2{%
\def#1{\CD@PG{#2\CD@mD\CD@W}\CD@DE#1\relax}}\def\CD@uF#1\repeat{\def\CD@p{#1}%
\CD@OF}\def\CD@OF{\CD@p\relax\expandafter\CD@OF\fi}\def\CD@sF#1\repeat{\def
\CD@q{#1}\CD@PF}\def\CD@PF{\CD@q\relax\expandafter\CD@PF\fi}\def\CD@tF#1%
\repeat{\def\CD@r{#1}\CD@QF}\def\CD@QF{\CD@r\relax\expandafter\CD@QF\fi}\def
\CD@tG#1#2#3{\def#2{\let#1\iftrue}\def#3{\let#1\iffalse}#3}\if y%
\diagram@help@messages\def\CD@rG#1#2{\csname newtoks\endcsname#1#1=%
\expandafter{\csname#2\endcsname}}\else\csname newtoks\endcsname\no@cd@help
\no@cd@help{See the manual}\def\CD@rG#1#2{\let#1\no@cd@help}\fi\chardef\CD@lF
=1 \chardef\CD@lI=2 \chardef\CD@MH=5 \chardef\CD@tH=6 \chardef\CD@sH=7
\chardef\CD@PC=9 \dimendef\CD@hI=2 \dimendef\CD@hF=3 \dimendef\CD@mF=4
\dimendef\CD@mI=5 \dimendef\CD@wJ=6 \dimendef\CD@tI=8 \dimendef\CD@sI=9
\skipdef\CD@uB=1 \skipdef\CD@NF=2 \skipdef\CD@tB=3 \skipdef\CD@ZE=4 \skipdef
\countdef\CD@JC=9 \countdef\CD@gD=8 \countdef\CD@A=7 \def\sdef#1#2{\def#1{#2}%
}\def\CD@L#1{\expandafter\aftergroup\csname#1\endcsname}\def\CD@RC#1{%
\expandafter\def\csname#1\endcsname}\def\CD@sD#1{\expandafter\gdef\csname#1%
\endcsname}\def\CD@vC#1{\expandafter\edef\csname#1\endcsname}\def\CD@nF#1#2{%
\expandafter\let\csname#1\expandafter\endcsname\csname#2\endcsname}\def\CD@EE
#1#2{\expandafter\CD@DE\csname#1\expandafter\endcsname\csname#2\endcsname}%
\def\CD@AK#1{\csname#1\endcsname}\def\CD@XJ#1{\expandafter\show\csname#1%
\endcsname}\def\CD@ZJ#1{\expandafter\showthe\csname#1\endcsname}\def\CD@WJ#1{%
\expandafter\showbox\csname#1\endcsname}\def\CD@tA{Commutative Diagram}\edef
\CD@kH{\string\par}\edef\CD@dC{\string\diagram}\edef\CD@HD{\string\enddiagram
}\edef\CD@EC{\string\\}\def\CD@eF{LaTeX}\ifx\@ignoretrue\CD@qK\expandafter
\CD@tG\csname if@ignore\endcsname\ignore@true\ignore@false\def\@ignoretrue{%
\global\ignore@true}\def\@ignorefalse{\global\ignore@false}\fi

\def\CD@g{{\ifnum0=`}\fi}\def\CD@wC{\ifnum0=`{\fi}}\def\catcase#1:{\ifcat
\noexpand\CD@EH#1\CD@tJ\expandafter\CD@kC\else\expandafter\CD@dJ\fi}\def
\tokcase#1:{\ifx\CD@EH#1\CD@tJ\expandafter\CD@kC\else\expandafter\CD@dJ\fi}%
\def\CD@kC#1;#2\endswitch{#1}\def\CD@dJ#1;{}\let\endswitch\relax\def\default:%
#1;#2\endswitch{#1}\ifx\at@\CD@qK\def\at@{@}\fi\edef\CD@P{\CD@o pt\CD@yC}%
\CD@RC{\CD@P>}#1>#2>{\CD@z\rTo\sp{#1}\sb{#2}\CD@z}\CD@RC{\CD@P<}#1<#2<{\CD@z
\lTo\sp{#1}\sb{#2}\CD@z}\CD@RC{\CD@P)}#1)#2){\CD@z\rTo\sp{#1}\sb{#2}\CD@z}%
\CD@RC{\CD@P(}#1(#2({\CD@z\lTo\sp{#1}\sb{#2}\CD@z}
\def\CD@O{\def\endCD{\enddiagram}\CD@RC{\CD@P A}##1A##2A{\uTo<{##1}>{##2}%
\CD@z\CD@z}\CD@RC{\CD@P V}##1V##2V{\dTo<{##1}>{##2}\CD@z\CD@z}\CD@RC{\CD@P=}{%
\CD@z\hEq\CD@z}\CD@RC{\CD@P\CD@KK}{\vEq\CD@z\CD@z}\CD@RC{\CD@P\string\vert}{%
\vEq\CD@z\CD@z}\CD@RC{\CD@P.}{\CD@z\CD@z}\let\CD@z\CD@Q}\def\CD@IE{\let\tmp
\CD@JE\ifcat A\noexpand\CD@CH\else\ifcat=\noexpand\CD@CH\else\ifcat\relax
\noexpand\CD@CH\else\let\tmp\at@\fi\fi\fi\tmp}\def\CD@JE#1{\CD@nF{tmp}{\CD@P
\string#1}\ifx\tmp\relax\def\tmp{\at@#1}\fi\tmp}\def\CD@z{}\begingroup
\aftergroup\def\aftergroup\CD@T\aftergroup{\aftergroup\def\catcode`\@\active
\aftergroup @\endgroup{\futurelet\CD@CH\CD@IE}}\def\CD@uK#1{\CD@nF{x}{tex_#1:%
D}\ifx\x\relax\else\CD@nF{#1}{x}\fi} \def\CD@vK{\CD@uK{par}\CD@uK{everypar}%
\CD@uK{noindent}\CD@uK{parskip}\CD@uK{unskip}\CD@uK{hskip}\CD@uK{indent}}

\newcount\CD@uA\newcount\CD@vA\newcount\CD@wA\newcount\CD@xA\newdimen\CD@OA
\newdimen\CD@PA\CD@tG\CD@gE\CD@@A\CD@y\CD@tG\CD@hE\CD@EA\CD@BA\newdimen\CD@RA
\newdimen\CD@SA\newcount\CD@yA\newcount\CD@zA\newdimen\CD@QA\newbox\CD@DA
\CD@tG\CD@lE\CD@dA\CD@bA\newcount\CD@LH\newcount\CD@TC\def\CD@V#1#2{\ifdim#1<%
#2\relax#1=#2\relax\fi}\def\CD@X#1#2{\ifdim#1>#2\relax#1=#2\relax\fi}%
\newdimen\CD@XH\CD@XH=1sp \newdimen\CD@zC\CD@zC\z@\def\CD@cJ{\ifdim\CD@zC=1em%
\else\CD@nJ\fi}\def\CD@nJ{\CD@zC1em\def\CD@NC{\fontdimen8\textfont3 }\CD@@J
\CD@NJ\setbox0=\vbox{\CD@t\noindent\CD@k\null\penalty-9993\null\CD@ND\null
\endgraf\setbox0=\lastbox\unskip\unpenalty\setbox1=\lastbox\global\setbox
\CD@IG=\hbox{\unhbox0\unskip\unskip\unpenalty\setbox0=\lastbox}\global\setbox
\CD@KG=\hbox{\unhbox1\unskip\unpenalty\setbox1=\lastbox}}}\newdimen\CD@@I
\CD@@I=1true in \divide\CD@@I300 \def\CD@zH#1{\multiply#1\tw@\advance#1\ifnum
#1<\z@-\else+\fi\CD@@I\divide#1\tw@\divide#1\CD@@I\multiply#1\CD@@I}\def
\MapBreadth{\afterassignment\CD@gI\CD@LF}\newdimen\CD@LF\newdimen\CD@oI\def
\CD@gI{\CD@oI\CD@LF\CD@V\CD@@I{4\CD@XH}\CD@X\CD@@I\p@\CD@zH\CD@oI\ifdim\CD@LF
>\z@\CD@V\CD@oI\CD@@I\fi\CD@cJ}\def\CD@RJ#1{\CD@zD\count@\CD@@I#1\ifnum
\count@>\z@\divide\CD@@I\count@\fi\CD@gI\CD@NJ}\def\CD@NJ{\dimen@\CD@QC
\count@\dimen@\divide\count@5\divide\count@\CD@@I\edef\CD@OC{\the\count@}}%
\def\CD@AJ{\CD@QJ\z@}\def\CD@QJ#1{\CD@tI\axisheight\advance\CD@tI#1\relax
\advance\CD@tI-.5\CD@oI\CD@zH\CD@tI\CD@sI-\CD@tI\advance\CD@tI\CD@LF}%
\newdimen\CD@DC\CD@DC\z@\newdimen\CD@eJ\CD@eJ\z@\def\CD@CJ#1{\CD@sI#1\relax
\CD@tI\CD@sI\advance\CD@tI\CD@LF\relax}\def\horizhtdp{height\CD@tI depth%
\CD@sI}\def\axisheight{\fontdimen22\the\textfont\tw@}\def\script@axisheight{%
\fontdimen22\the\scriptfont\tw@}\def\ss@axisheight{\fontdimen22\the
\scriptscriptfont\tw@}\def\CD@NC{0.4pt}\def\CD@VK{\fontdimen3\textfont\z@}%
\def\CD@UK{\fontdimen3\textfont\z@}\newdimen\PileSpacing\newdimen\CD@nA\CD@nA
\z@\def\CD@RG{\ifincommdiag1.3em\else2em\fi}\newdimen\CD@YB\def\CellSize{%
\afterassignment\CD@kB\DiagramCellHeight}\newdimen\DiagramCellHeight
\DiagramCellHeight-\maxdimen\newdimen\DiagramCellWidth\DiagramCellWidth-%
\maxdimen\def\CD@kB{\DiagramCellWidth\DiagramCellHeight}\def\CD@QC{3em}%
\newdimen\MapShortFall\def\MapsAbut{\MapShortFall\z@\objectheight\z@
\objectwidth\z@}\newdimen\CD@iA\CD@iA\z@\CD@tG\CD@vE\CD@aB\CD@ZB\expandafter
\ifx\expandafter\iftrue\csname ifUglyObsoleteDiagrams\endcsname\CD@ZB\else
\CD@aB\fi\CD@nF{ifUglyObsoleteDiagrams}{relax}\newif\ifUglyObsoleteDiagrams
\def\CD@nK{\CD@aB\UglyObsoleteDiagramsfalse}\def\CD@oK{\CD@ZB
\UglyObsoleteDiagramstrue}\CD@vE\CD@nK\else\CD@oK\fi\CD@tG\CD@hK\CD@dK\CD@cK
\CD@cK\def\CD@sK{\ifx\pdfoutput\CD@qK\else\ifx\pdfoutput\relax\else\ifnum
\pdfoutput>\z@\CD@pK\fi\fi\fi} \def\CD@pK{\global\CD@dK\global\CD@aB\global
\UglyObsoleteDiagramsfalse\global\let\CD@n\empty\global\let\CD@oK\relax
\global\let\CD@pK\relax\global\let\CD@sK\relax}\def\CD@tK#1{}\ifx\pdfliteral\CD@qK\else\ifx\pdfliteral\relax\else\let\CD@tK
\pdfliteral\fi\fi\ifx\XeTeXrevision\CD@qK\else\ifx\XeTeXrevision\relax\else
\ifdim\XeTeXrevision pt<.996pt \expandafter\message{! XeTeX version
\XeTeXrevision\space does not support PDF literals, so diagonals will not work%
!}\else\expandafter\message{RUNNING UNDER XETEX \XeTeXrevision}\CD@pK\fi\fi
\fi\CD@sK\def\newarrowhead{\CD@mG h\CD@BG\CD@GG>}\def\newarrowtail{\CD@mG t%
\CD@BG\CD@GG>}\def\newarrowmiddle{\CD@mG m\CD@BG\hbox@maths\empty}\def
\newarrowfiller{\CD@mG f\CD@bE\CD@MK-}\def\CD@mG#1#2#3#4#5#6#7#8#9{\CD@RC{r#1%
:#5}{#2{#6}}\CD@RC{l#1:#5}{#2{#7}}\CD@RC{d#1:#5}{#3{#8}}\CD@RC{u#1:#5}{#3{#9}%
}\CD@vC{-#1:#5}{\expandafter\noexpand\csname-#1:#4\endcsname\noexpand\CD@MC}%
\CD@vC{+#1:#5}{\expandafter\noexpand\csname+#1:#4\endcsname\noexpand\CD@MC}}%
\CD@ZA\CD@MC{\CD@eF\space diagonals are used unless PostScript is set}\def
\defaultarrowhead#1{\edef\CD@sJ{#1}\CD@@J}\def\CD@@J{\CD@IJ\CD@sJ<>ht\CD@IJ
\CD@sJ<>th}\def\CD@IJ#1#2#3#4#5{\CD@HJ{r#4}{#3}{l#5}{#2}{r#4:#1}\CD@HJ{r#5}{#%
2}{l#4}{#3}{l#4:#1}\CD@HJ{d#4}{#3}{u#5}{#2}{d#4:#1}\CD@HJ{d#5}{#2}{u#4}{#3}{u%
#4:#1}}\def\CD@HJ#1#2#3#4#5{\begingroup\aftergroup\CD@GJ\CD@L{#1+:#2}\CD@L{#1%
:#2}\CD@L{#3:#4}\CD@L{#5}\endgroup}\def\CD@GJ#1#2#3#4{\csname newbox%
\endcsname#1\def#2{\copy#1}\def#3{\copy#1}\setbox#1=\box\voidb@x}\def\CD@sJ{}%
\CD@@J\def\CD@GJ#1#2#3#4{\setbox#1=#4}\ifx\tenln\nullfont\def\CD@sJ{vee}\else
\let\CD@sJ\CD@eF\fi\def\CD@xF#1#2#3{\begingroup\aftergroup\CD@wF\CD@L{#1#2:#3%
#3}\CD@L{#1#2:#3}\aftergroup\CD@yF\CD@L{#1#2:#3-#3}\CD@L{#1#2:#3}\endgroup}%
\def\CD@wF#1#2{\def#1{\hbox{\rlap{#2}\kern.4\CD@zC#2}}}\def\CD@yF#1#2{\def#1{%
\hbox{\rlap{#2}\kern.4\CD@zC#2\kern-.4\CD@zC}}}\CD@xF lh>\CD@xF rt>\CD@xF rh<%
\CD@xF rt<\def\CD@yF#1#2{\def#1{\hbox{\kern-.4\CD@zC\rlap{#2}\kern.4\CD@zC#2}%
}}\CD@xF rh>\CD@xF lh<\CD@xF lt>\CD@xF lt<\def\CD@wF#1#2{\def#1{\vbox{\vbox to%
\z@{#2\vss}\nointerlineskip\kern.4\CD@zC#2}}}\def\CD@yF#1#2{\def#1{\vbox{%
\vbox to\z@{#2\vss}\nointerlineskip\kern.4\CD@zC#2\kern-.4\CD@zC}}}\CD@xF uh>%
\CD@xF dt>\CD@xF dh<\CD@xF dt<\def\CD@yF#1#2{\def#1{\vbox{\kern-.4\CD@zC\vbox
to\z@{#2\vss}\nointerlineskip\kern.4\CD@zC#2}}}\CD@xF dh>\CD@xF ut>\CD@xF uh<%
\CD@xF ut<\def\CD@BG#1{\hbox{\mathsurround\z@\offinterlineskip\CD@k\mkern-1.5%
mu{#1}\mkern-1.5mu\CD@ND}}\def\hbox@maths#1{\hbox{\CD@k#1\CD@ND}}\def\CD@GG#1%
{\hbox to\CD@LF{\setbox0=\hbox{\offinterlineskip\mathsurround\z@\CD@k{#1}%
\CD@ND}\dimen0.5\wd0\advance\dimen0-.5\CD@oI\CD@zH{\dimen0}\kern-\dimen0%
\unhbox0\hss}}\def\CD@sB#1{\hbox to2\CD@LF{\hss\offinterlineskip\mathsurround
\z@\CD@k{#1}\CD@ND\hss}}\def\CD@vF#1{\hbox{\mathsurround\z@\CD@k{#1}\CD@ND}}%
\def\CD@bE#1{\hbox{\kern-.15\CD@zC\CD@k{#1}\CD@ND\kern-.15\CD@zC}}\def\CD@MK#%
1{\vbox{\offinterlineskip\kern-.2ex\CD@GG{#1}\kern-.2ex}}\def\@fillh{%
\xleaders\vrule\horizhtdp}\def\@fillv{\xleaders\hrule width\CD@LF}\CD@nF{rf:-%
}{@fillh}\CD@nF{lf:-}{@fillh}\CD@nF{df:-}{@fillv}\CD@nF{uf:-}{@fillv}\CD@nF{%
rh:}{null}\CD@nF{rm:}{null}\CD@nF{rt:}{null}\CD@nF{lh:}{null}\CD@nF{lm:}{null%
}\CD@nF{lt:}{null}\CD@nF{dh:}{null}\CD@nF{dm:}{null}\CD@nF{dt:}{null}\CD@nF{%
uh:}{null}\CD@nF{um:}{null}\CD@nF{ut:}{null}\CD@nF{+h:}{null}\CD@nF{+m:}{null%
}\CD@nF{+t:}{null}\CD@nF{-h:}{null}\CD@nF{-m:}{null}\CD@nF{-t:}{null}\def
\CD@@D{\hbox{\vrule height 1pt depth-1pt width 1pt}}\CD@RC{rf:}{\CD@@D}\CD@nF
{lf:}{rf:}\CD@nF{+f:}{rf:}\CD@RC{df:}{\CD@@D}\CD@nF{uf:}{df:}\CD@nF{-f:}{df:}%
\def\CD@BD{\CD@U\null\CD@@D\null\CD@@D\null}\edef\CD@lG{\string\newarrow}\def
\newarrow#1#2#3#4#5#6{\begingroup\edef\@name{#1}\edef\CD@oJ{#2}\edef\CD@iD{#3%
}\edef\CD@QG{#4}\edef\CD@jD{#5}\edef\CD@LE{#6}\let\CD@HE\CD@sG\let\CD@FK
\CD@BH\let\@x\CD@AH\ifx\CD@oJ\CD@iD\let\CD@oJ\empty\fi\ifx\CD@LE\CD@jD\let
\CD@LE\empty\fi\def\CD@LI{r}\def\CD@SF{l}\def\CD@IC{d}\def\CD@yJ{u}\def\CD@gH
{+}\def\@m{-}\ifx\CD@iD\CD@jD\ifx\CD@QG\CD@iD\let\CD@QG\empty\fi\ifx\CD@LE
\empty\ifx\CD@iD\CD@aE\let\@x\CD@yG\else\let\@x\CD@zG\fi\fi\else\edef\CD@a{%
\CD@iD\CD@oJ}\ifx\CD@a\empty\ifx\CD@QG\CD@jD\let\CD@QG\empty\fi\fi\fi\ifmmode
\aftergroup\CD@kG\else\CD@@A\CD@oB rh{head\space\space}\CD@LE\CD@oB rf{filler%
}\CD@iD\CD@oB rm{middle}\CD@QG\ifx\CD@jD\CD@iD\else\CD@oB rf{filler}\CD@jD\fi
\CD@oB rt{tail\space\space}\CD@oJ\CD@gE\CD@HE\CD@FK\@x\CD@nG l-2+2{lu}{nw}%
\NorthWest\CD@nG r+2+2{ru}{ne}\NorthEast\CD@nG l-2-2{ld}{sw}\SouthWest\CD@nG r%
+2-2{rd}{se}\SouthEast\else\aftergroup\CD@b\CD@L{r\@name}\fi\fi\endgroup}\def
\CD@sG{\CD@vG\CD@LI\CD@SF rl\Horizontal@Map}\def\CD@BH{\CD@vG\CD@IC\CD@yJ du%
\Vertical@Map}\def\CD@AH{\CD@vG\CD@gH\@m+-\Vector@Map}\def\CD@yG{\CD@vG\CD@gH
\@m+-\Slant@Map}\def\CD@zG{\CD@vG\CD@gH\@m+-\Slope@Map}\catcode`\/=\active
\def\CD@vG#1#2#3#4#5{\CD@jG#1#3#5t:\CD@oJ/f:\CD@iD/m:\CD@QG/f:\CD@jD/h:\CD@LE
//\CD@jG#2#4#5h:\CD@LE/f:\CD@jD/m:\CD@QG/f:\CD@iD/t:\CD@oJ//}\def\CD@jG#1#2#3%
#4//{\edef\CD@fG{#2}\aftergroup\sdef\CD@L{#1\@name}\aftergroup{\aftergroup#3%
\CD@M#4//\aftergroup}}\def\CD@M#1/{\edef\CD@EH{#1}\ifx\CD@EH\empty\else\CD@L{%
\CD@fG#1}\expandafter\CD@M\fi}\catcode`\/=12 \def\CD@nG#1#2#3#4#5#6#7#8{%
\aftergroup\sdef\CD@L{#6\@name}\aftergroup{\CD@L{#2\@name}\if#2#4\aftergroup
\CD@CI\else\aftergroup\CD@BI\fi\CD@L{#1\@name}%
\aftergroup(\aftergroup#3\aftergroup,\aftergroup#5\aftergroup)\aftergroup}}%
\def\CD@oB#1#2#3#4{\expandafter\ifx\csname#1#2:#4\endcsname\relax\CD@y\CD@gB{%
arrow#3 "#4" undefined}\fi}\CD@rG\CD@VE{All five components must be defined
before an arrow.}\CD@rG\CD@SE{\CD@lG, unlike \string\HorizontalMap, is a
declaration.}\def\CD@b#1{\CD@YA{Arrows \string#1 etc could not be defined}%
\CD@VE}\def\CD@kG{\CD@YA{misplaced \CD@lG}\CD@SE}\def\newdiagramgrid#1#2#3{%
\CD@RC{cdgh@#1}{#2,],}
\CD@RC{cdgv@#1}{#3,],}}
\CD@tG\ifincommdiag\incommdiagtrue\incommdiagfalse\CD@tG\CD@@F\CD@IF\CD@HF
\newcount\CD@VA\CD@VA=0 \def\CD@yH{\CD@VA6 }\def\CD@OB{\CD@VA1 \global\CD@yA1
\CD@DE\CD@YF\empty}\def\CD@YF{}\def\CD@nB#1{\relax\CD@MD\edef\CD@vJ{#1}%
\begingroup\CD@rE\else\ifcase\CD@VA\ifmmode\else\CD@YG\CD@E0\fi\or\CD@cE5\or
\CD@YG\CD@F5\or\CD@YG\CD@B5\or\CD@YG\CD@B5\or\CD@YG\CD@C5\or\CD@cE7\or\CD@YG
\CD@D7\fi\fi\endgroup\xdef\CD@YF{#1}}\def\CD@pB#1#2#3#4#5{\relax\CD@MD\xdef
\CD@vJ{#4}\begingroup\ifnum\CD@VA<#1 \expandafter\CD@cE\ifcase\CD@VA0\or#2\or
#3\else#2\fi\else\ifnum\CD@VA<6 \CD@tJ\CD@YG\CD@B#2\else\CD@YG\CD@G#2\fi\fi
\endgroup\CD@DE\CD@YF\CD@vJ\ifincommdiag\let\CD@ZD#5\else\let\CD@ZD\CD@LK\fi}%
\def\CD@yI{\global\CD@yA=\ifnum\CD@VA<5 1\else2\fi\relax}\def\CD@OI{\CD@VA
\CD@yA}\def\CD@cE#1{\aftergroup\CD@VA\aftergroup#1\aftergroup\relax}\def
\CD@HH{\def\CD@nB##1{\relax}\let\CD@pB\CD@FH\let\CD@yH\relax\let\CD@OB\relax
\let\CD@yI\relax\let\CD@OI\relax}\def\CD@FH#1#2#3#4#5{\ifincommdiag\let\CD@ZD
#5\else\xdef\CD@vJ{#4}\let\CD@ZD\CD@LK\fi}\def\CD@YG#1{\aftergroup#1%
\aftergroup\relax\CD@cE}\def\CD@B{\CD@YE\CD@S\CD@ME\CD@Q}\def\CD@G{\CD@YE{%
\CD@yC\CD@S}\CD@XE\CD@QD\CD@Q}\def\CD@F{\CD@YE{*\CD@S}\CD@RE\clubsuit\CD@Q}%
\def\CD@C{\CD@YE{\CD@S*\CD@S}\CD@RE\CD@Q\clubsuit\CD@Q}\def\CD@D{\CD@YE\CD@EC
\CD@TE\\}\def\CD@E{\CD@YE\CD@nC\CD@QE\CD@k}\def\CD@LK{\CD@YA{\CD@vJ\space
ignored \CD@dH}\CD@WE}\def\CD@FE{}\def\CD@d{\CD@YA{maps must never be enclosed
in braces}\CD@OE}\def\CD@dH{outside diagram}\def\CD@FC{\string\HonV, \string
\VonH\space and \string\HmeetV}\CD@rG\CD@ME{The way that horizontal and
vertical arrows are terminated implicitly means\CD@uG that they cannot be
mixed with each other or with \CD@FC.}\CD@rG\CD@XE{\string\pile\space is for
parallel horizontal arrows; verticals can just be put together in\CD@uG a cell%
. \CD@FC\space are not meaningful in a \string\pile.}\CD@rG\CD@RE{The
horizontal maps must point to an object, not each other (I've put in\CD@uG one
which you're unlikely to want). Use \string\pile\space if you want them
parallel.}\CD@rG\CD@TE{Parallel horizontal arrows must be in separate layers
of a \string\pile.}\CD@rG\CD@QE{Horizontal arrows may be used \CD@dH s, but
must still be in maths.}\CD@rG\CD@WE{Vertical arrows, \CD@FC\space\CD@dH s don%
't know where\CD@uG where to terminate.}\CD@rG\CD@OE{This prevents them from
stretching correctly.}\def\CD@YE#1{\CD@YA{"#1" inserted \ifx\CD@YF\empty
before \CD@vJ\else between \CD@YF\ifx\CD@YF\CD@vJ s\else\space and \CD@vJ\fi
\fi}}\count@=\year\multiply\count@12 \advance\count@\month\ifnum\count@>24391
\expandafter\endinput\fi\def
\Horizontal@Map{\CD@nB{horizontal map}\CD@LC\CD@TJ\CD@qD}\def\CD@TJ{\CD@GB-%
9999 \let\CD@ZD\CD@XD\ifincommdiag\else\CD@cJ\ifinpile\else\skip2\z@ plus 1.5%
\CD@VK minus .5\CD@UK\skip4\skip2 \fi\fi\let\CD@kD\@fillh\CD@nF{fill@dot}{rf:%
.}}\def\Vector@Map{\CD@HK4}\def\Slant@Map{\CD@HK{\CD@EF255\else6\fi}}\def
\Slope@Map{\CD@HK\CD@OC}\def\CD@HK#1#2#3#4#5#6{\CD@LC\def\CD@WK{2}\def\CD@aK{%
2}\def\CD@ZK{1}\def\CD@bK{1}\let\Horizontal@Map\CD@nI\def\CD@OG{#1}\def\CD@NI
{\CD@U#2#3#4#5#6}}\def\CD@nI{\CD@TJ\CD@JB\let\CD@ZD\CD@TD\CD@qD}\CD@tG\CD@pE
\CD@rA\CD@qA\CD@rA\def\cds@missives{\CD@rA}\def\CD@TD{\CD@vE\let\CD@OG\CD@OC
\CD@x\CD@zE\CD@WF\fi\setbox0\hbox{\incommdiagfalse\CD@HI}\CD@pE\CD@aD\else
\global\CD@YC\CD@bD\fi\ifvoid6 \ifvoid7 \CD@eE\fi\fi\CD@zE\else\CD@BD\global
\CD@YC\let\CD@CG\CD@IH\CD@YD\fi\else\CD@NI\CD@MI\global\CD@YC\CD@YD\fi}\def
\CD@LC{\begingroup\dimen1=\MapShortFall\dimen2=\CD@RG\dimen5=\MapShortFall
\setbox3=\box\voidb@x\setbox6=\box\voidb@x\setbox7=\box\voidb@x\CD@pD
\mathsurround\z@\skip2\z@ plus1fill minus 1000pt\skip4\skip2 \CD@TB}\CD@tG
\CD@tE\CD@UB\CD@TB\def\CD@U#1#2#3#4#5{\let\CD@oJ#1\let\CD@iD#2\let\CD@QG#3%
\let\CD@jD#4\let\CD@LE#5\CD@TB\ifx\CD@iD\CD@jD\CD@UB\fi}\def\CD@qD#1#2#3#4#5{%
\CD@U#1#2#3#4#5\CD@tD}\def\Vertical@Map{\CD@pB433{vertical map}\CD@cD\CD@LC
\CD@GB-9995 \let\CD@kD\@fillv\CD@nF{fill@dot}{df:.}\CD@qD}\def\break@args{%
\def\CD@tD{\CD@ZD}\CD@ZD\endgroup\aftergroup\CD@FE}\def\CD@MJ{\setbox1=\CD@oJ
\setbox5=\CD@LE\ifvoid3 \ifx\CD@QG\null\else\setbox3=\CD@QG\fi\fi\CD@@G2%
\CD@iD\CD@@G4\CD@jD}\def\CD@pF#1{\ifvoid1\else\CD@oF1#1\fi\ifvoid2\else\CD@oF
2#1\fi\ifvoid3\else\CD@oF3#1\fi\ifvoid4\else\CD@oF4#1\fi\ifvoid5\else\CD@oF5#%
1\fi} \def\CD@oF#1#2{\setbox#1\vbox{\offinterlineskip\box#1\dimen@\prevdepth
\advance\dimen@-#2\relax\setbox0\null\dp0\dimen@\ht0-\dimen@\box0}}\def\CD@@G
#1#2{\ifx#2\CD@kD\setbox#1=\box\voidb@x\else\setbox#1=#2\def#2{\xleaders\box#%
1}\fi}\CD@ZA\CD@BK{\string\HorizontalMap, \string\VerticalMap\space and
\string\DiagonalMap\CD@uG are obsolete - use \CD@lG\space to pre-define maps}%
\def\HorizontalMap#1#2#3#4#5{\CD@BK\CD@nB{old horizontal map}\CD@LC\CD@TJ\def
\CD@oJ{\CD@UH{#1}}\CD@SH\CD@iD{#2}\def\CD@QG{\CD@UH{#3}}\CD@SH\CD@jD{#4}\def
\CD@LE{\CD@UH{#5}}\CD@tD}\def\VerticalMap#1#2#3#4#5{\CD@BK\CD@pB433{vertical
map}\CD@cD\CD@LC\CD@GB-9995 \let\CD@kD\@fillv\def\CD@oJ{\CD@GG{#1}}\CD@VH
\CD@iD{#2}\def\CD@QG{\CD@GG{#3}}\CD@VH\CD@jD{#4}\def\CD@LE{\CD@GG{#5}}\CD@tD}%
\def\DiagonalMap#1#2#3#4#5{\CD@BK\CD@LC\def\CD@OG{4}\let\CD@kD\CD@qK\let
\CD@ZD\CD@YD\def\CD@WK{2}\def\CD@aK{2}\def\CD@ZK{1}\def\CD@bK{1}\def\CD@QG{%
\CD@vF{#3}}\ifPositiveGradient\let\mv\raise\def\CD@oJ{\CD@vF{#5}}\def\CD@iD{%
\CD@vF{#4}}\def\CD@jD{\CD@vF{#2}}\def\CD@LE{\CD@vF{#1}}\else\let\mv\lower\def
\CD@oJ{\CD@vF{#1}}\def\CD@iD{\CD@vF{#2}}\def\CD@jD{\CD@vF{#4}}\def\CD@LE{%
\CD@vF{#5}}\fi\CD@tD}\def\CD@aE{-}\def\CD@AD{\empty}\def\CD@SH{\CD@EG\CD@bE
\CD@aE\@fillh}\def\CD@VH{\CD@EG\CD@MK\CD@KK\@fillv}\def\CD@EG#1#2#3#4#5{\def
\CD@CH{#5}\ifx\CD@CH#2\let#4#3\else\let#4\null\ifx\CD@CH\empty\else\ifx\CD@CH
\CD@AD\else\let#4\CD@CH\fi\fi\fi}\def\CD@UH#1{\hbox{\mathsurround\z@
\offinterlineskip\def\CD@CH{#1}\ifx\CD@CH\empty\else\ifx\CD@CH\CD@AD\else
\CD@k\mkern-1.5mu{\CD@CH}\mkern-1.5mu\CD@ND\fi\fi}}\def\CD@yD#1#2{\setbox#1=%
\hbox\bgroup\setbox0=\hbox{\CD@k\labelstyle()\CD@ND}
\setbox1=\null\ht1\ht0\dp1\dp0\box1 \kern.1\CD@zC\CD@k\bgroup\labelstyle
\aftergroup\CD@LD\CD@xD}\def\CD@LD{\CD@ND\kern.1\CD@zC\egroup\CD@tD}\def
\CD@xD{\futurelet\CD@EH\CD@mJ}\def\CD@mJ{
\catcase\bgroup:\CD@v;\catcase\egroup:\missing@label;\catcase\space:\CD@TF;%
\tokcase[:\CD@XF;
\default:\CD@zJ;\endswitch}\def\CD@v{\let\CD@MD\CD@c\let\CD@CH}\def\CD@zJ#1{%
\let\CD@UF\egroup{\let\actually@braces@missing@around@macro@in@label\CD@ZH
\let\CD@MD\CD@xC\let\CD@UF\CD@VF#1%
\actually@braces@missing@around@macro@in@label}\CD@UF}\def
\actually@braces@missing@around@macro@in@label{\let\CD@CH=}\def\missing@label
{\egroup\CD@YA{missing label}\CD@PE}\def\CD@xC{\egroup\missing@label}\outer
\def\CD@ZH{}\def\CD@UF{}\def\CD@VF{\CD@wC\CD@UF}\def\CD@MD{}\def\CD@XF{\let
\CD@N\CD@xD\get@square@arg\CD@AE}\CD@rG\CD@PE{The text which has just been
read is not allowed within map labels.}\def\CD@c{\egroup\CD@YA{missing \CD@yC
\space inserted after label}\CD@PE}\def\upper@label{\CD@oD\CD@yD6}\def
\lower@label{\def\positional@{\CD@@A\break@args}\CD@yD7}\def\middle@label{%
\CD@yD3}\CD@tG\CD@yE\CD@pD\CD@oD\def\CD@iF{\ifPositiveGradient\CD@tJ
\expandafter\upper@label\else\expandafter\lower@label\fi}\def\CD@iI{%
\ifPositiveGradient\CD@tJ\expandafter\lower@label\else\expandafter
\upper@label\fi}\def\positional@{\CD@gB{labels as positional arguments are
obsolete}\CD@yE\CD@tJ\expandafter\upper@label\else\expandafter\lower@label\fi
-}\def\CD@tD{\futurelet\CD@EH\switch@arg}\def\eat@space{\afterassignment
\CD@tD\let\CD@EH= }\def\CD@TF{\afterassignment\CD@xD\let\CD@EH= }\def\CD@BC{%
\get@round@pair\CD@uD}\def\CD@uD#1#2{\def\CD@WK{#1}\def\CD@aK{#2}\CD@tD}\def
\optional@{\let\CD@N\CD@tD\get@square@arg\CD@AE}\def\CD@JJ.{\CD@sC\CD@tD}\def
\CD@sC{\let\CD@iD\fill@dot\let\CD@jD\fill@dot\def\CD@MI{\let\CD@iD\dfdot\let
\CD@jD\dfdot}}\def\CD@MI{}\def\CD@@E#1,{\CD@nH#1,\begingroup\ifx\@name\CD@RD
\CD@FF\aftergroup\CD@e\fi\aftergroup\CD@jC\else\expandafter\def\expandafter
\CD@RF\expandafter{\csname\@name\endcsname}\expandafter\CD@vD\CD@RF\CD@KD\ifx
\CD@RF\empty\aftergroup\CD@pC\expandafter\aftergroup\csname\CD@FB\@name
\endcsname\expandafter\aftergroup\csname\CD@FB @\@name\endcsname\else\gdef
\CD@GE{#1}\CD@gB{\string\relax\space inserted before `[\CD@GE'}\message{(I was
trying to read this as a \CD@tA\ option.)}\aftergroup\CD@H\fi\fi\endgroup}%
\def\CD@vD#1#2\CD@KD{\def\CD@RF{#2}}\def\CD@jC{\let\CD@CH\CD@N\let\CD@N\relax
\CD@CH}\def\CD@H#1],{
\CD@jC\relax\def\CD@RF{#1}\ifx\CD@RF\empty\def\CD@RF{[\CD@GE]}%
\else\def\CD@RF{[\CD@GE,#1]}
\fi\CD@RF}\def\CD@pC#1#2{\ifx#2\CD@qK\ifx#1\CD@qK\CD@gB{option `\@name'
undefined}\else#1\fi\else\CD@FF\expandafter#2\CD@GK\CD@PK\else\CD@QK\fi\fi
\CD@DH}\CD@tG\CD@FF\CD@QK\CD@PK\def\CD@nH#1,{\CD@FF\ifx\CD@GK\CD@qK\CD@e\else
\expandafter\CD@oH\CD@GK,#1,(,),(,)[]%
\fi\fi\CD@FF\else\CD@mH#1==,\fi}\def\CD@e{\CD@gB{option `\@name' needs (x,y)
value}\CD@PK\let\@name\empty}\def\CD@mH#1=#2=#3,{\def\@name{#1}\def\CD@GK{#2}%
\def\CD@RF{#3}\ifx\CD@RF\empty\let\CD@GK\CD@qK\fi}%
\def\CD@oH#1(#2,#3)#4,(#5,#6)#7[]{\def\CD@GK{{#2}{#3}}\def\CD@RF{#1#4#5#6}%
\ifx\CD@RF\empty\def\CD@RF{#7}\ifx\CD@RF\empty\CD@e\fi\else\CD@e\fi}\def
\CD@FB{cds@}\let\CD@N\relax\def\CD@zD#1{\ifx\CD@GK\CD@qK\CD@gB{option `\@name
' needs a value}\else#1\CD@GK\relax\fi}\def\CD@BE#1#2{\ifx\CD@GK\CD@qK#1#2%
\relax\else#1\CD@GK\relax\fi}\def\cds@@showpair#1#2{\message{x=#1,y=#2}}\def
\cds@@diagonalbase#1#2{\edef\CD@ZK{#1}\edef\CD@bK{#2}}\def\CD@DI#1{\def\CD@CH
{#1}\CD@nF{@x}{cdps@#1}\ifx\CD@CH\empty\CD@f\CD@CH{cannot be used}\else\ifx
\CD@CH\relax\CD@f\CD@CH{unknown}\else\let\CD@IK\@x\fi\fi}\def\CD@f#1#2{\CD@gB
{PostScript translator `#1' #2}}\def\CD@PH{}\def\CD@PJ{\CD@fA\edef\CD@PH{%
\noexpand\CD@KB{\@name\space ignored within maths}}}\def\diagramstyle{\CD@cJ
\let\CD@N\relax\CD@CF\CD@AE\CD@AE}\CD@tG\CD@sE
\CD@SB\CD@RB\CD@tG\CD@qE\CD@EB\CD@DB\CD@tG\CD@oE\CD@pA\CD@oA\CD@tG\CD@iE
\CD@HA\CD@GA\CD@HA\CD@tG\CD@jE\CD@JA\CD@IA\CD@tG\CD@kE\CD@LA\CD@KA\CD@tG
\CD@EF\CD@DK\CD@CK\CD@tG\CD@rE\CD@JB\CD@IB\CD@tG\CD@mE\CD@gA\CD@fA\CD@tG
\CD@nE\CD@kA\CD@jA\CD@tG\CD@AF\CD@iG\CD@hG\CD@RC{cds@ }{}\CD@RC{cds@}{}\CD@RC
{cds@1em}{\CellSize1\CD@zC}\CD@RC{cds@1.5em}{\CellSize1.5\CD@zC}\CD@RC{cds@2%
em}{\CellSize2\CD@zC}\CD@RC{cds@2.5em}{\CellSize2.5\CD@zC}\CD@RC{cds@3em}{%
\CellSize3\CD@zC}\CD@RC{cds@3.5em}{\CellSize3.5\CD@zC}\CD@RC{cds@4em}{%
\CellSize4\CD@zC}\CD@RC{cds@4.5em}{\CellSize4.5\CD@zC}\CD@RC{cds@5em}{%
\CellSize5\CD@zC}\CD@RC{cds@6em}{\CellSize6\CD@zC}\CD@RC{cds@7em}{\CellSize7%
\CD@zC}\CD@RC{cds@8em}{\CellSize8\CD@zC}\def\cds@abut{\MapsAbut\dimen1\z@
\dimen5\z@}\def\cds@alignlabels{\CD@IA\CD@KA}\def\cds@amstex{\ifincommdiag
\CD@O\else\def\CD{\diagram[amstex]}
\fi\CD@T\catcode`\@\active}\def\cds@b{\let\CD@dB\CD@bB}\def\cds@balance{\let
\CD@hA\CD@AA}\let\cds@bottom\cds@b\def\cds@center{\cds@vcentre\cds@nobalance}%
\let\cds@centre\cds@center\def\cds@centerdisplay{\CD@HA\CD@PJ\cds@balance}%
\let\cds@centredisplay\cds@centerdisplay\def\cds@crab{\CD@BE\CD@DC{.5%
\PileSpacing}}\CD@RC{cds@crab-}{\CD@DC-.5\PileSpacing}\CD@RC{cds@crab+}{%
\CD@DC.5\PileSpacing}\CD@RC{cds@crab++}{\CD@DC1.5\PileSpacing}\CD@RC{cds@crab%
--}{\CD@DC-1.5\PileSpacing}\def\cds@defaultsize{\CD@BE{\let\CD@QC}{3em}\CD@NJ
}\def\cds@displayoneliner{\CD@DB}\let\cds@dotted\CD@sC\def\cds@dpi{\CD@RJ{1%
truein}}\def\cds@dpm{\CD@RJ{100truecm}}\let\CD@XA\CD@qK\def\cds@eqno{\let
\CD@XA\CD@GK\let\CD@EJ\empty}\def\cds@fixed{\CD@qA}\CD@tG\CD@fE\CD@J\CD@I\def
\cds@flushleft{\CD@I\CD@GA\CD@PJ\cds@nobalance\CD@BE\CD@nA\CD@nA}\def\cds@gap
{\CD@AJ\setbox3=\null\ht3=\CD@tI\dp3=\CD@sI\CD@BE{\wd3=}\MapShortFall} \def
\cds@grid{\ifx\CD@GK\CD@qK\let\h@grid\relax\let\v@grid\relax\else\CD@nF{%
h@grid}{cdgh@\CD@GK}\CD@nF{v@grid}{cdgv@\CD@GK}\ifx\h@grid\relax\CD@gB{%
unknown grid `\CD@GK'}\else\CD@WB\fi\fi}\let\h@grid\relax\let\v@grid\relax
\def\cds@gridx{\ifx\CD@GK\CD@qK\else\cds@grid\fi\let\CD@CH\h@grid\let\h@grid
\v@grid\let\v@grid\CD@CH}\def\cds@h{\CD@zD\DiagramCellHeight}\def\cds@hcenter
{\let\CD@hA\CD@aA}\let\cds@hcentre\cds@hcenter\def\cds@heads{\CD@BE{\let
\CD@sJ}\CD@sJ\CD@@J\CD@vE\else\ifx\CD@sJ\CD@eF\else\CD@MC\fi\fi}\let
\cds@height\cds@h\let\cds@hmiddle\cds@balance\def\cds@htriangleheight{\CD@BE
\DiagramCellHeight\DiagramCellHeight\DiagramCellWidth1.73205%
\DiagramCellHeight}\def\cds@htrianglewidth{\CD@BE\DiagramCellWidth
\DiagramCellWidth\DiagramCellHeight.57735\DiagramCellWidth}\CD@tG\CD@zE\CD@eE
\CD@dE\CD@eE\def\cds@hug{\CD@eE} \def\cds@inline{\CD@gA\let\CD@PH\empty}\def
\cds@inlineoneliner{\CD@EB}\CD@RC{cds@l>}{\CD@zD{\let\CD@RG}\dimen2=\CD@RG}%
\def\cds@labelstyle{\CD@zD{\let\labelstyle}}\def\cds@landscape{\CD@kA}\def
\cds@large{\CellSize5\CD@zC}\let\CD@EJ\empty\def\CD@FJ{\refstepcounter{%
equation}\def\CD@XA{\hbox{\@eqnnum}}}\def\cds@LaTeXeqno{\let\CD@EJ\CD@FJ}\def
\cds@lefteqno{\CD@pA}\def\cds@leftflush{\cds@flushleft\CD@J}\def
\cds@leftshortfall{\CD@zD{\dimen1 }}\def\cds@lowershortfall{%
\ifPositiveGradient\cds@leftshortfall\else\cds@rightshortfall\fi}\def
\cds@loose{\CD@VB}\def\cds@midhshaft{\CD@JA}\def\cds@midshaft{\CD@JA}\def
\cds@midvshaft{\CD@LA}\def\cds@moreoptions{\CD@@A}\let\cds@nobalance
\cds@hcenter\def\cds@nohcheck{\CD@HH}\def\cds@nohug{\CD@dE} \def
\cds@nooptions{\def\CD@aC{\CD@WD}}\let\cds@noorigin\cds@nobalance\def
\cds@nopixel{\CD@@I4\CD@XH\CD@cJ}\def\cds@UO{\CD@oK\global\let\CD@n\empty}%
\def\cds@UglyObsolete{\cds@UO\let\cds@PS\empty}\def\CD@rK#1{\CD@gB{option `#1%
' renamed as `UglyObsolete'}}\def\cds@noPostScript{\CD@rK{noPostScript}}\def
\cds@noPS{\CD@rK{noPostScript}}\def\cds@notextflow{\CD@RB}\def\cds@noTPIC{%
\CD@CK}\def\cds@objectstyle{\CD@zD{\let\objectstyle}}\def\cds@origin{\let
\CD@hA\CD@iB}\def\cds@p{\CD@zD\PileSpacing}\let\cds@pilespacing\cds@p\def
\cds@pixelsize{\CD@zD\CD@@I\CD@gI}\def\cds@portrait{\CD@jA}\def
\cds@PostScript{\CD@nK\global\let\CD@n\empty\CD@BE\CD@DI\empty}\def\cds@PS{%
\CD@nK\global\let\CD@n\empty}\CD@GF\CD@n{\typeout{\CD@tA: try the PostScript
option for better results}}\def\cds@repositionpullbacks{\let\make@pbk\CD@fH
\let\CD@qH\CD@pH}\def\cds@righteqno{\CD@oA}\def\cds@rightshortfall{\CD@zD{%
\dimen5 }}\def\cds@ruleaxis{\CD@zD{\let\axisheight}}\def\cds@cmex{\let\CD@GG
\CD@sB\let\CD@QJ\CD@CJ}\def\cds@s{\cds@height\DiagramCellWidth
\DiagramCellHeight}\def\cds@scriptlabels{\let\labelstyle\scriptstyle}\def
\cds@shortfall{\CD@zD\MapShortFall\dimen1\MapShortFall\dimen5\MapShortFall}%
\def\cds@showfirstpass{\CD@BE{\let\CD@nD}\z@}\def\cds@silent{\def\CD@KB##1{}%
\def\CD@gB##1{}}\let\cds@size\cds@s\def\cds@small{\CellSize2\CD@zC}\def
\cds@snake{\CD@BE\CD@eJ\z@}\def\cds@t{\let\CD@dB\CD@fB}\def\cds@textflow{%
\CD@SB\CD@PJ}\def\cds@thick{\let\CD@rF\tenlnw\CD@LF\CD@NC\CD@BE\MapBreadth{2%
\CD@LF}\CD@@J}\def\cds@thin{\let\CD@rF\tenln\CD@BE\MapBreadth{\CD@NC}\CD@@J}%
\def\cds@tight{\CD@WB}\let\cds@top\cds@t\def\cds@TPIC{\CD@DK}\def
\cds@uppershortfall{\ifPositiveGradient\cds@rightshortfall\else
\cds@leftshortfall\fi}\def\cds@vcenter{\let\CD@dB\CD@cB}\let\cds@vcentre
\cds@vcenter\def\cds@vtriangleheight{\CD@BE\DiagramCellHeight
\DiagramCellHeight\DiagramCellWidth.577035\DiagramCellHeight}\def
\cds@vtrianglewidth{\CD@BE\DiagramCellWidth\DiagramCellWidth
\DiagramCellHeight1.73205\DiagramCellWidth}\def\cds@vmiddle{\let\CD@dB\CD@eB}%
\def\cds@w{\CD@zD\DiagramCellWidth}\let\cds@width\cds@w\def\diagram{\relax
\protect\CD@bC}\def\enddiagram{\protect\CD@SG}\def\CD@bC{\CD@g\CD@uI
\incommdiagtrue\edef\CD@wI{\the\CD@NB}\global\CD@NB\z@\boxmaxdepth\maxdimen
\everycr{}\CD@sK\everymath{}\everyhbox{}\ifx\pdfsyncstop\CD@qK\else
\pdfsyncstop\fi\CD@aC}\def\CD@aC{\CD@y\let\CD@N\CD@ZC\CD@CF\CD@AE\CD@WD}\def
\CD@ZC{\CD@gE\expandafter\CD@aC\else\expandafter\CD@WD\fi}\def\CD@WD{\let
\CD@EH\relax\CD@nE\CD@vE\else\CD@hK\else\CD@KB{landscape ignored without
PostScript}\CD@jA\fi\fi\fi\CD@EJ\setbox2=\vbox\bgroup\CD@JF\CD@VD}\def\CD@cH{%
\CD@nE\CD@fB\else\CD@dB\fi\CD@hA\nointerlineskip\setbox0=\null\ht0-\CD@pI\dp0%
\CD@pI\wd0\CD@kI\box0 \global\CD@QA\CD@kF\global\CD@yA\CD@XB\ifx\CD@NK\CD@qK
\global\CD@RA\CD@kF\else\global\CD@RA\CD@NK\fi\egroup\CD@zF\CD@nE\setbox2=%
\hbox to\dp2{\vrule height\wd2 depth\CD@QA width\z@\global\CD@QA\ht2\ht2\z@
\dp2\z@\wd2\z@\CD@hK\CD@tK{q 0 1 -1 0 0 0 cm}\else\global\CD@iG\CD@IK{0 1
bturn}\fi\box2\CD@gK\hss}\CD@DB\fi\ifnum\CD@yA=1 \else\CD@DB\fi\global
\@ignorefalse\CD@mE\leavevmode\fi\ifvmode\CD@TA\else\ifmmode\CD@PH\CD@GI\else
\CD@qE\CD@gA\fi\ifinner\CD@gA\fi\CD@mE\CD@GI\else\CD@sE\CD@QB\else\CD@TA\fi
\fi\fi\fi\CD@dD}\def\CD@dD{\global\CD@NB\CD@wI\relax\CD@xE\global\CD@ID\else
\aftergroup\CD@mC\fi\if@ignore\aftergroup\ignorespaces\fi\CD@wC\ignorespaces}%
\def\CD@fB{\advance\CD@pI\dimen1\relax}\def\CD@eB{\advance\CD@pI.5\dimen1%
\relax}\def\CD@bB{}\def\CD@cB{\CD@fB\advance\CD@pI\CD@YB\divide\CD@pI2
\advance\CD@pI-\axisheight\relax}\def\CD@aA{}\def\CD@iB{\CD@kF\z@}\def\CD@AA{%
\ifdim\dimen2>\CD@kF\CD@kF\dimen2 \else\dimen2\CD@kF\CD@kI\dimen0 \advance
\CD@kI\dimen2 \fi}\def\CD@QB{\skip0\z@\relax\loop\skip1\lastskip\ifdim\skip1>%
\z@\unskip\advance\skip0\skip1 \repeat\vadjust{\prevdepth\dp\strutbox\penalty
\predisplaypenalty\vskip\abovedisplayskip\CD@UA\penalty\postdisplaypenalty
\vskip\belowdisplayskip}\ifdim\skip0=\z@\else\hskip\skip0 \global\@ignoretrue
\fi}\def\CD@TA{\CD@LG\kern-\displayindent\CD@UA\CD@LG\global\@ignoretrue}\def
\CD@UA{\hbox to\hsize{\CD@fE\ifdim\CD@RA=\z@\else\advance\CD@QA-\CD@RA\setbox
2=\hbox{\kern\CD@RA\box2}\fi\fi\setbox1=\hbox{\ifx\CD@XA\CD@qK\else\CD@k
\CD@XA\CD@ND\fi}\CD@oE\CD@iE\else\advance\CD@QA\wd1 \fi\wd1\z@\box1 \fi\dimen
0\wd2 \advance\dimen0\wd1 \advance\dimen0-\hsize\ifdim\dimen0>-\CD@nA\CD@HA
\fi\advance\dimen0\CD@QA\ifdim\dimen0>\z@\CD@KB{wider than the page by \the
\dimen0 }\CD@HA\fi\CD@iE\hss\else\CD@V\CD@QA\CD@nA\fi\CD@GI\hss\kern-\wd1\box
1 }}\def\CD@GI{\CD@AF\CD@@F\else\CD@SC\global\CD@hG\fi\fi\kern\CD@QA\box2 }%
\CD@tG\CD@wE\CD@YC\CD@XC\def\CD@JF{\CD@cJ\ifdim\DiagramCellHeight=-\maxdimen
\DiagramCellHeight\CD@QC\fi\ifdim\DiagramCellWidth=-\maxdimen
\DiagramCellWidth\CD@QC\fi\global\CD@XC\CD@IF\let\CD@FE\empty\let\CD@z\CD@Q
\let\overprint\CD@eH\let\CD@s\CD@rJ\let\enddiagram\CD@ED\let\\\CD@cC\let\par
\CD@jH\let\CD@MD\empty\let\switch@arg\CD@PB\let\shift\CD@iA\baselineskip
\DiagramCellHeight\lineskip\z@\lineskiplimit\z@\mathsurround\z@\tabskip\z@
\CD@OB}\def\CD@VD{\penalty-123 \begingroup\CD@jA\aftergroup\CD@K\halign
\bgroup\global\advance\CD@NB1 \vadjust{\penalty1}\global\CD@FA\z@\CD@OB\CD@j#%
#\CD@DD\CD@Q\CD@Q\CD@OI\CD@j##\CD@DD\cr}\def\CD@ED{\CD@MD\CD@GD\crcr\egroup
\global\CD@JD\endgroup}\def\CD@j{\global\advance\CD@FA1 \futurelet\CD@EH\CD@i
}\def\CD@i{\ifx\CD@EH\CD@DD\CD@tJ\hskip1sp plus 1fil \relax\let\CD@DD\relax
\CD@vI\else\hfil\CD@k\objectstyle\let\CD@FE\CD@d\fi}\def\CD@DD{\CD@MD\relax
\CD@yI\CD@vI\global\CD@QA\CD@iA\penalty-9993 \CD@ND\hfil\null\kern-2\CD@QA
\null}\def\CD@cC{\cr}\def\across#1{\span\omit\mscount=#1 \global\advance
\CD@FA\mscount\global\advance\CD@FA\m@ne\CD@sF\ifnum\mscount>2 \CD@fJ\repeat
\ignorespaces}\def\CD@fJ{\relax\span\omit\advance\mscount\m@ne}\def\CD@qJ{%
\ifincommdiag\ifx\CD@iD\@fillh\ifx\CD@jD\@fillh\ifdim\dimen3>\z@\else\ifdim
\dimen2>93\CD@@I\ifdim\dimen2>18\p@\ifdim\CD@LF>\z@\count@\CD@bJ\advance
\count@\m@ne\ifnum\count@<\z@\count@20\let\CD@aJ\CD@uJ\fi\xdef\CD@bJ{\the
\count@}\fi\fi\fi\fi\fi\fi\fi}\def\CD@cG#1{\vrule\horizhtdp width#1\dimen@
\kern2\dimen@}\def\CD@uJ{\rlap{\dimen@\CD@@I\CD@V\dimen@{.182\p@}\CD@zH
\dimen@\advance\CD@tI\dimen@\CD@cG0\CD@cG0\CD@cG2\CD@cG6\CD@cG6\CD@cG2\CD@cG0%
\CD@cG0\CD@cG2\CD@cG6\CD@cG0\CD@cG0\CD@cG2\CD@cG2\CD@cG6\CD@cG0\CD@cG0\CD@cG2%
\CD@cG6\CD@cG2\CD@cG2\CD@cG0\CD@cG0}}\def\CD@bJ{10}\def\CD@aJ{}\def\CD@XD{%
\CD@gE\CD@TB\fi\CD@x\CD@WF\CD@HI}\def\CD@x{\CD@QJ\CD@DC\CD@MJ\ifdim\CD@DC=\z@
\else\CD@pF\CD@DC\fi\ifvoid3 \setbox3=\null\ht3\CD@tI\dp3\CD@sI\else\CD@V{\ht
3}\CD@tI\CD@V{\dp3}\CD@sI\fi\dimen3=.5\wd3 \ifdim\dimen3=\z@\CD@tE\else\dimen
3-\CD@XH\fi\else\CD@TB\fi\CD@V{\dimen2}{\wd7}\CD@V{\dimen2}{\wd6}\CD@qJ
\advance\dimen2-2\dimen3 \dimen4.5\dimen2 \dimen2\dimen4 \advance\dimen2%
\CD@eJ\advance\dimen4-\CD@eJ\advance\dimen2-\wd1 \advance\dimen4-\wd5 \ifvoid
2 \else\CD@V{\ht3}{\ht2}\CD@V{\dp3}{\dp2}\CD@V{\dimen2}{\wd2}\fi\ifvoid4 \else
\CD@V{\ht3}{\ht4}\CD@V{\dp3}{\dp4}\CD@V{\dimen4}{\wd4}\fi\advance\skip2\dimen
2 \advance\skip4\dimen4 \CD@tE\advance\skip2\skip4 \dimen0\dimen5 \advance
\dimen0\wd5 \skip3-\skip4 \advance\skip3-\dimen0 \let\CD@jD\empty\else\skip3%
\z@\relax\dimen0\z@\fi}\def\CD@WF{\offinterlineskip\lineskip.2\CD@zC\ifvoid6
\else\setbox3=\vbox{\hbox to2\dimen3{\hss\box6\hss}\box3}\fi\ifvoid7 \else
\setbox3=\vtop{\box3 \hbox to2\dimen3{\hss\box7\hss}}\fi}\def\CD@HI{\kern
\dimen1 \box1 \CD@aJ\CD@iD\hskip\skip2 \kern\dimen0 \ifincommdiag\CD@jE
\penalty1\fi\kern\dimen3 \penalty\CD@GB\hskip\skip3 \null\kern-\dimen3 \else
\hskip\skip3 \fi\box3 \CD@jD\hskip\skip4 \box5 \kern\dimen5}\def\CD@MF{\ifnum
\CD@LH>\CD@TC\CD@V{\dimen1}\objectheight\CD@V{\dimen5}\objectheight\else\CD@V
{\dimen1}\objectwidth\CD@V{\dimen5}\objectwidth\fi}\def\CD@Y{\begingroup
\ifdim\dimen7=\z@\kern\dimen8 \else\ifdim\dimen6=\z@\kern\dimen9 \else\dimen5%
\dimen6 \dimen6\dimen9 \CD@KJ\dimen4\dimen2 \CD@dG{\dimen4}\dimen6\dimen5
\dimen7\dimen8 \CD@KJ\CD@iC{\dimen2}\ifdim\dimen2<\dimen4 \kern\dimen2 \else
\kern\dimen4 \fi\fi\fi\endgroup}\def\CD@jJ{\CD@JI\setbox\z@\hbox{\lower
\axisheight\hbox to\dimen2{\CD@DF\ifPositiveGradient\dimen8\ht\CD@MH\dimen9%
\CD@mI\else\dimen8\dp3 \dimen9\dimen1 \fi\else\dimen8 \ifPositiveGradient
\objectheight\else\z@\fi\dimen9\objectwidth\fi\advance\dimen8
\ifPositiveGradient-\fi\axisheight\CD@Y\unhbox\z@\CD@DF\ifPositiveGradient
\dimen8\dp3 \dimen9\dimen0 \else\dimen8\ht\CD@MH\dimen9\CD@mF\fi\else\dimen8
\ifPositiveGradient\z@\else\objectheight\fi\dimen9\objectwidth\fi\advance
\dimen8 \ifPositiveGradient\else-\fi\axisheight\CD@Y}}}\def\CD@bD{\dimen6
\CD@aK\DiagramCellHeight\dimen7 \CD@WK\DiagramCellWidth\CD@jJ
\ifPositiveGradient\advance\dimen7-\CD@ZK\DiagramCellWidth\else\dimen7 \CD@ZK
\DiagramCellWidth\dimen6\z@\fi\advance\dimen6-\CD@bK\DiagramCellHeight\CD@mK
\setbox0=\rlap{\kern-\dimen7 \lower\dimen6\box\z@}\ht0\z@\dp0\z@\raise
\axisheight\box0 }\def\CD@mK{\setbox0\hbox{\ht\z@\z@\dp\z@\z@\wd\z@\z@\CD@hK
\expandafter\CD@tK{q \CD@eK\space\CD@lK\space\CD@kK\space\CD@eK\space0 0 cm}%
\else\global\CD@iG\CD@eD{\the\CD@TC\space\ifPositiveGradient\else-\fi\the
\CD@LH\space bturn}\fi\box\z@\CD@gK}}\def\CD@vB{\advance\CD@hF-\CD@mI\CD@wJ
\CD@hF\advance\CD@wJ\CD@hI\ifvoid\CD@sH\ifdim\CD@wJ<.1em\ifnum\CD@gD=\@m\else
\CD@aG h\CD@wJ<.1em:objects overprint:\CD@FA\CD@gD\fi\fi\else\ifhbox\CD@sH
\CD@SK\else\CD@TK\fi\advance\CD@wJ\CD@mI\CD@bH{-\CD@mI}{\box\CD@sH}{\CD@wJ}%
\z@\fi\CD@hF-\CD@mF\CD@gD\CD@FA\CD@hI\z@}\def\CD@SK{\setbox\CD@sH=\hbox{%
\unhbox\CD@sH\unskip\unpenalty}\setbox\CD@tH=\hbox{\unhbox\CD@tH\unskip
\unpenalty}\setbox\CD@sH=\hbox to\CD@wJ{\CD@OA\wd\CD@sH\unhbox\CD@sH\CD@PA
\lastkern\unkern\ifdim\CD@PA=\z@\CD@UB\advance\CD@OA-\wd\CD@tH\else\CD@TB\fi
\ifnum\lastpenalty=\z@\else\CD@JA\unpenalty\fi\kern\CD@PA\ifdim\CD@hF<\CD@OA
\CD@JA\fi\ifdim\CD@hI<\wd\CD@tH\CD@JA\fi\CD@jE\CD@hI\CD@wJ\advance\CD@hI-%
\CD@OA\advance\CD@hI\wd\CD@tH\ifdim\CD@hI<2\wd\CD@tH\CD@aG h\CD@hI<2\wd\CD@tH
:arrow too short:\CD@FA\CD@gD\fi\divide\CD@hI\tw@\CD@hF\CD@wJ\advance\CD@hF-%
\CD@hI\fi\CD@tE\kern-\CD@hI\fi\hbox to\CD@hI{\unhbox\CD@tH}\CD@HG}}\CD@tG
\ifinpile\inpiletrue\inpilefalse\inpilefalse\def\pile{\protect\CD@UJ\protect
\CD@uH}\def\CD@uH#1{\CD@l#1\CD@QD}\def\CD@UJ{\CD@nB{pile}\setbox0=\vtop
\bgroup\aftergroup\CD@lD\inpiletrue\let\CD@FE\empty\let\pile\CD@KF\let\CD@QD
\CD@PD\let\CD@GD\CD@FD\CD@yH\baselineskip.5\PileSpacing\lineskip.1\CD@zC
\relax\lineskiplimit\lineskip\mathsurround\z@\tabskip\z@\let\\\CD@wH}\def
\CD@l{\CD@DE\CD@YF\empty\halign\bgroup\hfil\CD@k\let\CD@FE\CD@d\let\\\CD@vH##%
\CD@MD\CD@ND\hfil\CD@Q\CD@R##\cr}\CD@rG\CD@NE{pile only allows one column.}%
\CD@rG\CD@UE{you left it out!}\def\CD@R{\CD@QD\CD@Q\relax\CD@YA{missing \CD@yC
\space inserted after \string\pile}\CD@NE}\def\CD@PD{\CD@MD\crcr\egroup
\egroup}\def\CD@GD{\CD@MD}\def\CD@FD{\CD@MD\relax\CD@QD\CD@YA{missing \CD@yC
\space inserted between \string\pile\space and \CD@HD}\CD@UE}\def\CD@QD{%
\CD@MD}\def\CD@lD{\vbox{\dimen1\dp0 \unvbox0 \setbox0=\lastbox\advance\dimen1%
\dp0 \nointerlineskip\box0 \nointerlineskip\setbox0=\null\dp0.5\dimen1\ht0-%
\dp0 \box0}\ifincommdiag\CD@tJ\penalty-9998 \fi\xdef\CD@YF{pile}}\def\CD@vH{%
\cr}\def\CD@wH{\noalign{\skip@\prevdepth\advance\skip@-\baselineskip
\prevdepth\skip@}}\def\CD@KF#1{#1}\def\CD@TK{\setbox\CD@sH=\vbox{\unvbox
\CD@sH\setbox1=\lastbox\setbox0=\box\voidb@x\CD@tF\setbox\CD@sH=\lastbox
\ifhbox\CD@sH\CD@rC\repeat\unvbox0 \global\CD@QA\CD@ZE}\CD@ZE\CD@QA}\def
\CD@rC{\CD@jE\setbox\CD@sH=\hbox{\unhbox\CD@sH\unskip\setbox\CD@sH=\lastbox
\unskip\unhbox\CD@sH}\ifdim\CD@wJ<\wd\CD@sH\CD@aG h\CD@wJ<\wd\CD@sH:arrow in
pile too short:\CD@FA\CD@gD\else\setbox\CD@sH=\hbox to\CD@wJ{\unhbox\CD@sH}%
\fi\else\CD@gJ\fi\setbox0=\vbox{\box\CD@sH\nointerlineskip\ifvoid0 \CD@tJ\box
1 \else\vskip\skip0 \unvbox0 \fi}\skip0=\lastskip\unskip}\def\CD@gJ{\penalty7
\noindent\unhbox\CD@sH\unskip\setbox\CD@sH=\lastbox\unskip\unhbox\CD@sH
\endgraf\setbox\CD@tH=\lastbox\unskip\setbox\CD@tH=\hbox{\CD@JG\unhbox\CD@tH
\unskip\unskip\unpenalty}\ifcase\prevgraf\cd@shouldnt P\or\ifdim\CD@wJ<\wd
\CD@tH\CD@aG h\CD@wJ<\wd\CD@sH:object in pile too wide:\CD@FA\CD@gD\setbox
\CD@sH=\hbox to\CD@wJ{\hss\unhbox\CD@tH\hss}\else\setbox\CD@sH=\hbox to\CD@wJ
{\hss\kern\CD@hF\unhbox\CD@tH\kern\CD@hI\hss}\fi\or\setbox\CD@sH=\lastbox
\unskip\CD@SK\else\cd@shouldnt Q\fi\unskip\unpenalty}\def\CD@cD{\CD@MJ\ifvoid
3 \setbox3=\null\ht3\axisheight\dp3-\ht3 \dimen3.5\CD@LF\else\dimen4\dp3
\dimen3.5\wd3 \setbox3=\CD@GG{\box3}\dp3\dimen4 \ifdim\ht3=-\dp3 \else\CD@TB
\fi\fi\dimen0\dimen3 \advance\dimen0-.5\CD@LF\setbox0\null\ht0\ht3\dp0\dp3\wd
0\wd3 \ifvoid6\else\setbox6\hbox{\unhbox6\kern\dimen0\kern2pt}\dimen0\wd6 \fi
\ifvoid7\else\setbox7\hbox{\kern2pt\kern\dimen3\unhbox7}\dimen3\wd7 \fi
\setbox3\hbox{\ifvoid6\else\kern-\dimen0\unhbox6\fi\unhbox3 \ifvoid7\else
\unhbox7\kern-\dimen3\fi}\ht3\ht0\dp3\dp0\wd3\wd0 \CD@tE\dimen4=\ht\CD@MH
\advance\dimen4\dp5 \advance\dimen4\dimen1 \let\CD@jD\empty\else\dimen4\ht3
\fi\setbox0\null\ht0\dimen4 \offinterlineskip\setbox8=\vbox spread2ex{\kern
\dimen5 \box1 \CD@iD\vfill\CD@tE\else\kern\CD@eJ\fi\box0}\ht8=\z@\setbox9=%
\vtop spread2ex{\kern-\ht3 \kern-\CD@eJ\box3 \CD@jD\vfill\box5 \kern\dimen1}%
\dp9=\z@\hskip\dimen0plus.0001fil \box9 \kern-\CD@LF\box8 \CD@kE\penalty2 \fi
\CD@tE\penalty1 \fi\kern\PileSpacing\kern-\PileSpacing\kern-.5\CD@LF\penalty
\CD@GB\null\kern\dimen3}\def\CD@cI{\ifhbox\CD@VA\CD@KB{clashing verticals}\ht
\CD@MH.5\dp\CD@VA\dp\CD@MH-\ht5 \CD@yB\ht\CD@MH\z@\dp\CD@MH\z@\fi\dimen1\dp
\CD@VA\CD@xA\prevgraf\unvbox\CD@VA\CD@wA\lastpenalty\unpenalty\setbox\CD@VA=%
\null\setbox\CD@lI=\hbox{\CD@JG\unhbox\CD@lI\unskip\unpenalty\dimen0\lastkern
\unkern\unkern\unkern\kern\dimen0 \CD@HG}\setbox\CD@lF=\hbox{\unhbox\CD@lF
\dimen0\lastkern\unkern\unkern\global\CD@QA\lastkern\unkern\kern\dimen0 }%
\CD@tF\ifnum\CD@xA>4 \CD@zI\repeat\unskip\unskip\advance\CD@mF.5\wd\CD@VA
\advance\CD@mF\wd\CD@lF\advance\CD@mI.5\wd\CD@VA\advance\CD@mI\wd\CD@lI\ifnum
\CD@FA=\CD@lA\CD@OA.5\wd\CD@VA\edef\CD@NK{\the\CD@OA}\fi\setbox\CD@VA=\hbox{%
\kern-\CD@mF\box\CD@lF\unhbox\CD@VA\box\CD@lI\kern-\CD@mI\penalty\CD@wA
\penalty\CD@NB}\ht\CD@VA\dimen1 \dp\CD@VA\z@\wd\CD@VA\CD@tB\CD@vB}\def\CD@zI{%
\ifdim\wd\CD@lF<\CD@QA\setbox\CD@lF=\hbox to\CD@QA{\CD@JG\unhbox\CD@lF}\fi
\advance\CD@xA\m@ne\setbox\CD@VA=\hbox{\box\CD@lF\unhbox\CD@VA}\unskip\setbox
\CD@lF=\lastbox\setbox\CD@lF=\hbox{\unhbox\CD@lF\unskip\unpenalty\dimen0%
\lastkern\unkern\unkern\global\CD@QA\lastkern\unkern\kern\dimen0 }}\def\CD@yB
{\dimen1\dp\CD@VA\ifhbox\CD@VA\CD@xB\else\CD@zB\fi\setbox\CD@VA=\vbox{%
\penalty\CD@NB}\dp\CD@VA-\dp\CD@MH\wd\CD@VA\CD@tB}\def\CD@zB{\unvbox\CD@VA
\CD@wA\lastpenalty\unpenalty\ifdim\dimen1<\ht\CD@MH\CD@aG v\dimen1<\ht\CD@MH:%
rows overprint:\CD@NB\CD@wA\fi}\def\CD@xB{\dimen0=\ht\CD@VA\setbox\CD@VA=%
\hbox\bgroup\advance\dimen1-\ht\CD@MH\unhbox\CD@VA\CD@xA\lastpenalty
\unpenalty\CD@wA\lastpenalty\unpenalty\global\CD@RA-\lastkern\unkern\setbox0=%
\lastbox\CD@tF\setbox\CD@VA=\hbox{\box0\unhbox\CD@VA}\setbox0=\lastbox\ifhbox
0 \CD@kJ\repeat\global\CD@SA-\lastkern\unkern\global\CD@QA\CD@JK\unhbox\CD@VA
\egroup\CD@JK\CD@QA\CD@bH{\CD@SA}{\box\CD@VA}{\CD@RA}{\dimen1}}\def\CD@kJ{%
\setbox0=\hbox to\wd0\bgroup\unhbox0 \unskip\unpenalty\dimen7\lastkern\unkern
\ifnum\lastpenalty=1 \unpenalty\CD@UB\else\CD@TB\fi\ifnum\lastpenalty=2
\unpenalty\dimen2.5\dimen0\advance\dimen2-.5\dimen1\advance\dimen2-%
\axisheight\else\dimen2\z@\fi\setbox0=\lastbox\dimen6\lastkern\unkern\setbox1%
=\lastbox\setbox0=\vbox{\unvbox0 \CD@tE\kern-\dimen1 \else\ifdim\dimen2=\z@
\else\kern\dimen2 \fi\fi}\ifdim\dimen0<\ht0 \CD@aG v\dimen0<\ht0:upper part of
vertical too short:{\CD@tE\CD@NB\else\CD@wA\fi}\CD@xA\else\setbox0=\vbox to%
\dimen0{\unvbox0}\fi\setbox1=\vtop{\unvbox1}\ifdim\dimen1<\dp1 \CD@aG v\dimen
1<\dp1:lower part of vertical too short:\CD@NB\CD@wA\else\setbox1=\vtop to%
\dimen1{\ifdim\dimen2=\z@\else\kern-\dimen2 \fi\unvbox1 }\fi\box1 \kern\dimen
6 \box0 \kern\dimen7 \CD@HG\global\CD@QA\CD@JK\egroup\CD@JK\CD@QA\relax}%
\countdef\CD@u=14 \newcount\CD@CA\newcount\CD@XB\newcount\CD@NB\let\CD@LB
\insc@unt\newcount\CD@FA\newcount\CD@lA\let\CD@mA\CD@XB\newcount\CD@MB\CD@tG
\CD@DF\CD@bI\CD@aI\CD@aI\def\CD@nD{-1}\def\CD@K{\CD@t\ifnum\CD@nD<\z@\else
\begingroup\scrollmode\showboxdepth\CD@nD\showboxbreadth\maxdimen\showlists
\endgroup\fi\CD@bI\CD@zF\CD@CA=\CD@u\advance\CD@CA1 \CD@XB=\CD@CA\ifnum\CD@NB
=1 \CD@JA\fi\advance\CD@XB\CD@NB\dimen1\z@\skip0\z@\count@=\insc@unt\advance
\count@\CD@u\divide\count@2 \ifnum\CD@XB>\count@\CD@KB{The diagram has too
many rows! It can't be reformatted.}\else\CD@NG\CD@WI\fi\CD@cH}\def\CD@NG{%
\CD@NB\CD@CA\CD@uF\ifnum\CD@NB<\CD@XB\setbox\CD@NB\box\voidb@x\advance\CD@NB1%
\relax\repeat\CD@NB\CD@CA\skip\z@\z@\CD@uF\CD@GB\lastpenalty\unpenalty\ifnum
\CD@GB>\z@\CD@KE\repeat\ifnum\CD@GB=-123 \CD@tJ\unpenalty\else\cd@shouldnt D%
\fi\ifx\v@grid\relax\else\CD@NB\CD@XB\advance\CD@NB\m@ne\expandafter\CD@VJ
\v@grid\fi\CD@MB\CD@mA\CD@tB\z@\CD@XG\ifx\h@grid\relax\else\expandafter\CD@LJ
\h@grid\fi\count@\CD@XB\advance\count@\m@ne\CD@YB\ht\count@}\def\CD@KE{%
\ifcase\CD@GB\or\CD@MG\else\CD@uA-\lastpenalty\unpenalty\CD@vA\lastpenalty
\unpenalty\setbox0=\lastbox\CD@WG\fi\CD@wD}\def\CD@wD{\skip1\lastskip\unskip
\advance\skip0\skip1 \ifdim\skip1=\z@\else\expandafter\CD@wD\fi}\def\CD@MG{%
\setbox0=\lastbox\CD@pI\dp0 \advance\CD@pI\skip\z@\skip\z@\z@\advance\CD@NF
\CD@pI\CD@uE\ifnum\CD@NB>\CD@CA\CD@NF\DiagramCellHeight\CD@pI\CD@NF\advance
\CD@pI-\CD@qI\fi\fi\CD@qI\ht0 \CD@NF\CD@qI\setbox\CD@NB\hbox{\unhbox\CD@NB
\unhbox0}\dp\CD@NB\CD@pI\ht\CD@NB\CD@qI\advance\CD@NB1 }\def\CD@WG{\ifnum
\CD@uA<\z@\advance\CD@uA\CD@XB\ifnum\CD@uA<\CD@CA\CD@UG\else\CD@OA\dp\CD@uA
\CD@PA\ht\CD@uA\setbox\CD@uA\hbox{\box\z@\penalty\CD@vA\penalty\CD@GB\unhbox
\CD@uA}\dp\CD@uA\CD@OA\ht\CD@uA\CD@PA\fi\else\CD@UG\fi}\def\CD@UG{\CD@KB{%
diagonal goes outside diagram (lost)}}\def\CD@fI{\advance\CD@uA\CD@XB\ifnum
\CD@uA<\CD@CA\CD@UG\else\ifnum\CD@uA=\CD@NB\CD@VG\else\ifnum\CD@uA>\CD@NB
\cd@shouldnt M\else\CD@OA\dp\CD@uA\CD@PA\ht\CD@uA\setbox\CD@uA\hbox{\box\z@
\penalty\CD@vA\penalty\CD@GB\unhbox\CD@uA}\dp\CD@uA\CD@OA\ht\CD@uA\CD@PA\fi
\fi\fi}\def\CD@WI{\CD@AJ\setbox\CD@PC=\hbox{\CD@k A\@super f\CD@lJ f\CD@ND}%
\CD@ZE\z@\CD@JK\z@\CD@kI\z@\CD@kF\z@\CD@NB=\CD@XB\CD@NF\z@\CD@uB\z@\CD@uF
\ifnum\CD@NB>\CD@CA\advance\CD@NB\m@ne\CD@qI\ht\CD@NB\CD@pI\dp\CD@NB\advance
\CD@NF\CD@qI\CD@rI\advance\CD@uB\CD@NF\CD@KC\CD@ZI\CD@w\ht\CD@NB\CD@qI\dp
\CD@NB\CD@pI\nointerlineskip\box\CD@NB\CD@NF\CD@pI\setbox\CD@NB\null\ht\CD@NB
\CD@uB\repeat\CD@wB\nointerlineskip\box\CD@NB\CD@gG\CD@ZE\DiagramCellWidth{%
width}\CD@gG\CD@JK\DiagramCellHeight{height}\CD@VA\CD@LB\advance\CD@VA-\CD@lA
\advance\CD@VA\m@ne\advance\CD@VA\CD@mA\dimen0\wd\CD@VA\CD@tI\axisheight
\dimen1\CD@uB\advance\dimen1-\CD@YB\dimen2\CD@kI\advance\dimen2-\dimen0
\advance\CD@XB-\CD@CA\advance\CD@LB-\CD@lA}\count@\year\multiply\count@12
\advance\count@\month\ifnum\count@>24398 \loop\iftrue\repeat\fi\def\CD@wB{\CD@qI-\CD@NF\CD@pI\CD@NF\setbox\CD@MH=\null\dp
\CD@MH\CD@NF\ht\CD@MH-\CD@NF\CD@mF\z@\CD@mI\z@\CD@lA\CD@LB\advance\CD@lA-%
\CD@MB\advance\CD@lA\CD@mA\CD@FA\CD@LB\CD@VA\CD@MB\CD@sF\ifnum\CD@FA>\CD@lA
\advance\CD@FA\m@ne\advance\CD@VA\m@ne\CD@tB\wd\CD@VA\setbox\CD@FA=\box
\voidb@x\CD@yB\repeat\CD@w\ht\CD@NB\CD@qI\dp\CD@NB\CD@pI}\def\CD@gG#1#2#3{%
\ifdim#1>.01\CD@zC\CD@PA#2\relax\advance\CD@PA#1\relax\advance\CD@PA.99\CD@zC
\count@\CD@PA\divide\count@\CD@zC\CD@KB{increase cell #3 to \the\count@ em}%
\fi}\def\CD@rI{\CD@FA=\CD@LB\penalty4 \noindent\unhbox\CD@NB\CD@sF\unskip
\setbox0=\lastbox\ifhbox0 \advance\CD@FA\m@ne\setbox\CD@FA\hbox to\wd0{\null
\penalty-9990\null\unhbox0}\repeat\CD@lA\CD@FA\advance\CD@FA\CD@MB\advance
\CD@FA-\CD@mA\ifnum\CD@FA<\CD@LB\count@\CD@FA\advance\count@\m@ne\dimen0=\wd
\count@\count@\CD@MB\advance\count@\m@ne\CD@tB\wd\count@\CD@sF\ifnum\CD@FA<%
\CD@LB\CD@DJ\CD@XG\dimen0\wd\CD@FA\advance\CD@FA1 \repeat\fi\CD@sF\CD@GB
\lastpenalty\unpenalty\ifnum\CD@GB>\z@\CD@vA\lastpenalty\unpenalty\CD@VG
\repeat\endgraf\unskip\ifnum\lastpenalty=4 \unpenalty\else\cd@shouldnt S\fi}%
\def\CD@VG{\advance\CD@vA\CD@lA\advance\CD@vA\m@ne\setbox0=\lastbox\ifnum
\CD@vA<\CD@LB\setbox\CD@vA\hbox{\box0\penalty\CD@GB\unhbox\CD@vA}\else\CD@UG
\fi}\def\CD@bG{}\CD@tG\CD@uE\CD@WB\CD@VB\def\CD@DJ{\advance\dimen0\wd\CD@FA
\divide\dimen0\tw@\CD@uE\dimen0\DiagramCellWidth\else\CD@V{\dimen0}%
\DiagramCellWidth\CD@pJ\fi\advance\CD@tB\dimen0 }\def\CD@XG{\setbox\CD@MB=%
\vbox{}\dp\CD@MB=\CD@uB\wd\CD@MB\CD@tB\advance\CD@MB1 }\def\CD@LJ#1,{\def
\CD@GK{#1}\ifx\CD@GK\CD@RD\else\advance\CD@tB\CD@GK\DiagramCellWidth\CD@XG
\expandafter\CD@LJ\fi}\def\CD@VJ#1,{\def\CD@GK{#1}\ifx\CD@GK\CD@RD\else\ifnum
\CD@NB>\CD@CA\CD@NF\CD@GK\DiagramCellHeight\advance\CD@NF-\dp\CD@NB\advance
\CD@NB\m@ne\ht\CD@NB\CD@NF\fi\expandafter\CD@VJ\fi}\def\CD@pJ{\CD@wE\CD@OA
\dimen0 \advance\CD@OA-\DiagramCellWidth\ifdim\CD@OA>2\MapShortFall\CD@KB{%
badly drawn diagonals (see manual)}\let\CD@pJ\empty\fi\else\let\CD@pJ\empty
\fi}\def\CD@KC{\CD@VA\CD@mA\CD@sF\ifnum\CD@VA<\CD@MB\dimen0\dp\CD@VA\advance
\dimen0\CD@NF\dp\CD@VA\dimen0 \advance\CD@VA1 \repeat}\def\CD@bH#1#2#3#4{%
\ifnum\CD@FA<\CD@LB\CD@OA=#1\relax\setbox\CD@FA=\hbox{\setbox0=#2\dimen7=#4%
\relax\dimen8=#3\relax\ifhbox\CD@FA\unhbox\CD@FA\advance\CD@OA-\lastkern
\unkern\fi\ifdim\CD@OA=\z@\else\kern-\CD@OA\fi\raise\dimen7\box0 \kern-\dimen
8 }\ifnum\CD@FA=\CD@lA\CD@V\CD@kF\CD@OA\fi\else\cd@shouldnt O\fi}\def\CD@w{%
\setbox\CD@NB=\hbox{\CD@FA\CD@lA\CD@VA\CD@mA\CD@PA\z@\relax\CD@sF\ifnum\CD@FA
<\CD@LB\CD@tB\wd\CD@VA\relax\CD@eI\advance\CD@FA1 \advance\CD@VA1 \repeat}%
\CD@V\CD@kI{\wd\CD@NB}\wd\CD@NB\z@}\def\CD@eI{\ifhbox\CD@FA\CD@OA\CD@tB\relax
\advance\CD@OA-\CD@PA\relax\ifdim\CD@OA=\z@\else\kern\CD@OA\fi\CD@PA\CD@tB
\advance\CD@PA\wd\CD@FA\relax\unhbox\CD@FA\advance\CD@PA-\lastkern\unkern\fi}%
\def\CD@ZI{\setbox\CD@sH=\box\voidb@x\CD@VA=\CD@MB\CD@FA\CD@LB\CD@VA\CD@mA
\advance\CD@VA\CD@FA\advance\CD@VA-\CD@lA\advance\CD@VA\m@ne\CD@tB\wd\CD@VA
\count@\CD@LB\advance\count@\m@ne\CD@hF.5\wd\count@\advance\CD@hF\CD@tB\CD@A
\m@ne\CD@gD\@m\CD@sF\ifnum\CD@FA>\CD@lA\advance\CD@FA\m@ne\advance\CD@hF-%
\CD@tB\CD@PI\wd\CD@VA\CD@tB\advance\CD@hF\CD@tB\advance\CD@VA\m@ne\CD@tB\wd
\CD@VA\repeat\CD@mF\CD@kF\CD@mI-\CD@mF\CD@vB}\newcount\CD@GB\def\CD@s{}\def
\CD@t{\CD@vK\mathsurround\z@\hsize\z@\rightskip\z@ plus1fil minus\maxdimen
\parfillskip\z@\linepenalty9000 \looseness0 \hfuzz\maxdimen\hbadness10000
\clubpenalty0 \widowpenalty0 \displaywidowpenalty0 \interlinepenalty0
\predisplaypenalty0 \postdisplaypenalty0 \interdisplaylinepenalty0
\interfootnotelinepenalty0 \floatingpenalty0 \brokenpenalty0 \everypar{}%
\leftskip\z@\parskip\z@\parindent\z@\pretolerance10000 \tolerance10000
\hyphenpenalty10000 \exhyphenpenalty10000 \binoppenalty10000 \relpenalty10000
\adjdemerits0 \doublehyphendemerits0 \finalhyphendemerits0 \baselineskip\z@
\CD@IA\def\parskip{\cd@shouldnt{PS}}\prevdepth\z@}\newbox\CD@KG\newbox\CD@IG
\def\CD@JG{\unhcopy\CD@KG}\def\CD@HG{\unhcopy\CD@IG}\def\CD@iJ{\hbox{}%
\penalty1\nointerlineskip}\def\CD@PI{\penalty5 \noindent\setbox\CD@MH=\null
\CD@mF\z@\CD@mI\z@\ifnum\CD@FA<\CD@LB\ht\CD@MH\ht\CD@FA\dp\CD@MH\dp\CD@FA
\unhbox\CD@FA\skip0=\lastskip\unskip\else\CD@OK\skip0=\z@\fi\endgraf\ifcase
\prevgraf\cd@shouldnt Y \or\cd@shouldnt Z \or\CD@RI\or\CD@XI\else\CD@QI\fi
\unskip\setbox0=\lastbox\unskip\unskip\unpenalty\noindent\unhbox0\setbox0%
\lastbox\unpenalty\unskip\unskip\unpenalty\setbox0\lastbox\CD@tF\CD@GB
\lastpenalty\unpenalty\ifnum\CD@GB>\z@\setbox\z@\lastbox\CD@lB\repeat\endgraf
\unskip\unskip\unpenalty}\def\CD@YJ{\CD@uA\CD@XB\advance\CD@uA-\CD@NB\CD@vA
\CD@FA\advance\CD@vA-\CD@lA\advance\CD@vA1 \expandafter\message{prevgraf=\the
\prevgraf at (\the\CD@uA,\the\CD@vA)}}\def\CD@XI{\CD@CE\setbox\CD@lI=\lastbox
\setbox\CD@lI=\hbox{\unhbox\CD@lI\unskip\unpenalty}\unskip\ifdim\ht\CD@lI>\ht
\CD@PC\setbox\CD@MH=\copy\CD@lI\else\ifdim\dp\CD@lI>\dp\CD@PC\setbox\CD@MH=%
\copy\CD@lI\else\CD@FG\CD@lI\fi\fi\advance\CD@mF.5\wd\CD@lI\advance\CD@mI.5%
\wd\CD@lI\setbox\CD@lI=\hbox{\unhbox\CD@lI\CD@HG}\CD@bH\CD@mF{\box\CD@lI}%
\CD@mI\z@\CD@yB\CD@vB}\def\CD@CE{\ifnum\CD@A>0 \advance\dimen0-\CD@tB\CD@iA-.%
5\dimen0 \CD@A-\CD@A\else\CD@A0 \CD@iA\z@\fi\setbox\CD@MH=\lastbox\setbox
\CD@MH=\hbox{\unhbox\CD@MH\unskip\unskip\unpenalty\setbox0=\lastbox\global
\CD@QA\lastkern\unkern}\advance\CD@iA-.5\CD@QA\unskip\setbox\CD@MH=\null
\CD@mI\CD@iA\CD@mF-\CD@iA}\def\CD@Z{\ht\CD@MH\CD@tI\dp\CD@MH\CD@sI}\def\CD@FG
#1{\setbox\CD@MH=\hbox{\CD@V{\ht\CD@MH}{\ht#1}\CD@V{\dp\CD@MH}{\dp#1}\CD@V{%
\wd\CD@MH}{\wd#1}\vrule height\ht\CD@MH depth\dp\CD@MH width\wd\CD@MH}}\def
\CD@QI{\CD@CE\CD@Z\setbox\CD@lI=\lastbox\unskip\setbox\CD@lF=\lastbox\unskip
\setbox\CD@lF=\hbox{\unhbox\CD@lF\unskip\global\CD@yA\lastpenalty\unpenalty}%
\advance\CD@yA9999 \ifcase\CD@yA\CD@VI\or\CD@YI\or\CD@TI\or\CD@dI\or\CD@cI\or
\CD@SI\else\cd@shouldnt9\fi}\def\CD@VI{\CD@FG\CD@lI\CD@UI\setbox\CD@sH=\box
\CD@lF\setbox\CD@tH=\box\CD@lI}\def\CD@YI{\CD@FG\CD@lF\setbox\CD@lI\hbox{%
\penalty8 \unhbox\CD@lI\unskip\unpenalty\ifnum\lastpenalty=8 \else\CD@xH\fi}%
\CD@UI\setbox\CD@lF=\hbox{\unhbox\CD@lF\unskip\unpenalty\global\setbox\CD@DA=%
\lastbox}\ifdim\wd\CD@lF=\z@\else\CD@xH\fi\setbox\CD@sH=\box\CD@DA}\def\CD@xH
{\CD@KB{extra material in \string\pile\space cell (lost)}}\def\CD@UI{\CD@yB
\ifvoid\CD@sH\else\CD@KB{Clashing horizontal arrows}\CD@mI.5\CD@hF\CD@mF-%
\CD@mI\CD@vB\CD@mI\z@\CD@mF\z@\fi\CD@hI\CD@hF\advance\CD@hI-\CD@mI\CD@hF-%
\CD@mF\CD@JC\CD@FA}\def\CD@RI{\setbox0\lastbox\unskip\CD@iA\z@\CD@Z\ifdim
\skip0>\z@\CD@tJ\CD@A0 \else\ifnum\CD@A<1 \CD@A0 \dimen0\CD@tB\fi\advance
\CD@A1 \fi}\def\VonH{\CD@MA46\VonH{.5\CD@LF}}\def\HonV{\CD@MA57\HonV{.5\CD@LF
}}\def\HmeetV{\CD@MA44\HmeetV{-\MapShortFall}}\def\CD@MA#1#2#3#4{\CD@pB34#1{%
\string#3}\CD@SD\CD@GB-999#2 \dimen0=#4\CD@tI\dimen0\advance\CD@tI\axisheight
\CD@sI\dimen0\advance\CD@sI-\axisheight\CD@CF\CD@HC\CD@ZD}\def\CD@HC#1{%
\setbox0=\hbox{\CD@k#1\CD@ND}\dimen0.5\wd0 \CD@tI\ht0 \CD@sI\dp0 \CD@ZD}\def
\CD@SD{\setbox0=\null\ht0=\CD@tI\dp0=\CD@sI\wd0=\dimen0 \copy0\penalty\CD@GB
\box0 }\def\CD@TI{\CD@GC\CD@yB}\def\CD@dI{\CD@GC\CD@vB}\def\CD@SI{\CD@GC
\CD@yB\CD@vB}\def\CD@GC{\setbox\CD@lI=\hbox{\unhbox\CD@lI}\setbox\CD@lF=\hbox
{\unhbox\CD@lF\global\setbox\CD@DA=\lastbox}\ht\CD@MH\ht\CD@DA\dp\CD@MH\dp
\CD@DA\advance\CD@mF\wd\CD@DA\advance\CD@mI\wd\CD@lI}\CD@tG
\ifPositiveGradient\CD@CI\CD@BI\CD@CI\CD@tG\ifClimbing\CD@rB\CD@qB\CD@rB
\newcount\DiagonalChoice\DiagonalChoice\m@ne\ifx\tenln\nullfont\CD@tJ\def
\CD@qF{\CD@KH\ifPositiveGradient/\else\CD@k\backslash\CD@ND\fi}\else\def
\CD@qF{\CD@rF\char\count@}\fi\let\CD@rF\tenln\def\Use@line@char#1{\hbox{#1%
\CD@rF\ifPositiveGradient\else\advance\count@64 \fi\char\count@}}\def\CD@cF{%
\Use@line@char{\count@\CD@TC\multiply\count@8\advance\count@-9\advance\count@
\CD@LH}}\def\CD@ZF{\Use@line@char{\ifcase\DiagonalChoice\CD@gF\or\CD@fF\or
\CD@fF\else\CD@gF\fi}}\def\CD@gF{\ifnum\CD@TC=\z@\count@'33 \else\count@
\CD@TC\multiply\count@\sixt@@n\advance\count@-9\advance\count@\CD@LH\advance
\count@\CD@LH\fi}\def\CD@fF{\count@'\ifcase\CD@LH55\or\ifcase\CD@TC66\or22\or
52\or61\or72\fi\or\ifcase\CD@TC66\or25\or22\or63\or52\fi\or\ifcase\CD@TC66\or
16\or36\or22\or76\fi\or\ifcase\CD@TC66\or27\or25\or67\or22\fi\fi\relax}\def
\CD@uC#1{\hbox{#1\setbox0=\Use@line@char{#1}\ifPositiveGradient\else\raise.3%
\ht0\fi\copy0 \kern-.7\wd0 \ifPositiveGradient\raise.3\ht0\fi\box0}}\def
\CD@jF#1{\hbox{\setbox0=#1\kern-.75\wd0 \vbox to.25\ht0{\ifPositiveGradient
\else\vss\fi\box0 \ifPositiveGradient\vss\fi}}}\def\CD@jI#1{\hbox{\setbox0=#1%
\dimen0=\wd0 \vbox to.25\ht0{\ifPositiveGradient\vss\fi\box0
\ifPositiveGradient\else\vss\fi}\kern-.75\dimen0 }}\CD@RC{+h:>}{%
\Use@line@char\CD@fF}\CD@RC{-h:>}{\Use@line@char\CD@gF}\CD@nF{+t:<}{-h:>}%
\CD@nF{-t:<}{+h:>}\CD@RC{+t:>}{\CD@jF{\Use@line@char\CD@fF}}\CD@RC{-t:>}{%
\CD@jI{\Use@line@char\CD@gF}}\CD@nF{+h:<}{-t:>}\CD@nF{-h:<}{+t:>}\CD@RC{+h:>>%
}{\CD@uC\CD@fF}\CD@RC{-h:>>}{\CD@uC\CD@gF}\CD@nF{+t:<<}{-h:>>}\CD@nF{-t:<<}{+%
h:>>}\CD@nF{+h:>->}{+h:>>}\CD@nF{-h:>->}{-h:>>}\CD@nF{+t:<-<}{-h:>>}\CD@nF{-t%
:<-<}{+h:>>}\CD@RC{+t:>>}{\CD@jF{\CD@uC\CD@fF}}\CD@RC{-t:>>}{\CD@jI{\CD@uC
\CD@gF}}\CD@nF{+h:<<}{-t:>>}\CD@nF{-h:<<}{+t:>>}\CD@nF{+t:>->}{+t:>>}\CD@nF{-%
t:>->}{-t:>>}\CD@nF{+h:<-<}{-t:>>}\CD@nF{-h:<-<}{+t:>>}\CD@RC{+f:-}{\CD@EF
\null\else\CD@cF\fi}\CD@nF{-f:-}{+f:-}\def\CD@tC#1#2{\vbox to#1{\vss\hbox to#%
2{\hss.\hss}\vss}}\def\hfdot{\CD@tC{2\axisheight}{.5em}}%
\def\vfdot{\CD@tC{1ex}\z@}
\def\CD@bF{\hbox{\dimen0=.3\CD@zC\dimen1\dimen0 \ifnum\CD@LH>\CD@TC\CD@iC{%
\dimen1}\else\CD@dG{\dimen0}\fi\CD@tC{\dimen0}{\dimen1}}}\newarrowfiller{.}%
\hfdot\hfdot\vfdot\vfdot\def\dfdot{\CD@bF\CD@CK}\CD@RC{+f:.}{\dfdot}\CD@RC{-f%
:.}{\dfdot}\def\CD@@K#1{\hbox\bgroup\def\CD@CH{#1\egroup}\afterassignment
\CD@CH
\count@='}\def\lnchar{\CD@@K\CD@qF}\def\CD@dF#1{\setbox#1=\hbox{\dimen5\dimen
#1 \setbox8=\box#1 \dimen1\wd8 \count@\dimen5 \divide\count@\dimen1 \ifnum
\count@=0 \box8 \ifdim\dimen5<.95\dimen1 \CD@gB{diagonal line too short}\fi
\else\dimen3=\dimen5 \advance\dimen3-\dimen1 \divide\dimen3\count@\dimen4%
\dimen3 \CD@dG{\dimen4}\ifPositiveGradient\multiply\dimen4\m@ne\fi\dimen6%
\dimen1 \advance\dimen6-\dimen3 \loop\raise\count@\dimen4\copy8 \ifnum\count@
>0 \kern-\dimen6 \advance\count@\m@ne\repeat\fi}}\def\CD@CG#1{\CD@EF\CD@xJ{#1%
}\else\CD@dF{#1}\fi}\def\CD@IH#1{}\newdimen\objectheight\objectheight1.8ex
\newdimen\objectwidth\objectwidth1em \def\CD@YD{\dimen6=\CD@aK
\DiagramCellHeight\dimen7=\CD@WK\DiagramCellWidth\CD@KJ\ifnum\CD@LH>0 \ifnum
\CD@TC>0 \CD@aF\else\aftergroup\CD@VC\fi\else\aftergroup\CD@UC\fi}\def\CD@VC{%
\CD@YA{diagonal map is nearly vertical}\CD@NA}\def\CD@UC{\CD@YA{diagonal map
is nearly horizontal}\CD@NA}\CD@rG\CD@NA{Use an orthogonal map instead}\def
\CD@aF{\CD@MJ\dimen3\dimen7\dimen7\dimen6\CD@iC{\dimen7}\advance\dimen3-%
\dimen7 \CD@MF\ifnum\CD@LH>\CD@TC\advance\dimen6-\dimen1\advance\dimen6-%
\dimen5 \CD@iC{\dimen1}\CD@iC{\dimen5}\else\dimen0\dimen1\advance\dimen0%
\dimen5\CD@dG{\dimen0}\advance\dimen6-\dimen0 \fi\dimen2.5\dimen7\advance
\dimen2-\dimen1 \dimen4.5\dimen7\advance\dimen4-\dimen5 \ifPositiveGradient
\dimen0\dimen5 \advance\dimen1-\CD@WK\DiagramCellWidth\advance\dimen1 \CD@ZK
\DiagramCellWidth\setbox6=\llap{\unhbox6\kern.1\ht2}\setbox7=\rlap{\kern.1\ht
2\unhbox7}\else\dimen0\dimen1 \advance\dimen1-\CD@ZK\DiagramCellWidth\setbox7%
=\llap{\unhbox7\kern.1\ht2}\setbox6=\rlap{\kern.1\ht2\unhbox6}\fi\setbox6=%
\vbox{\box6\kern.1\wd2}\setbox7=\vtop{\kern.1\wd2\box7}\CD@dG{\dimen0}%
\advance\dimen0-\axisheight\advance\dimen0-\CD@bK\DiagramCellHeight\dimen5-%
\dimen0 \advance\dimen0\dimen6 \advance\dimen1.5\dimen3 \ifdim\wd3>\z@\ifdim
\ht3>-\dp3\CD@TB\fi\fi\dimen3\dimen2 \dimen7\dimen2\advance\dimen7\dimen4
\ifvoid3 \else\CD@tE\else\dimen8\ht3\advance\dimen8-\axisheight\CD@iC{\dimen8%
}\CD@X{\dimen8}{.5\wd3}\dimen9\dp3\advance\dimen9\axisheight\CD@iC{\dimen9}%
\CD@X{\dimen9}{.5\wd3}\ifPositiveGradient\advance\dimen2-\dimen9\advance
\dimen4-\dimen8 \else\advance\dimen4-\dimen9\advance\dimen2-\dimen8 \fi\fi
\advance\dimen3-.5\wd3 \fi\dimen9=\CD@aK\DiagramCellHeight\advance\dimen9-2%
\DiagramCellHeight\CD@tE\advance\dimen2\dimen4 \CD@CG{2}\dimen2-\dimen0%
\advance\dimen2\dp2 \else\CD@CG{2}\CD@CG{4}\ifPositiveGradient\dimen2-\dimen0%
\advance\dimen2\dp2 \dimen4\dimen5\advance\dimen4-\ht4 \else\dimen4-\dimen0%
\advance\dimen4\dp4 \dimen2\dimen5\advance\dimen2-\ht2 \fi\fi\setbox0=\hbox to%
\z@{\kern\dimen1 \ifvoid1 \else\ifPositiveGradient\advance\dimen0-\dp1 \lower
\dimen0 \else\advance\dimen5-\ht1 \raise\dimen5 \fi\rlap{\unhbox1}\fi\raise
\dimen2\rlap{\unhbox2}\ifvoid3 \else\lower.5\dimen9\rlap{\kern\dimen3\unhbox3%
}\fi\kern.5\dimen7 \lower.5\dimen9\box6 \lower.5\dimen9\box7 \kern.5\dimen7
\CD@tE\else\raise\dimen4\llap{\unhbox4}\fi\ifvoid5 \else\ifPositiveGradient
\advance\dimen5-\ht5 \raise\dimen5 \else\advance\dimen0-\dp5 \lower\dimen0 \fi
\llap{\unhbox5}\fi\hss}\ht0=\axisheight\dp0=-\ht0\box0 }\def\NorthWest{\CD@BI
\CD@rB\DiagonalChoice0 }\def\NorthEast{\CD@CI\CD@rB\DiagonalChoice1 }\def
\SouthWest{\CD@CI\CD@qB\DiagonalChoice3 }\def\SouthEast{\CD@BI\CD@qB
\DiagonalChoice2 }\def\CD@aD{\vadjust{\CD@uA\CD@FA\advance\CD@uA
\ifPositiveGradient\else-\fi\CD@ZK\relax\CD@vA\CD@NB\advance\CD@vA-\CD@bK
\relax\hbox{\advance\CD@uA\ifPositiveGradient-\fi\CD@WK\advance\CD@vA\CD@aK
\hbox{\box6 \kern\CD@DC\kern\CD@eJ\penalty1 \box7 \box\z@}\penalty\CD@uA
\penalty\CD@vA}\penalty\CD@uA\penalty\CD@vA\penalty104}}\def\CD@eH#1{\relax
\vadjust{\hbox@maths{#1}\penalty\CD@FA\penalty\CD@NB\penalty\tw@}}\def\CD@lB{%
\ifcase\CD@GB\or\or\CD@bH{.5\wd0}{\box0}{.5\wd0}\z@\or\unhbox\z@\setbox\z@
\lastbox\CD@bH{.5\wd0}{\box0}{.5\wd0}\z@\unpenalty\unpenalty\setbox\z@
\lastbox\or\CD@TG\else\advance\CD@GB-100 \ifnum\CD@GB<\z@\cd@shouldnt B\fi
\setbox\z@\hbox{\kern\CD@mF\copy\CD@MH\kern\CD@mI\CD@uA\CD@XB\advance\CD@uA-%
\CD@NB\penalty\CD@uA\CD@uA\CD@FA\advance\CD@uA-\CD@lA\penalty\CD@uA\unhbox\z@
\global\CD@yA\lastpenalty\unpenalty\global\CD@zA\lastpenalty\unpenalty}\CD@uA
-\CD@yA\CD@vA\CD@zA\CD@fI\fi}\def\CD@TG{\unhbox\z@\setbox\z@\lastbox\CD@uA
\lastpenalty\unpenalty\advance\CD@uA\CD@mA\CD@vA\CD@XB\advance\CD@vA-%
\lastpenalty\unpenalty\dimen1\lastkern\unkern\setbox3\lastbox\dimen0\lastkern
\unkern\setbox0=\hbox to\z@{\unhbox0\setbox0\lastbox\setbox7\lastbox
\unpenalty\CD@eJ\lastkern\unkern\CD@DC\lastkern\unkern\setbox6\lastbox\dimen7%
\CD@tB\advance\dimen7-\wd\CD@uA\ifdim\dimen7<\z@\CD@CI\multiply\dimen7\m@ne
\let\mv\empty\else\CD@BI\def\mv{\raise\ht1}\kern-\dimen7 \fi\ifnum\CD@vA>%
\CD@NB\dimen6\CD@uB\advance\dimen6-\ht\CD@vA\else\dimen6\z@\fi\CD@jJ\CD@mK
\setbox1\null\ht1\dimen6\wd1\dimen7 \dimen7\dimen2 \dimen6\wd1 \CD@KJ\CD@uA
\CD@LH\CD@vA\CD@TC\dimen6\ht1 \CD@KJ\setbox2\null\divide\dimen2\tw@\advance
\dimen2\CD@eJ\CD@eG{\dimen2}\wd2\dimen2 \dimen0.5\dimen7 \advance\dimen0%
\ifPositiveGradient\else-\fi\CD@eJ\CD@dG{\dimen0}\advance\dimen0-\axisheight
\ht2\dimen0 \dimen0\CD@DC\CD@eG{\dimen0}\advance\dimen0\ht2\ht2\dimen0 \dimen
0\ifPositiveGradient-\fi\CD@DC\CD@dG{\dimen0}\advance\dimen0\wd2\wd2\dimen0
\setbox4\null\dimen0 .6\CD@zC\CD@eG{\dimen0}\ht4\dimen0 \dimen0 .2\CD@zC
\CD@dG{\dimen0}\wd4\dimen0 \dimen0\wd2 \ifvoid6\else\dimen1\ht4 \advance
\dimen1\ht2 \CD@CC6+-\raise\dimen1\rlap{\ifPositiveGradient\advance\dimen0-%
\wd6\advance\dimen0-\wd4 \else\advance\dimen0\wd4 \fi\kern\dimen0\box6}\fi
\dimen0\wd2 \ifvoid7\else\dimen1\ht4 \advance\dimen1-\ht2 \CD@CC7-+\lower
\dimen1\rlap{\ifPositiveGradient\advance\dimen0\wd4 \else\advance\dimen0-\wd7%
\advance\dimen0-\wd4 \fi\kern\dimen0\box7}\fi\mv\box0\hss}\ht0\z@\dp0\z@
\CD@bH{\z@}{\box\z@}{\z@}{\axisheight}}\def\CD@CC#1#2#3{\dimen4.5\wd#1 \ifdim
\dimen4>.25\dimen7\dimen4=.25\dimen7\fi\ifdim\dimen4>\CD@zC\dimen4.4\dimen4
\advance\dimen4.6\CD@zC\fi\CD@eG{\dimen4}\dimen5\axisheight\CD@dG{\dimen5}%
\advance\dimen4-\dimen5 \dimen5\dimen4\CD@eG{\dimen5}\advance\dimen0%
\ifPositiveGradient#2\else#3\fi\dimen5 \CD@dG{\dimen4}\advance\dimen1\dimen4 }
\def\CD@eD#1{\expandafter\CD@IK{#1}}\CD@ZA\CD@EK{output is PostScript
dependent}\def\CD@SC{\CD@IK{/bturn {gsave currentpoint currentpoint translate
4 2 roll neg exch atan rotate neg exch neg exch translate } def /eturn {%
currentpoint grestore moveto} def}}\def\CD@gK{\relax\CD@hK\CD@tK{Q}\else
\CD@IK{eturn}\fi} \def\CD@OJ#1{\count@#1\relax\multiply\count@7\advance
\count@16577\divide\count@33154 }\def\CD@fD#1{\expandafter\special{#1}} \def
\CD@xJ#1{\setbox#1=\hbox{\dimen0\dimen#1\CD@dG{\dimen0}\CD@OJ{\dimen0}\setbox
0=\null\ifPositiveGradient\count@-\count@\ht0\dimen0 \else\dp0\dimen0 \fi\box
0 \CD@uA\count@\CD@OJ\CD@LF\CD@fD{pn \the\count@}\CD@fD{pa 0 0}\CD@OJ{\dimen#%
1}\CD@fD{pa \the\count@\space\the\CD@uA}\CD@fD{fp}\kern\dimen#1}}\def\CD@JI{%
\CD@KJ\begingroup\ifdim\dimen7<\dimen6 \dimen2=\dimen6 \dimen6=\dimen7 \dimen
7=\dimen2 \count@\CD@LH\CD@LH\CD@TC\CD@TC\count@\else\dimen2=\dimen7 \fi
\ifdim\dimen6>.01\p@\CD@KI\global\CD@QA\dimen0 \else\global\CD@QA\dimen7 \fi
\endgroup\dimen2\CD@QA\CD@iK\CD@lK{\ifPositiveGradient\else-\fi\dimen6}\CD@iK
\CD@kK{\ifPositiveGradient-\fi\dimen6}\CD@iK\CD@eK{\dimen7}}\def\CD@KI{\CD@hJ
\ifdim\dimen7>1.73\dimen6 \divide\dimen2 4 \multiply\CD@TC2 \else\dimen2=0.%
353553\dimen2 \advance\CD@LH-\CD@TC\multiply\CD@TC4 \fi\dimen0=4\dimen2 \CD@ZG
4\CD@ZG{-2}\CD@ZG2\CD@ZG{-2.5}}\def\CD@AI{\begingroup\count@\dimen0 \dimen2 45%
pt \divide\count@\dimen2 \ifdim\dimen0<\z@\advance\count@\m@ne\fi\ifodd
\count@\advance\count@1\CD@@A\else\CD@y\fi\advance\dimen0-\count@\dimen2
\CD@gE\multiply\dimen0\m@ne\fi\ifnum\count@<0 \multiply\count@-7 \fi\dimen3%
\dimen1 \dimen6\dimen0 \dimen7 3754936sp \ifdim\dimen0<6\p@\def\CD@OG{4000}%
\fi\CD@KJ\dimen2\dimen3\CD@dG{\dimen2}\CD@hJ\multiply\CD@TC-6 \dimen0\dimen2
\CD@ZG1\CD@ZG{0.3}\dimen1\dimen0 \dimen2\dimen3 \dimen0\dimen3 \CD@ZG3\CD@ZG{%
1.5}\CD@ZG{0.3}\divide\count@2 \CD@gE\multiply\dimen1\m@ne\fi\ifodd\count@
\dimen2\dimen1\dimen1\dimen0\dimen0-\dimen2 \fi\divide\count@2 \ifodd\count@
\multiply\dimen0\m@ne\multiply\dimen1\m@ne\fi\global\CD@QA\dimen0\global
\CD@RA\dimen1\endgroup\dimen6\CD@QA\dimen7\CD@RA}\def\CD@OC{255}\let\CD@OG
\CD@OC\def\CD@KJ{\begingroup\ifdim\dimen7<\dimen6 \dimen9\dimen7\dimen7\dimen
6\dimen6\dimen9\CD@@A\else\CD@y\fi\dimen2\z@\dimen3\CD@XH\dimen4\CD@XH\dimen0%
\z@\dimen8=\CD@OG\CD@XH\CD@lC\global\CD@yA\dimen\CD@gE0\else3\fi\global\CD@zA
\dimen\CD@gE3\else0\fi\endgroup\CD@LH\CD@yA\CD@TC\CD@zA}\def\CD@lC{\count@
\dimen6 \divide\count@\dimen7 \advance\dimen6-\count@\dimen7 \dimen9\dimen4
\advance\dimen9\count@\dimen0 \ifdim\dimen9>\dimen8 \CD@@C\else\CD@AC\ifdim
\dimen6>\z@\dimen9\dimen6 \dimen6\dimen7 \dimen7\dimen9 \expandafter
\expandafter\expandafter\CD@lC\fi\fi}\def\CD@@C{\ifdim\dimen0=\z@\ifdim\dimen
9<2\dimen8 \dimen0\dimen8 \fi\else\advance\dimen8-\dimen4 \divide\dimen8%
\dimen0 \ifdim\count@\CD@XH<2\dimen8 \count@\dimen8 \dimen9\dimen4 \advance
\dimen9\count@\dimen0 \CD@AC\fi\fi}\def\CD@AC{\dimen4\dimen0 \dimen0\dimen9
\advance\dimen2\count@\dimen3 \dimen9\dimen2 \dimen2\dimen3 \dimen3\dimen9 }%
\def\CD@ZG#1{\CD@dG{\dimen2}\advance\dimen0 #1\dimen2 }\def\CD@dG#1{\divide#1%
\CD@TC\multiply#1\CD@LH}\def\CD@eG#1{\divide#1\CD@vA\multiply#1\CD@uA}\def
\CD@iC#1{\divide#1\CD@LH\multiply#1\CD@TC}\def\CD@hJ{\dimen6\CD@LH\CD@XH
\multiply\dimen6\CD@LH\dimen7\CD@TC\CD@XH\multiply\dimen7\CD@TC\CD@KJ}\def
\CD@iK#1#2{\begingroup\dimen@#2\relax\loop\ifdim\dimen2<.4\maxdimen\multiply
\dimen2\tw@\multiply\dimen@\tw@\repeat\divide\dimen2\@cclvi\divide\dimen@
\dimen2\relax\multiply\dimen@\@cclvi\expandafter\CD@jK\the\dimen@\endgroup
\let#1\CD@fK}{\catcode`p=12 \catcode`0=12 \catcode`.=12 \catcode`t=12 \gdef
\CD@jK#1pt{\gdef\CD@fK{#1}}}\ifx\errorcontextlines\CD@qK\CD@tJ\let\CD@GH
\relax\else\def\CD@GH{\errorcontextlines\m@ne}\fi\ifnum\inputlineno<0 \let
\CD@CD\empty\let\CD@W\empty\let\CD@mD\relax\let\CD@uI\relax\let\CD@vI\relax
\let\CD@zF\relax\else\def\CD@W{ at line \number\inputlineno}\def\CD@mD{ - first occurred%
}\def\CD@uI{\edef\CD@h{\the\inputlineno}\global\let\CD@jB\CD@h}\def\CD@h{9999%
}\def\CD@vI{\xdef\CD@jB{\the\inputlineno}}\def\CD@jB{\CD@h}\def\CD@zF{\ifnum
\CD@h<\inputlineno\edef\CD@CD{\space at lines \CD@h--\the\inputlineno}\else
\edef\CD@CD{\CD@W}\fi}\fi\let\CD@CD\empty\def\CD@YA#1#2{\CD@GH\errhelp=#2%
\expandafter\errmessage{\CD@tA: #1}}\def\CD@KB#1{\begingroup\expandafter
\message{! \CD@tA: #1\CD@CD}\ifnum\CD@XB>\CD@NB\ifnum\CD@CA>\CD@NB\else\ifnum
\CD@lA>\CD@FA\else\ifnum\CD@LB>\CD@FA\advance\CD@XB-\CD@NB\advance\CD@FA-%
\CD@lA\advance\CD@FA1\relax\expandafter\message{! (error detected at row \the
\CD@XB, column \the\CD@FA, but probably caused elsewhere)}\fi\fi\fi\fi
\endgroup}\def\CD@gB#1{{\expandafter\message{\CD@tA\space Warning: #1\CD@W}}}%
\def\CD@CB#1#2{\CD@gB{#1 \string#2 is obsolete\CD@mD}}\def\CD@AB#1{\CD@CB{%
Dimension}{#1}\CD@DE#1\CD@BB\CD@BB}\def\CD@BB{\CD@OA=}\def\CD@@B#1{\CD@CB{%
Count}{#1}\CD@DE#1\CD@OH\CD@OH}\def\CD@OH{\count@=}\def\HorizontalMapLength{%
\CD@AB\HorizontalMapLength}\def\VerticalMapHeight{\CD@AB\VerticalMapHeight}%
\def\VerticalMapDepth{\CD@AB\VerticalMapDepth}\def\VerticalMapExtraHeight{%
\CD@AB\VerticalMapExtraHeight}\def\VerticalMapExtraDepth{\CD@AB
\VerticalMapExtraDepth}\def\DiagonalLineSegments{\CD@@B\DiagonalLineSegments}%
\ifx\tenln\nullfont\CD@ZA\CD@KH{\CD@eF\space diagonal line and arrow font not
available}\else\let\CD@KH\relax\fi\def\CD@aG#1#2<#3:#4:#5#6{\begingroup\CD@PA
#3\relax\advance\CD@PA-#2\relax\ifdim.1em<\CD@PA\CD@uA#5\relax\CD@vA#6\relax
\ifnum\CD@uA<\CD@vA\count@\CD@vA\advance\count@-\CD@uA\CD@KB{#4 by \the\CD@PA
}\if#1v\let\CD@CH\CD@JK\edef\tmp{\the\CD@uA--\the\CD@vA,\the\CD@FA}\else
\advance\count@\count@\if#1l\advance\count@-\CD@A\else\if#1r\advance\count@
\CD@A\fi\fi\advance\CD@PA\CD@PA\let\CD@CH\CD@ZE\edef\tmp{\the\CD@NB,\the
\CD@uA--\the\CD@vA}\fi\divide\CD@PA\count@\ifdim\CD@CH<\CD@PA\global\CD@CH
\CD@PA\fi\fi\fi\endgroup}\CD@tG\CD@xE\CD@JD\CD@ID\CD@rG\CD@xI{See the message
above.}\CD@rG\CD@lH{Perhaps you've forgotten to end the diagram before
resuming the text, in\CD@uG which case some garbage may be added to the
diagram, but we should be ok now.\CD@uG Alternatively you've left a blank line
in the middle - TeX will now complain\CD@uG that the remaining \CD@S s are
misplaced - so please use comments for layout.}\CD@rG\CD@hD{You have already
closed too many brace pairs or environments; an \CD@HD\CD@uG command was (%
over)due.}\CD@rG\CD@hH{\CD@dC\space and \CD@HD\space commands must match.}%
\def\CD@jH{\ifnum\inputlineno=0 \else\expandafter\CD@iH\fi}\def\CD@iH{\CD@MD
\CD@GD\crcr\CD@YA{missing \CD@HD\space inserted before \CD@kH- type "h"}%
\CD@lH\enddiagram\CD@AG\CD@kH\par}\def\CD@AG#1{\edef\enddiagram{\noexpand
\CD@rD{#1\CD@W}}}\def\CD@rD#1{\CD@YA{\CD@HD\space(anticipated by #1) ignored}%
\CD@xI\let\enddiagram\CD@SG}\def\CD@SG{\CD@YA{misplaced \CD@HD\space ignored}%
\CD@hH}\def\CD@mC{\CD@YA{missing \CD@HD\space inserted.}\CD@hD\CD@AG{closing
group}}\ifx\ifx
\fi\fi

\catcode`\$=3 
\def\vboxtoz{\vbox to\z@}

\def\scriptaxis#1{\@scriptaxis{$\scriptstyle#1$}}
\def\ssaxis#1{\@ssaxis{$\scriptscriptstyle#1$}}
\def\@scriptaxis#1{\dimen0\axisheight\advance\dimen0-\ss@axisheight\raise
\dimen0\hbox{#1}}\def\@ssaxis#1{\dimen0\axisheight\advance\dimen0-%
\ss@axisheight\raise\dimen0\hbox{#1}}

\ifx\boldmath\CD@qK
\let\boldscriptaxis\scriptaxis
\def\boldscript#1{\hbox{$\scriptstyle#1$}}
\else\def\boldscriptaxis#1{\@scriptaxis{\boldmath$\scriptstyle#1$}}
\def\boldscript#1{\hbox{\boldmath$\scriptstyle#1$}}
\fi

\def\raisehook#1#2#3{\hbox{\setbox3=\hbox{#1$\scriptscriptstyle#3$}%
\dimen0\ss@axisheight
\dimen1\axisheight\advance\dimen1-\dimen0
\dimen2\ht3\advance\dimen2-\dimen0%
\advance\dimen2-0.021em\advance\dimen1 #2\dimen2%
\raise\dimen1\box3}}
\def\shifthook#1#2#3{\setbox1=\hbox{#1$\scriptscriptstyle#3$}\dimen0\wd1%
\divide\dimen0 12\CD@zH{\dimen0}
\dimen1\wd1\advance\dimen1-2\dimen0 \advance\dimen1-2\CD@oI\CD@zH{\dimen1}%
\kern#2\dimen1\box1}

\def\@cmex{\mathchar"03}

\def\make@pbk#1{\setbox\tw@\hbox to\z@{#1}\ht\tw@\z@\dp\tw@\z@\box\tw@}\def
\CD@fH#1{\overprint{\hbox to\z@{#1}}}\def\CD@qH{\kern0.11em}\def\CD@pH{\kern0%
.35em}

\def\dblvert{\def\CD@rH{\kern.5\PileSpacing}}\def\CD@rH{}

\def\SEpbk{\make@pbk{\CD@qH\CD@rH\vrule depth 2.87ex height -2.75ex width 0.%
95em \vrule height -0.66ex depth 2.87ex width 0.05em \hss}}

\def\SWpbk{\make@pbk{\hss\vrule height -0.66ex depth 2.87ex width 0.05em
\vrule depth 2.87ex height -2.75ex width 0.95em \CD@qH\CD@rH}}

\def\NEpbk{\make@pbk{\CD@qH\CD@rH\vrule depth -3.81ex height 4.00ex width 0.%
95em \vrule height 4.00ex depth -1.72ex width 0.05em \hss}}

\def\NWpbk{\make@pbk{\hss\vrule height 4.00ex depth -1.72ex width 0.05em
\vrule depth -3.81ex height 4.00ex width 0.95em \CD@qH\CD@rH}}

\def\puncture{{\setbox0\hbox{A}\vrule height.53\ht0 depth-.47\ht0 width.35\ht
0 \kern.12\ht0 \vrule height\ht0 depth-.65\ht0 width.06\ht0 \kern-.06\ht0
\vrule height.35\ht0 depth0pt width.06\ht0 \kern.12\ht0 \vrule height.53\ht0
depth-.47\ht0 width.35\ht0 }}

\def\NEclck{\overprint{\raise2.5ex\rlap{ \CD@rH$\scriptstyle\searrow$}}}
\def\NEanti{\overprint{\raise2.5ex\rlap{ \CD@rH$\scriptstyle\nwarrow$}}}
\def\NWclck{\overprint{\raise2.5ex\llap{$\scriptstyle\nearrow$ \CD@rH}}}
\def\NWanti{\overprint{\raise2.5ex\llap{$\scriptstyle\swarrow$ \CD@rH}}}
\def\SEclck{\overprint{\lower1ex\rlap{ \CD@rH$\scriptstyle\swarrow$}}}
\def\SEanti{\overprint{\lower1ex\rlap{ \CD@rH$\scriptstyle\nearrow$}}}
\def\SWclck{\overprint{\lower1ex\llap{$\scriptstyle\nwarrow$ \CD@rH}}}
\def\SWanti{\overprint{\lower1ex\llap{$\scriptstyle\searrow$ \CD@rH}}}

\def\rhvee{\mkern-10mu\greaterthan}
\def\lhvee{\lessthan\mkern-10mu}
\def\dhvee{\vboxtoz{\vss\hbox{$\vee$}\kern0pt}}
\def\uhvee{\vboxtoz{\hbox{$\wedge$}\vss}}
\newarrowhead{vee}\rhvee\lhvee\dhvee\uhvee

\def\dhlvee{\vboxtoz{\vss\hbox{$\scriptstyle\vee$}\kern0pt}}
\def\uhlvee{\vboxtoz{\hbox{$\scriptstyle\wedge$}\vss}}
\newarrowhead{littlevee}{\mkern1mu\scriptaxis\rhvee}{\scriptaxis\lhvee}%
\dhlvee\uhlvee\ifx\boldmath\CD@qK
\newarrowhead{boldlittlevee}{\mkern1mu\scriptaxis\rhvee}{\scriptaxis\lhvee}%
\dhlvee\uhlvee\else
\def\dhblvee{\vboxtoz{\vss\boldscript\vee\kern0pt}}
\def\uhblvee{\vboxtoz{\boldscript\wedge\vss}}
\newarrowhead{boldlittlevee}{\mkern1mu\boldscriptaxis\rhvee}{\boldscriptaxis
\lhvee}\dhblvee\uhblvee
\fi

\def\rhcvee{\mkern-10mu\succ}
\def\lhcvee{\prec\mkern-10mu}
\def\dhcvee{\vboxtoz{\vss\hbox{$\curlyvee$}\kern0pt}}
\def\uhcvee{\vboxtoz{\hbox{$\curlywedge$}\vss}}
\newarrowhead{curlyvee}\rhcvee\lhcvee\dhcvee\uhcvee

\def\rhvvee{\mkern-13mu\gg}
\def\lhvvee{\ll\mkern-13mu}
\def\dhvvee{\vboxtoz{\vss\hbox{$\vee$}\kern-.6ex\hbox{$\vee$}\kern0pt}}
\def\uhvvee{\vboxtoz{\hbox{$\wedge$}\kern-.6ex \hbox{$\wedge$}\vss}}
\newarrowhead{doublevee}\rhvvee\lhvvee\dhvvee\uhvvee

\def\rhtriangle{\triangleright\mkern1.2mu}
\def\lhtriangle{\triangleleft\mkern.8mu}
\def\uhtriangle{\vbox{\kern-.2ex \hbox{$\scriptscriptstyle\bigtriangleup$}%
\kern-.25ex}}
\def\dhtriangle{\vbox{\kern-.28ex \hbox{$\scriptscriptstyle\bigtriangledown$}%
\kern-.1ex}}
\def\dhblack{\vbox{\kern-.25ex\nointerlineskip\hbox{$\blacktriangledown$}}}%
\def\uhblack{\vbox{\kern-.25ex\nointerlineskip\hbox{$\blacktriangle$}}}%
\def\dhlblack{\vbox{\kern-.25ex\nointerlineskip\hbox{$\scriptstyle
\blacktriangledown$}}}
\def\uhlblack{\vbox{\kern-.25ex\nointerlineskip\hbox{$\scriptstyle
\blacktriangle$}}}
\newarrowhead{triangle}\rhtriangle\lhtriangle\dhtriangle\uhtriangle
\newarrowhead{blacktriangle}{\mkern-1mu\blacktriangleright\mkern.4mu}{%
\blacktriangleleft}\dhblack\uhblack\newarrowhead{littleblack}{\mkern-1mu%
\scriptaxis\blacktriangleright}{\scriptaxis\blacktriangleleft\mkern-2mu}%
\dhlblack\uhlblack

\def\rhla{\hbox{\setbox0=\lnchar55\dimen0=\wd0\kern-.6\dimen0\ht0\z@\raise
\axisheight\box0\kern.1\dimen0}}
\def\lhla{\hbox{\setbox0=\lnchar33\dimen0=\wd0\kern.05\dimen0\ht0\z@\raise
\axisheight\box0\kern-.5\dimen0}}
\def\dhla{\vboxtoz{\vss\rlap{\lnchar77}}}
\def\uhla{\vboxtoz{\setbox0=\lnchar66 \wd0\z@\kern-.15\ht0\box0\vss}}
\newarrowhead{LaTeX}\rhla\lhla\dhla\uhla

\def\lhlala{\lhla\kern.3em\lhla}
\def\rhlala{\rhla\kern.3em\rhla}
\def\uhlala{\hbox{\uhla\raise-.6ex\uhla}}
\def\dhlala{\hbox{\dhla\lower-.6ex\dhla}}
\newarrowhead{doubleLaTeX}\rhlala\lhlala\dhlala\uhlala

\def\hhO{\scriptaxis\bigcirc\mkern.4mu} \def\hho{{\circ}\mkern1.2mu}%
\newarrowhead{o}\hho\hho\circ\circ
\newarrowhead{O}\hhO\hhO{\scriptstyle\bigcirc}{\scriptstyle\bigcirc}

\def\rhtimes{\mkern-5mu{\times}\mkern-.8mu}\def\lhtimes{\mkern-.8mu{\times}%
\mkern-5mu}\def\uhtimes{\setbox0=\hbox{$\times$}\ht0\axisheight\dp0-\ht0%
\lower\ht0\box0 }\def\dhtimes{\setbox0=\hbox{$\times$}\ht0\axisheight\box0 }%
\newarrowhead{X}\rhtimes\lhtimes\dhtimes\uhtimes\newarrowhead+++++

\newarrowhead{Y}{\mkern-3mu\Yright}{\Yleft\mkern-3mu}\Ydown\Yup

\newarrowhead{->}\rightarrow\leftarrow\downarrow\uparrow

\newarrowhead{=>}\Rightarrow\Leftarrow{\@cmex7F}{\@cmex7E}

\newarrowhead{harpoon}\rightharpoonup\leftharpoonup\downharpoonleft
\upharpoonleft

\def\twoheaddownarrow{\rlap{$\downarrow$}\raise-.5ex\hbox{$\downarrow$}}
\def\twoheaduparrow{\rlap{$\uparrow$}\raise.5ex\hbox{$\uparrow$}}
\newarrowhead{->>}\twoheadrightarrow\twoheadleftarrow\twoheaddownarrow
\twoheaduparrow

\def\ltvee{\mkern-1mu{\lessthan}\mkern.4mu}
\newarrowtail{vee}\greaterthan\ltvee\vee\wedge

\newarrowtail{littlevee}{\scriptaxis\greaterthan}{\mkern-1mu\scriptaxis
\lessthan}{\scriptstyle\vee}{\scriptstyle\wedge}\ifx\boldmath\CD@qK
\newarrowtail{boldlittlevee}{\scriptaxis\greaterthan}{\mkern-1mu\scriptaxis
\lessthan}{\scriptstyle\vee}{\scriptstyle\wedge}\else\newarrowtail{%
boldlittlevee}{\boldscriptaxis\greaterthan}{\mkern-1mu\boldscriptaxis
\lessthan}{\boldscript\vee}{\boldscript\wedge}\fi

\newarrowtail{curlyvee}\succ{\mkern-1mu{\prec}\mkern.4mu}\curlyvee\curlywedge

\def\rttriangle{\mkern1.2mu\triangleright}
\newarrowtail{triangle}\rttriangle\lhtriangle\dhtriangle\uhtriangle
\newarrowtail{blacktriangle}\blacktriangleright{\mkern-1mu\blacktriangleleft
\mkern.4mu}\dhblack\uhblack\newarrowtail{littleblack}{\scriptaxis
\blacktriangleright\mkern-2mu}{\mkern-1mu\scriptaxis\blacktriangleleft}%
\dhlblack\uhlblack

\def\rtla{\hbox{\setbox0=\lnchar55\dimen0=\wd0\kern-.5\dimen0\ht0\z@\raise
\axisheight\box0\kern-.2\dimen0}}
\def\ltla{\hbox{\setbox0=\lnchar33\dimen0=\wd0\kern-.15\dimen0\ht0\z@\raise
\axisheight\box0\kern-.5\dimen0}}
\def\dtla{\vbox{\setbox0=\rlap{\lnchar77}\dimen0=\ht0\kern-.7\dimen0\box0%
\kern-.1\dimen0}}
\def\utla{\vbox{\setbox0=\rlap{\lnchar66}\dimen0=\ht0\kern-.1\dimen0\box0%
\kern-.6\dimen0}}
\newarrowtail{LaTeX}\rtla\ltla\dtla\utla

\def\rtvvee{\gg\mkern-3mu}
\def\ltvvee{\mkern-3mu\ll}
\def\dtvvee{\vbox{\hbox{$\vee$}\kern-.6ex \hbox{$\vee$}\vss}}
\def\utvvee{\vbox{\vss\hbox{$\wedge$}\kern-.6ex \hbox{$\wedge$}\kern\z@}}
\newarrowtail{doublevee}\rtvvee\ltvvee\dtvvee\utvvee

\def\ltlala{\ltla\kern.3em\ltla}
\def\rtlala{\rtla\kern.3em\rtla}
\def\utlala{\hbox{\utla\raise-.6ex\utla}}
\def\dtlala{\hbox{\dtla\lower-.6ex\dtla}}
\newarrowtail{doubleLaTeX}\rtlala\ltlala\dtlala\utlala

\def\utbar{\vrule height 0.093ex depth0pt width 0.4em}
\let\dtbar\utbar
\def\rtbar{\mkern1.5mu\vrule height 1.1ex depth.06ex width .04em\mkern1.5mu}%
\let\ltbar\rtbar
\newarrowtail{mapsto}\rtbar\ltbar\dtbar\utbar
\newarrowtail{|}\rtbar\ltbar\dtbar\utbar

\def\rthooka{\raisehook{}+\subset\mkern-1mu}
\def\lthooka{\mkern-1mu\raisehook{}+\supset}
\def\rthookb{\raisehook{}-\subset\mkern-2mu}
\def\lthookb{\mkern-1mu\raisehook{}-\supset}

\def\dthooka{\shifthook{}+\cap}
\def\dthookb{\shifthook{}-\cap}
\def\uthooka{\shifthook{}+\cup}
\def\uthookb{\shifthook{}-\cup}

\newarrowtail{hooka}\rthooka\lthooka\dthooka\uthooka\newarrowtail{hookb}%
\rthookb\lthookb\dthookb\uthookb

\ifx\boldmath\CD@qK\newarrowtail{boldhooka}\rthooka\lthooka\dthooka\uthooka
\newarrowtail{boldhookb}\rthookb\lthookb\dthookb\uthookb\newarrowtail{%
boldhook}\rthooka\lthooka\dthookb\uthooka\else\def\rtbhooka{\raisehook
\boldmath+\subset\mkern-1mu}
\def\ltbhooka{\mkern-1mu\raisehook\boldmath+\supset}
\def\rtbhookb{\raisehook\boldmath-\subset\mkern-2mu}
\def\ltbhookb{\mkern-1mu\raisehook\boldmath-\supset}
\def\dtbhooka{\shifthook\boldmath+\cap}
\def\dtbhookb{\shifthook\boldmath-\cap}
\def\utbhooka{\shifthook\boldmath+\cup}
\def\utbhookb{\shifthook\boldmath-\cup}
\newarrowtail{boldhooka}\rtbhooka\ltbhooka\dtbhooka\utbhooka\newarrowtail{%
boldhookb}\rtbhookb\ltbhookb\dtbhookb\utbhookb\newarrowtail{boldhook}%
\rtbhooka\ltbhooka\dtbhooka\utbhooka\fi

\def\dtsqhooka{\shifthook{}+\sqcap}
\def\ltsqhooka{\mkern-1mu\raisehook{}+\sqsupset}
\def\rtsqhooka{\raisehook{}+\sqsubset\mkern-1mu}
\def\utsqhooka{\shifthook{}+\sqcup}
\newarrowtail{sqhook}\rtsqhooka\ltsqhooka\dtsqhooka\utsqhooka

\newarrowtail{hook}\rthooka\lthookb\dthooka\uthooka\newarrowtail{C}\rthooka
\lthookb\dthooka\uthooka

\newarrowtail{o}\hho\hho\circ\circ
\newarrowtail{O}\hhO\hhO{\scriptstyle\bigcirc}{\scriptstyle\bigcirc}

\newarrowtail{X}\lhtimes\rhtimes\uhtimes\dhtimes\newarrowtail+++++

\newarrowtail{Y}\Yright\Yleft\Ydown\Yup

\newarrowtail{harpoon}\leftharpoondown\rightharpoondown\upharpoonright
\downharpoonright

\newarrowtail{<=}\Leftarrow\Rightarrow{\@cmex7E}{\@cmex7F}

\newarrowfiller{=}=={\@cmex77}{\@cmex77}
\def\vfthree{\mid\!\!\!\mid\!\!\!\mid}
\newarrowfiller{3}\equiv\equiv\vfthree\vfthree

\def\vfdashstrut{\vrule width0pt height1.3ex depth0.7ex}
\def\vfthedash{\vrule width\CD@LF height0.6ex depth 0pt}
\def\hfthedash{\CD@AJ\vrule\horizhtdp width 0.26em}
\def\hfdash{\mkern5.5mu\hfthedash\mkern5.5mu}
\def\vfdash{\vfdashstrut\vfthedash}
\newarrowfiller{dash}\hfdash\hfdash\vfdash\vfdash

\newarrowmiddle+++++

\iffalse
\newarrow{To}----{vee}
\newarrow{Arr}----{LaTeX}
\newarrow{Dotsto}....{vee}
\newarrow{Dotsarr}....{LaTeX}
\newarrow{Dashto}{}{dash}{}{dash}{vee}
\newarrow{Dasharr}{}{dash}{}{dash}{LaTeX}
\newarrow{Mapsto}{mapsto}---{vee}
\newarrow{Mapsarr}{mapsto}---{LaTeX}
\newarrow{IntoA}{hooka}---{vee}
\newarrow{IntoB}{hookb}---{vee}
\newarrow{Embed}{vee}---{vee}
\newarrow{Emarr}{LaTeX}---{LaTeX}
\newarrow{Onto}----{doublevee}
\newarrow{Dotsonarr}....{doubleLaTeX}
\newarrow{Dotsonto}....{doublevee}
\newarrow{Dotsonarr}....{doubleLaTeX}
\else
\newarrow{To}---->
\newarrow{Arr}---->
\newarrow{Dotsto}....>
\newarrow{Dotsarr}....>
\newarrow{Dashto}{}{dash}{}{dash}>
\newarrow{Dasharr}{}{dash}{}{dash}>
\newarrow{Mapsto}{mapsto}--->
\newarrow{Mapsarr}{mapsto}--->
\newarrow{IntoA}{hooka}--->
\newarrow{IntoB}{hookb}--->
\newarrow{Embed}>--->
\newarrow{Emarr}>--->
\newarrow{Onto}----{>>}
\newarrow{Dotsonarr}....{>>}
\newarrow{Dotsonto}....{>>}
\newarrow{Dotsonarr}....{>>}
\fi

\newarrow{Implies}===={=>}
\newarrow{Project}----{triangle}
\newarrow{Pto}----{harpoon}
\newarrow{Relto}{harpoon}---{harpoon}

\newarrow{Eq}=====
\newarrow{Line}-----
\newarrow{Dots}.....
\newarrow{Dashes}{}{dash}{}{dash}{}

\newarrow{SquareInto}{sqhook}--->

\newarrowhead{cmexbra}{\@cmex7B}{\@cmex7C}{\@cmex3B}{\@cmex38}
\newarrowtail{cmexbra}{\@cmex7A}{\@cmex7D}{\@cmex39}{\@cmex3A}
\newarrowmiddle{cmexbra}{\braceru\bracelu}{\bracerd\braceld}{\vcenter{%
\hbox@maths{\@cmex3D\mkern-2mu}}}
{\vcenter{\hbox@maths{\mkern2mu\@cmex3C}}}
\newarrow{@brace}{cmexbra}-{cmexbra}-{cmexbra}
\newarrow{@parenth}{cmexbra}---{cmexbra}
\def\rightBrace{\d@brace[thick,cmex]}
\def\leftBrace{\u@brace[thick,cmex]}
\def\upperBrace{\r@brace[thick,cmex]}
\def\lowerBrace{\l@brace[thick,cmex]}
\def\rightParenth{\d@parenth[thick,cmex]}
\def\leftParenth{\u@parenth[thick,cmex]}
\def\upperParenth{\r@parenth[thick,cmex]}
\def\lowerParenth{\l@parenth[thick,cmex]}

\let\hEq\rEq
\let\vEq\uEq

\def\labelstyle{
\ifincommdiag
\textstyle
\else
\scriptstyle
\fi}
\let\objectstyle\displaystyle

\newdiagramgrid{pentagon}{0.618034,0.618034,1,1,1,1,0.618034,0.618034}{1.%
17557,1.17557,1.902113,1.902113}

\newdiagramgrid{perspective}{0.75,0.75,1.1,1.1,0.9,0.9,0.95,0.95,0.75,0.75}{0%
.75,0.75,1.1,1.1,0.9,0.9}

\diagramstyle[
dpi=300,
vmiddle,nobalance,
loose,
thin,
pilespacing=10pt,%
shortfall=4pt,
]

\ifx\CD@qK\else\CD@PK\relax\CD@FF\CD@e\fi\fi

\CD@vE\CD@hK\else\fi\else\fi\cdrestoreat
\usepackage{amssymb}
\usepackage{mathrsfs}
\usepackage{cite}
\usepackage{tikz-cd}
\usepackage{MnSymbol}
\usepackage{a4wide}

\DeclareMathOperator\C{\mathbb C}
\DeclareMathOperator\Z{\mathbb Z}

\newtheorem{theorem}{Theorem}[section]
\newtheorem{lemma}[theorem]{Lemma}
\newtheorem{cor}[theorem]{Corollary}
\newtheorem{prop}[theorem]{Proposition}
\theoremstyle{definition}
\newtheorem{definition}[theorem]{Definition}
\newtheorem{example}[theorem]{Example}

\theoremstyle{remark}

\numberwithin{equation}{section}

\newcommand{\dontprint}[1]\relax

\newcommand{\De}{\Delta}
\newcommand{\La}{\Lambda}

\newcommand{\Aut}{\operatorname{Aut}}
\newcommand{\und}{\underline}

\newcommand{\hra}{\hookrightarrow}

\renewcommand{\P}{{\mathbb P}}
\newcommand{\A}{{\mathbb A}}

\newcommand{\wt}{\widetilde}
\newcommand{\ot}{\otimes}
\newcommand{\Hom}{\operatorname{Hom}}
\newcommand{\Ext}{\operatorname{Ext}}

\renewcommand{\SS}{{\mathcal S}}
\newcommand{\FF}{{\mathcal F}}

\newcommand{\LL}{{\mathcal L}}

\newcommand{\OO}{{\mathcal O}}

\newcommand{\UU}{{\mathcal U}}

\newcommand{\de}{\delta}
\newcommand{\sub}{\subset}
\newcommand{\Spec}{\operatorname{Spec}}

\newcommand{\ov}{\overline}

\newcommand{\om}{\omega}
\newcommand{\la}{\lambda}
\renewcommand{\a}{\alpha}

\newcommand{\id}{\operatorname{id}}
\newcommand{\GL}{\operatorname{GL}}

\newcommand{\G}{{\mathbb G}}

\newcommand{\ga}{\gamma}
\newcommand{\lan}{\langle}
\newcommand{\ran}{\rangle}

\newcommand{\SL}{{\operatorname{SL}}}
\newcommand{\XX}{{\mathcal X}}
\newcommand{\YY}{{\mathcal Y}}
\newcommand{\fg}{{\mathfrak g}}

\newcommand{\Bun}{{\operatorname{Bun}}}
\newcommand{\End}{{\operatorname{End}}}
\newcommand{\vol}{{\operatorname{vol}}}

\newcommand{\PGL}{{\operatorname{PGL}}}
\renewcommand{\Re}{{\operatorname{Re}}}

\newcommand{\R}{{\mathbb R}}

\newcommand{\ad}{{\operatorname{ad}}}

\title{$L^2$-property for algebraic stacks over local non-archimedean fields}

\author{David Kazhdan}
\author{Alexander Polishchuk}
\thanks{D.K. is partially supported by the ERC grant No 101142781.
A.P. is partially supported by the NSF grant DMS-2349388, by the Simons Travel grant MPS-TSM-00002745,
and within the framework of the HSE University Basic Research Program}

\address{Einstein Institute of Mathematics,
The Hebrew University of Jerusalem,
Jerusalem 91904, Israel}
\email{kazhdan@math.huji.ac.il}
\address{
    Department of Mathematics, 
    University of Oregon, 
    Eugene, OR 97403, USA; National Research University Higher School of Economics  }
  \email{apolish@uoregon.edu}

\begin{document}

\begin{abstract} We introduce an $L^2$-norm on the space of Schwartz half-densities over algebraic stacks over local non-archimedean fields. We show that these $L^2$-norms are finite
for the stacks of $\PGL_2$-bundles on $\P^1$ with parabolic structures at $\ge 3$ points. The latter property was conjectured in the context of the analytic Langlands correspondence
of \cite{EFK1,EFK2}.
\end{abstract}

\maketitle

\section{Introduction}

Let $K$ be a local nonarchimedean field. For a smooth algebraic stack $\XX$ over $K$ one has naturally defined Schwartz spaces of $\kappa$-densities
$\SS(\XX(K),|\om|^{\kappa})$ (where $\kappa\in\C$), see \cite{GK}, \cite{BKP2}. Our interest to these spaces is due to their role in the analytic Langlands correspondence.
Namely, in the case of the stack $\Bun_G$ of $G$-bundles over a curve $C$ over $K$, where $G$ is a semisimple algebraic group, for $\kappa=1/2$, the spaces $\SS(\Bun_G(K),|\om|^{1/2})$ are
equipped with the action of Hecke operators, which form a commutative algebra of operators. Conjecturally, the space $\SS(\Bun_G(K),|\om|^{1/2})$ has
a basis of Hecke eigenfunctions, and there is an extension of Hecke operators to compact operators on a certain $L^2$-space containing $\SS(\Bun_G(K),|\om|^{1/2})$ (see \cite{EFK1}, \cite{BK}).

This picture also makes sense for bundles with parabolic structures, and in the case $G=\PGL_2$, for bundles on $\P^1$ with parabolic structures at fixed $N$ points, where $N\ge 3$,
the extension of the Hecke operators to compact operators on the $L^2$-space of $1/2$-densities on the stable locus was constructed in \cite{EFK2}. The goal of the present paper
is to prove that for $N\ge 3$, every element of the Schwartz space $\SS(\Bun_{\PGL_2}(\P^1,p_1,\ldots,p_N)(K),|\om|^{1/2})$ defines an element of this $L^2$-space (see Theorem A below). 
Here $\Bun_{\PGL_2}(\P^1,p_1,\ldots,p_N)$ is the moduli stack of $\PGL_2$-bundles over $\P^1$ with parabolic structures at $p_1,\ldots,p_N$.

Recall (see \cite{GK}, \cite{BK}) that for quotient stacks $\XX=[X/H]$ (say, with $H=\GL_m$) the Schwartz space $\SS(\XX(K),|\om|^{\kappa})$ is defined as the coinvariants space $\SS(X(K),|\om|^{\kappa})_{H(K)}$. 
For the stacks like $\Bun_G$, one defines the Schwartz spaces as inductive limits of the Schwartz spaces of open substacks of finite type.
For $G=\SL_2$ we showed in \cite{BKP} that for $\Re(\kappa)\ge 1/2$, there is a natural map
$$\pi_{\kappa}:\SS(\Bun_G(K),|\om|^{\kappa})\to C^\infty(\Bun_G^{vs}(K),|\om|^{\kappa}),$$
where $\Bun_G^{vs}$ is the locus of very stable bundles in the coarse moduli space (the map is given by converging integrals over orbits using quotient presentations of open substacks of $\Bun_G$).

The arguments of \cite{BKP} work also in the case of $G=\PGL_2$ and for bundles with parabolic structures at fixed points, see Theorem \ref{vs-smooth-thm} below, where we construct
the parabolic analog of the maps $\pi_{\kappa}$ for $\Re(\kappa)\ge 1/2$. 

In the case $\kappa=1/2+is$, there is a natural $L^2$-norm on the space of $(1/2+is)$-densities on
a $K$-variety (see Sec.\ \ref{L2-def-sec}). 
In particular, we can consider this $L^2$-norm over the moduli of very stable $\PGL_2$-bundles. Now we can formulate our main theorem.


\bigskip

\noindent
{\bf Theorem A}. {\it Fix $N\ge 3$ and $s\in \R$. Assume that the characteristic of $K$ is $\neq 2$.
Then for any $\varphi\in \SS(\Bun_{\PGL_2}(\P^1,p_1,\ldots,p_N)(K),|\om|^{\frac{1}{2}+is})$, the smooth twisted density
$\pi_{1/2+is}(\varphi)$ on the very stable locus $\Bun_{\PGL_2}^{vs}(\P^1,p_1,\ldots,p_N)$ is of class $L^2$.
}

\bigskip

The notion of an {\it admissible stack} was introduced in \cite{GK} (see also Sec.\ \ref{L2-def-sec} below).
We introduce an analog of the $L^2$-norm on the space $\SS(\XX(K),|\om|^{1/2})$ for any admissible algebraic stack $\XX$ over $K$ (possibly of infinite type)
and define the class of $L^2$ stacks as those for which the $L^2$-norm of all elements of $\SS(\XX(K),|\om|^{1/2})$ is finite. This property is relatively easy to study for the stacks of the form $[V/G]$, where
$V$ is a finite-dimensional representation of a reductive group $G$. In particular, in this way we check that the $L^2$-property holds for
the stacks $[(\P^1)^r/G]$, with $G=\SL_2$ or $G=\PGL_2$, for $r\ge 3$. Then we use the results of \cite{EFK2} on Hecke operators to deduce that the $L^2$-property holds for
$\Bun_{\PGL_2}(\P^1,p_1,\ldots,p_N)$, where $N\ge 3$, thus, proving Theorem A.

We expect the similar result to Theorem A to hold for higher genus (with $N\ge 1$ for genus $1$). 
In the case of genus $1$ and $N\ge 1$ the proof should be analogous to that of Theorem A, using the results of \cite{Klyuev}.
In the case of a curve $C$ of genus $\ge 2$ and no parabolic points, the neighborhood
of the trivial bundle in the moduli space $\Bun_G(C)$ is modeled \'etale locally by the quotient stack $[\fg^g/G]$, where $\fg$ is the adjoint representation.
We check in Proposition \ref{adjoint-prop} that this quotient is an $L^2$-stack for $g\ge 2$ for any split semisimple $G$. This essentially reduces the problem of checking
that $\Bun_G(C)$ is $L^2$ to checking that the Hecke operators are bounded with respect to the $L^2$-norm.

The study of examples associated with linear actions suggests that the $L^2$-property may be related to the {\it very good} property of algebraic stacks introduced by Beilinson and Drinfeld
in \cite{BD}. Namely, we conjecture that the very good property (known to hold for $\Bun_G(C)$ for genus $\ge 2$) implies the $L^2$-property
(see Conjecture B and discussion in Sec.\ \ref{L2-def-sec} and Sec.\ \ref{vg-sec}). 

We expect that our results are also valid in the archimedean case (see \cite{Sak} and \cite{BKP} for the relevant definitions in the archimedean case).
In the case $K=\C$ it was proved in \cite{AT} using separation of variables that Hecke eigenfunctions for $\Bun_{\PGL_2}(\P^1,p_1,\ldots,p_N)$ (where $N\ge 3$) are of class $L^2$.

Note that in the case of stacks of bundles over $C$, for which semistability and stability coincide (e.g., for rank $2$ bundles with fixed determinant of odd degree),
the $L^2$ property would follow from the conjectural behavior of Schwartz half-densities near stable bundles, see \cite[Conj.\ 3.5]{BK} and \cite[Conj. A and A']{KP}.
In particular, the results of \cite{KP} imply the $L^2$ property for the stack of rank $2$ bundles with fixed determinant of odd degree when the field $K$ has characteristic zero
and either $C$ has genus $2$ or $C$ is non-hyperelliptic of genus $3$.



\bigskip

\noindent
{\it Convention}. Whenever working with $\PGL_2$-bundles, we assume that the characteristic of $K$ is $\neq 2$.

\medskip

\noindent
{\it Acknowledgment}. We thank Pavel Etingof for helpful discussions.

\section{$L^2$-stacks}

\subsection{Positive sections of local systems on $X(K)$ associated with line bundles}

Let $X$ be a smooth variety over $K$. We can consider complex/real local systems over the topological space $X(K)$, i.e., locally trivial sheaves of
complex/real vector spaces. For such a local system $\LL$, we denote by $C^\infty(X(K),\LL)$ the space of smooth (i.e., locally constant) sections, and by 
$\SS(X(K),\LL)\sub C^\infty(X(K),\LL)$
the subspace of sections with compact support.
 
For a line bundle $L$ on $X$ and a complex number $\kappa$, we denote by $|L|^{\kappa}$ the rank $1$ complex local system on $X(K)$,
corresponding to the $\C^*$-torsor (with respect to the natural topology on $X(K)$) obtained
as the push-out with respect to the homomorphism 
\begin{equation}\label{chi-kappa-eq}
\chi_\kappa:K^*\to \C^*:x\mapsto |x|^{\kappa}
\end{equation}
from the $K^*$-torsor associated with $L$. 

An isomorphism of line bundles $\a:L_1\to L_2$ induces an isomorphism of local systems
$$|\a|^{\kappa}:|L_1|^{\kappa}\to |L_2|^{\kappa}.$$
In particular, a nonvanishing section $s$ of a line bundle $L$ gives a trivializing section $|s|^{\kappa}$ of $|L|^{\kappa}$.

Note since $\chi_{\ov{\kappa}}$ is the composition of $\chi_{\kappa}$ with the complex conjugation on $\C^*$, we have a well defined $\C$-antilinear conjugation map
$$|L|^{\kappa}\to |L|^{\ov{\kappa}}:\varphi\mapsto \ov{\varphi},$$
which is an isomorphism of sheaves of $\R$-vector spaces.

\begin{definition} A {\it positive} rank $1$ real local system on $X(K)$ is a rank $1$ real local system $\LL_\R$ with a fixed identification of the corresponding
$\R^*$-torsor with the push-out of an $\R_{>0}$-torsor.
\end{definition} 

In other words, a positive rank $1$ real local system is given by positive transition functions. For such a local system $\LL_\R$
we have a well defined notion of positivity (or non-negativity) for sections of $\LL_{\R}$, satisfying the following natural properties (where $U\sub X(K)$ is an open subset):
\begin{itemize}
\item
If $\varphi_1\ge 0$, $\varphi_2\ge 0$ for $\varphi_1,\varphi_2\in \LL_\R(U)$ then $\varphi_1+\varphi_2\ge 0$.
\item
If $\LL_\R$ and $\LL'_\R$ are positive rank $1$ real local systems then so is $\LL_\R\ot_{\R} \LL'_\R$ and for $\varphi\in \LL_\R(U)$,
$\varphi'\in \LL'_\R(U)$, if $\varphi\ge 0$ and $\varphi'\ge 0$ then $\varphi\cdot \varphi'\ge 0$.
\item 
If $\LL_\R$ is a positive rank $1$ real local system then so is $\LL^{-1}_\R$, so that the induced positive structure on the trivial local system $\LL_\R\ot_{\R} \LL^{-1}_\R$
is the standard one.
\item 
If $f:X\to Y$ is a morphism and $\LL_\R$ is a positive local system on $Y$ then $f^*\LL_\R$ has a natural positive structure, and
the pullback map $f^*:\LL_\R(U)\to \LL_\R(f^{-1}U)$ preserves positivity: if $\varphi\ge 0$ then $f^*\varphi\ge 0$.
\item
Similar properties hold for strict inequalities.
\end{itemize}

For $s\in \R$, the character $\chi_s$ factors through $\R_{>0}\sub \C^*$, which means that $|L|^s$ is associated with the push-out from an $\R_{>0}$-torsor.
We denote by $|L|^s_{\R}$ the corresponding positive rank $1$ real local system on $X(K)$. 

For any rank $1$ complex local system $\LL$, we have the corresponding rank $1$ positive local system $|\LL|_{\R}$, obtained from $\LL$ as the push-out via $z\mapsto |z|$.
We have a well defined map of sheaves of sets
$$\LL\to |\LL|_{\R}:\varphi\mapsto |\varphi|,$$
such that $|\varphi|\ge 0$, $|f\varphi|=|f|\cdot |\varphi|$ for a function $f$, and for $\varphi,\varphi'\in C^\infty(X(K),\LL)$ one has the triangle inequality
$|\varphi+\varphi'|\le |\varphi|+|\varphi'|$.

Note that for $\LL=|L|^{\kappa}$, where $L$ is a line bundle on $X$, and $\kappa\in\C$, we have an identification
$$||L|^{\kappa}|_{\R}\simeq |L|^{\Re(\kappa)}_{\R}.$$
More generally, for line bundles $L_1,L_2,\ldots$ and complex numbers $\kappa_1,\kappa_2,\ldots$, we have 
$$||L_1|^{\kappa_1}\ot |L_2|^{\kappa_2}\ldots|_{\R}\simeq |L_1|^{\Re(\kappa_1)}_{\R}\ot_{\R} |L_2|^{\Re(\kappa_2)}_\R\ot\ldots,$$
and for local nonvanishing sections $s_i$ of $L_i$ one has 
$$||s_1|^{\kappa_1}\cdot |s_2|^{\kappa_2}\ldots|=|s_1|^{\Re(\kappa_1)|}\cdot |s_2|^{\Re(\kappa_2)}\ldots$$

From now on we will omit the subscript $\R$ when talking about real local systems.
We say that an isomorphism of positive rank 1 real local systems $\a:\LL_1\to \LL_2$ is positive if it such as a section of $\LL_1^{-1}\ot \LL_2$.
Such an isomorphism is compatible with the notion of positivity: if $\varphi\ge 0$ then $\a(\varphi)\ge 0$.
For example, for $s\in \R$, the map $|\a|^s$ induced by an isomorphism of line bundles $\a:L_1\to L_2$ is positive.
 
\subsection{Integration}

As is well known, sections of
$|\om_X|$ can be viewed as densities on $X(K)$, so sections of $\SS(X(K),|\om_X|)$ can be integrated. More generally,
for $\nu\in C^\infty(X(K),|\om_X|)$ such that $\nu\ge 0$, we have a well defined integral $\int_{X(K)}\nu\in [0,+\infty]$.

For arbitrary $\nu\in C^\infty(X(K),|\om_X|)$, we say that the integral $\int_{X(K)}\nu$ converges absolutely if $\int_{X(K)}|\nu|<\infty$.
We have $\int_{X(K)}\ov{\nu}=\ov{\int_{X(K)}\nu}$.

Let $f:X\to Y$ be a smooth morphism, and let $\LL$ be a rank 1 local system on $Y(K)$. 
We have a natural push-forward map
$$f_*:\SS(X(K),f^*\LL\ot |\om_f|)\to \SS(Y(K),\LL)$$
given by 
\begin{equation}\label{push-forward-def}
f_*\varphi(y)=\int_{f^{-1}(y)}\varphi|_{f^{-1}(y)}.
\end{equation}

\begin{definition}
For $\varphi\in C^\infty(f^*\LL\ot |\om_f|)$, we say that $f_*\varphi$ is {\it defined (or converges) as a distribution}, if for every $\psi\in \SS(Y(K),\LL^{-1}\ot |\om_Y|)$,
the integral $\lan f_*\varphi,\psi\ran=\int_{X(K)}\varphi\cdot f^*\psi$ converges absolutely, and hence defines a twisted distribution $f_*\varphi$.
Note that this is equivalent to $f_*|\varphi|$ converging as a distribution, where $|\varphi|\in C^\infty(f^*|\LL|\ot |\om_f|)$.
\end{definition}

\begin{lemma}\label{integration-lem} Let $f:X\to Y$ be a smooth morphism, and let $\LL$ be a rank 1 local system on $Y(K)$. 

\noindent
(i) Assume $\LL$ is positive, $\varphi\in C^\infty(X(K),f^*\LL\ot |\om_f|)$ is such that $\varphi\ge 0$, $f_*(\varphi)$ is defined as a distribution, and $f_*(\varphi)$ is of class $C^\infty$.
Then one has $f_*(\varphi)\ge 0$.
Furthermore, $f_*(\varphi)>0$ over the image of the locus where $\varphi>0$.

\noindent
(ii) Assume that  $\LL$ is positive and $f$ is surjective on $K$-points.
Then for any $\varphi\in \SS(Y(K),\LL)$, $\varphi\ge 0$, there exists $\psi\in \SS(X(K),f^*\LL\ot |\om_f|)$, such that $\psi\ge 0$ and $f_*\psi=\varphi$.


\noindent
(iii) Let $u:Y'\to Y$ be a smooth morphism, so that we have a cartesian diagram
\begin{diagram}
X'&\rTo{u'}&X\\
\dTo{f'}&&\dTo{f}\\
Y'&\rTo{u}&Y
\end{diagram}
Suppose $\varphi\in C^\infty(f^*\LL\ot |\om_f|)$ is such that $f_*\varphi$ converges as a distribution, then $f'_*(u')^*\varphi$ is also defined as a distribution,
and
$$f'_*(u')^*\varphi=u^*f_*\varphi.$$

\noindent
(iv) Let $g:Z\to X$ be a smooth morphism, and let $\varphi\in C^\infty(g^*(f^*\LL\ot |\om_f|)\ot |\om_g|)$ be such that $g_*\varphi$ converges as a distribution to an element in 
$C^\infty(f^*\LL\ot |\om_f|)$.
Then $f_*(g_*|\varphi|)$ converges as a distribution if and only if $(f\circ g)_*|\varphi|$ converges as a distribution, and in this case
$$f_*(g_*\varphi)=(f\circ g)_*(\varphi).$$
\end{lemma}

\begin{proof}
(i) This is clear.

\noindent
(ii) This follows from the fact that locally in analytic topology, $f$ looks like a projection onto a factor of a product of varieties.

\noindent
(iii) Let $\psi\in \SS(Y'(K),u^*\LL^{-1}\ot |\om_{Y'}|)$. We need to check that the integral $\int_{X'(K)}(u')^*|\varphi|\cdot (f')^*|\psi|$ converges.
Let $\ldots\sub X_n\sub X_{n+1}\sub\ldots\sub X(K)$ be an exhaustive sequence of open compact subsets.
For each $n$, consider the induced cartesian square of topological spaces
\begin{diagram}
X'_n&\rTo{u'_n}&X_n\\
\dTo{f'_n}&&\dTo{f_n}\\
Y'&\rTo{u}&Y
\end{diagram}
where $X'_n=(u')^{-1}(X_n)$.

Since $f_n$ and $f'_n$ are proper, $(f'_n)^*|\psi|$ has compact support, and we have equalities of Schwartz sections,
$$(u'_n)_*(f'_n)^*\psi=f_n^*u_*\psi, \ \ (u'_n)_*(f'_n)^*|\psi|=f_n^*u_*|\psi|.$$
Hence,
$$\int_{X'(K)}(u')^*(|\varphi|\de_{X_n-X_{n-1}})\cdot (f')^*|\psi|=\int_{X(K)}|\varphi|\de_{X_n-X_{n-1}}\cdot f^*u_*|\psi|,$$
$$\int_{X'(K)}(u')^*(\varphi\de_{X_n-X_{n-1}})\cdot (f')^*\psi=\int_{X(K)}\varphi\de_{X_n-X_{n-1}}\cdot f^*u_*\psi.$$
Since $u_*|\psi|\in \SS(Y(K),|\LL|^{-1}\ot |\om_Y|)$, the series
$$\sum_n  \int_{X(K)}|\varphi|\de_{X_n-X_{n-1}}\cdot f^*u_*|\psi|=\int_{X(K)}|\varphi|\cdot f^*u_*|\psi|=\lan f_*|\varphi|,u_*|\psi|\ran$$
converges by assumption.
It follows that
$$\lan f'_*(u')^*\varphi,\psi\ran=\int_{X'(K)}(u')^*\varphi\cdot (f')^*\psi=\int_{X(K)}\varphi\cdot f^*u_*\psi=\lan f_*\varphi,u_*\psi\ran=\lan u^*f_*\varphi,\psi\ran,$$
where the integrals are absolutely convergent,
so we get the claimed equality of distributions.

\noindent
(iv) Let $\psi\in \SS(Y(K),\LL^{-1}\ot |\om_Y|)$. Then $g_*(\varphi\cdot (f\circ g)^*\psi)$ converges as a distribution to $g_*(\varphi)\cdot f^*\psi$.

Assume first that $(f\circ g)_*|\varphi|$ converges as a distribution. Then $\int_{Z(K)}\varphi\cdot (f\circ g)^*\psi$ converges absolutely.
Hence, 
$$\int_{Z(K)}\varphi\cdot (f\circ g)^*\psi=\int_{X(K)}g_*(\varphi\cdot (f\circ g)^*\psi),$$ 
and a similar equality holds for $|\varphi$, $|\psi|$.

Let $C_n\sub C_{n+1}\sub\ldots X(K)$ be an exhaustive sequence of compact open subsets.
Then, we have 
$$\int_{X(K)}\de_{C_n-C_{n-1}}\cdot g_*(\varphi\cdot (f\circ g)^*\psi)=\int_{X(K)}\de_{C_n-C_{n-1}}\cdot g_*(\varphi)\cdot f^*\psi,$$
and a similar equality holds for $|\varphi|$, $|\psi|$.
Hence, we deduce that $\int_{X(K)} g_*(\varphi)\cdot f^*\psi$ converges absolutely and
$$\int_{X(K)} g_*(\varphi)\cdot f^*\psi=\int_{X(K)}g_*(\varphi\cdot (f\circ g)^*\psi)=\int_{Z(K)}\varphi\cdot (f\circ g)^*\psi.$$
Hence, $f_*(g_*\varphi)$ converges as a distribution to $(f\circ g)_*(\varphi)$.
 
Conversely, if $f_*(g_*|\varphi|)$ converges as a distribution, we can repeat the above arguments backwards, and get that $(f\circ g)_*(\varphi)$ converges as a distribution
to $f_*(g_*\varphi)$. 
\end{proof}

We also have the following version of the Cauchy-Schwarz inequality.
We say that a twisted distribution, i.e., a functional on $\SS(X(K),\LL)$, where $\LL$ is a positive local system on $X(K)$, is nonnegative if it takes
nonnegative values on nonnegative elements of $\SS(X(K),\LL)$. Thus, we can consider inequalities between such distributions.

\begin{lemma}\label{CS-lem}
Let $f:X\to Y$ be a smooth morphism, $\LL_1$ and $\LL_2$ positive local systems on $Y$.
Suppose for $i=1,2$, $\varphi_i\in C^\infty(X(K),|\om_f|^{1/2}\ot f^*\LL_i)$ is a smooth nonnegative section such that $f_*(\varphi_i^2)$ converges as a distribution
to an element of $C^\infty(Y(K),\LL_i^2)$. Then $f_*(\varphi_1\cdot\varphi_2)$ converges as a distribution, and we have an inequality of twisted distributions,
\begin{equation}\label{CS-ineq}
f_*(\varphi_1\cdot\varphi_2)\le (f_*(\varphi_1^2))^{1/2}\cdot (f_*(\varphi_2^2))^{1/2},
\end{equation}
\end{lemma}

\begin{proof}
Given a nonnegative section $\psi\in \SS(Y(K),|\om_Y|\ot \LL_1^{-1}\ot \LL_2^{-1})$, we need to check that
$$\int_{X(K)} \varphi_1\varphi_2f^*\psi\le \int_{Y(K)} (f_*(\varphi_1^2))^{1/2}\cdot (f_*(\varphi_2^2))^{1/2}\psi.$$
The statement is local on $Y(K)$, so we can assume $\LL_1$ and $\LL_2$ to be trivial. 
Furthermore, it is enough to prove this for $\psi$ supported on a compact open $C\sub Y(K)$, such that $f_*(\varphi_1)|_C$ and $f_*(\varphi_2)|_C$ are constant,
say, $f_*(\varphi_i)|_C=c_i$ for $i=1,2$. 

Now applying the usual Cauchy-Schwarz inequality to the half-densities $\varphi_i\cdot f^*\psi^{1/2}$ on $X(K)$, we get
\begin{align*}
&\int_{X(K)}\varphi_1\varphi_2f^*\psi\le \bigl(\int_{X(K)}\varphi_1^2f^*\psi\bigr)^{1/2}\cdot \bigl(\int_{X(K)}\varphi_2^2f^*\psi\bigr)^{1/2}
=\bigl(\int_{C}f_*(\varphi_1^2)\psi\bigr)^{1/2}\cdot \bigl(\int_{C}f_*(\varphi_2^2)\psi\bigr)^{1/2}\\
&=\bigl(c_1\int_C\psi\bigr)^{1/2}\cdot \bigl(c_2\int_C\psi\bigr)^{1/2}=c_1^{1/2}c_2^{1/2}\int_C\psi=
\int_{Y(K)}(f_*(\varphi_1^2))^{1/2}(f_*(\varphi_2^2))^{1/2}\psi.
\end{align*}
\end{proof}

\subsection{Definition of the $L^2$-pairing}\label{L2-def-sec}

Let $\XX$ be a smooth algebraic stack of finite type over $K$. 
Let $\pi:X\to \XX$ be a smooth morphism, where $X$ is a smooth $K$-variety.
Set $Y:=X\times_{\XX} X$, and let $p_1,p_2:Y\to X$ be the two projections.

For $s\in \R$, the natural isomorphisms
$$\a_1:p_1^*\om_X\ot p_2^*\om_\pi\rTo{\sim} \om_Y, \ \ \a_2:p_2^*\om_X\ot p_1^*\om_\pi\rTo{\sim} \om_Y$$
induce isomorphisms
\begin{align*}
&
|p_1^*\om_X|^{1/2+is}\cdot |p_1^*\om_\pi|^{1/2-is}\cdot |p_2^*\om_X|^{1/2-is}\cdot |p_2^*\om_\pi|^{1/2+is}\simeq\\
&|p_1^*\om_X\ot p_2^*\om_\pi|^{1/2+is}
\ot |p_2^*\om_X\ot p_1^*\om_\pi|^{1/2-is}\rTo{|\a_1|^{1/2+is}|\a_2|^{1/2-is}} |\om_Y|
\end{align*}
(these isomorphisms are positive for $s=0$).

Now, given sections 
$$\varphi_1,\varphi_2\in C^\infty(X(K),|\pi^*\om_{\XX}|^{1/2+is}\cdot |\om_\pi|)\simeq C^\infty(X(K),|\om_X|^{1/2+is}\cdot |\om_\pi|^{1/2-is}),$$
we have on $Y=X\times_{\XX} X$,
$$p_1^*\varphi_1\cdot p_2^*\ov{\varphi}_2\in |p_1^*\om_X|^{1/2+is}\cdot |p_1^*\om_\pi|^{1/2-is}\cdot |p_2^*\om_X|^{1/2-is}\cdot |p_2^*\om_\pi|^{1/2+is},
$$
hence, $|\a_1|^{1/2+is}|\a_2|^{1/2-is}(p_1^*\varphi_1\cdot p_2^*\ov{\varphi}_2)$ is a smooth section of $|\om_Y|$.
Thus, we can look at the integral
$$\lan\varphi_1,\varphi_2\ran_s:=\int_{Y(K)}(|\a_1|^{1/2+is}\cdot |\a_2|^{1/2-is})(p_1^*\varphi_1\cdot p_2^*\ov{\varphi}_2).$$


In the case $s=0$, we will set $\lan\varphi_1,\varphi_2\ran:=\lan\varphi_1,\varphi_2\ran_0$. Note that in this case, if $\varphi_1\ge 0$ and $\varphi_2\ge 0$ then the integrand of 
$\lan\varphi_1,\varphi_2\ran$ is $\ge 0$.
It is clear from the definition that $\lan \varphi_1,\varphi_2\ran=\ov{\lan \varphi_2,\varphi_1\ran}$.

More generally, we can consider a pair of smooth morphisms $\pi_i:X_i\to \XX$, $i=1,2$, and repeat the above definition of $\lan\varphi_1,\varphi_2\ran_s$ with 
$\varphi_i\in C^\infty(X_i(K),|\om_{X_i}|^{1/2+is}\cdot |\om_{\pi_i}|^{1/2-is})$, $i=1,2$, and with
$Y=X_1\times_{\XX} X_2$.

\begin{lemma}\label{independence-lem} 
Let $\pi_i:X_i\to \XX$, $i=1,2$ be a pair of smooth morphisms, 
and let $f_i:X'_i\to X_i$, $i=1,2$, be smooth morphisms, and let $\pi'_i:=\pi_i\circ f_i:X'_i\to \XX$, $i=1,2$.
Suppose we are given 
$$\varphi_i\in C^\infty(X'_i(K),|\om_{X'}|^{1/2+is}\ot |\om_{\pi'_i}|^{1/2-is})\simeq
C^\infty(X'_i(K),|\om_{f_i}|\ot f_i^*(|\om_{X}|^{1/2+is}\ot |\om_{\pi_i}|^{1/2-is})),$$ 
for $i=1,2$, such that 
$f_{i_*}\varphi_i$ and $f_{i*}|\varphi_i|$ converge as distributions and belong to  \\
$C^\infty(X_i(K),|\om_X|^{1/2}\ot |\om_{\pi_i}|^{1/2})$.
Then one has
$$\lan |\varphi_1|, |\varphi_2|\ran=\lan f_{1*}|\varphi_1|,f_{2*}|\varphi_2|\ran,$$
where both sides are in $[0,+\infty]$. Furthermore, if one of the sides is finite then
$$\lan \varphi_1, \varphi_2\ran=\lan f_{1*}\varphi_1,f_{2*}\varphi_2\ran.$$
\end{lemma}

\begin{proof}
The idea is that applying first $(f_1\times \id)_*$ and then $(\id\times f_2)_*$, we get
\begin{equation}\label{independence-eq}
\int_{(X'_1\times_{\XX} X'_2)(K)}p_1^*\varphi_1\cdot p_2^*\ov{\varphi}_2=\int_{(X_1\times_{\XX} X'_2)(K)} p_1^*(f_{1*}\varphi_1)\cdot p_2^*\ov{\varphi}_2=
\int_{(X_1\times_{\XX}X_2)(K)} 
p_1^*(f_{1*}\varphi_1)\cdot p_2^*(f_{2*}\ov{\varphi}_2),
\end{equation}
where we use the cartesian squares
\begin{diagram}
X'_1\times_{\XX} X'_2&\rTo{p_1}& X'_1\\
\dTo{f_1\times \id}&&\dTo{f_1}\\
X_1\times_{\XX} X'_2&\rTo{p_1}& X_1
\end{diagram}
\begin{diagram}
X_1\times_{\XX} X'_2&\rTo{p_2}& X'_2\\
\dTo{\id\times f_2}&&\dTo{f_2}\\
X_1\times_{\XX} X_2&\rTo{p_2}& X_2
\end{diagram}

In more detail, we first apply Lemma \ref{integration-lem}(iii) to the first cartesian square and $\varphi_1$ to deduce that
$$(f_1\times\id)_*(p_1^*\varphi_1)=p_1^*f_{1*}\varphi_1$$
as twisted distributions on $X_1\times_{\XX}X'_2$.

Let $\ldots\sub C_n\sub C_{n+1}\sub\ldots \sub (X_1\times_{\XX}X'_2)(K)$ be an exhaustive sequence of open compacts.
Applying the above equality of distributions to Schwartz sections $p_2^*\ov{\varphi}_2\cdot \de_{C_n-C_{n+1}}$ we get equalities
\begin{equation}\label{indep-terms-int-eq}
\int_{(X'_1\times_{\XX} X'_2)(K)}p_1^*\varphi_1\cdot p_2^*\ov{\varphi}_2\de_{(f_1\times\id)^{-1}(C_n-C_{n+1)}}=
\int_{(X_1\times_{\XX} X'_2)(K)} p_1^*(f_{1*}\varphi_1)\cdot p_2^*\ov{\varphi}_2\de_{C_n-C_{n+1}},
\end{equation}
$$\int_{(X'_1\times_{\XX} X'_2)(K)}p_1^*|\varphi_1|\cdot p_2^*|\varphi_2|\de_{(f_1\times\id)^{-1}(C_n-C_{n+1})}=
\int_{(X_1\times_{\XX} X'_2)(K)} p_1^*(f_{1*}|\varphi_1|)\cdot p_2^*|\varphi_2|\de_{C_n-C_{n+1}}.$$
Thus,
$$\int_{(X'_1\times_{\XX} X'_2)(K)}p_1^*|\varphi_1|\cdot p_2^*|\varphi_2|=\int_{(X_1\times_{\XX} X'_2)(K)} (f_{1*}|\varphi_1|)\cdot |\varphi_2|,$$
and if this is finite then
$$\int_{(X'_1\times_{\XX} X'_2)(K)}p_1^*\varphi_1\cdot p_2^*\ov{\varphi}_2=\int_{(X_1\times_{\XX} X'_2)(K)} (f_{1*}\varphi_1)\cdot \ov{\varphi}_2$$
(we can think of both sides as absolutely converging infinite series with terms given by the sides of \eqref{indep-terms-int-eq}).


Similarly, we apply Lemma \ref{integration-lem}(iii) to the second cartesian square and $\ov{\varphi}_2$ to get the
equality of twisted distributions on $X_1\times_{\XX} X_2$,
$$p_2^*f_{2*}(\ov{\varphi}_2)=(\id\times f_2)_*p_2^*\ov{\varphi}_2,$$
which leads to the second equality in \eqref{independence-eq}.
\end{proof}

\begin{definition}
Let $\pi:X\to \XX$ be a smooth morphism, and let $s\in \R$. We say that
$\varphi\in C^\infty(X(K),|\om_X|^{1/2+is}\ot |\om_\pi|^{1/2-is})$ {\it is of class $L^2$ (relative to $\XX$)} if 
$\lan |\varphi|,|\varphi|\ran<+\infty$.
\end{definition}

One can check in the standard way that if $\varphi_1,\varphi_2$ are of class $L^2$ then so is $\varphi_1+\varphi_2$.

Recall (see \cite{GK}) that a stack of finite type $\XX$ over $K$ is called {\it admissible} if there exists a smooth covering $\pi:X\to \XX$ (called {\it admissible}), such that for any $Y\to \XX$,
the map $X\times_\XX Y\to Y$ is surjective on $K$-points.
In this case $\SS(\XX(K),|\om|^{1/2})$ is a quotient of $\SS(X(K),|\pi^*\om_{\XX}|^{1/2}\cdot |\om_\pi|)\simeq \SS(X(K),|\om_X|^{1/2}\cdot |\om_\pi|^{1/2})$
by the subspaces consisting of differences of two different push-forwards from Schwartz twisted densities on $(X\times_{\XX} X)(K)$.

\begin{prop}\label{L2-prop-def}
Assume that there exists an admissible covering $\pi:X\to \XX$ such that 
all $\varphi\in \SS(X(K),|\om_X|^{1/2}\ot |\om_\pi|^{1/2})$ are of class $L^2$ relative to $\XX$ (or equivalently, all such $\varphi\ge 0$ are of class $L^2$).
Then the same is true for any smooth morphism $\wt{X}\to \XX$. In this case $\lan\varphi_1,\varphi_2\ran$ converges absolutely for
all $\varphi_1,\varphi_2\in \SS(X(K),|\om_X|^{1/2+is}\ot |\om_\pi|^{1/2-is})$ (where $s\in\R$), and the pairing $\lan\cdot,\cdot\ran$ descends to a pairing on 
$\SS(\XX(K),|\om_{\XX}|^{1/2+is})$ that does not depend on an admissible covering.
\end{prop}

\begin{proof}
First, suppose $f:X'\to X$ is a smooth morphism, and let $\pi'=\pi\circ f:X'\to \XX$. By Lemma \ref{independence-lem}, any $\varphi$ on $X'$ is $L^2$ relative to $\XX$, and
$\lan\varphi,\varphi'\ran=\lan f_*\varphi,f_*\varphi'\ran$. 
If $\wt{\pi}:\wt{X}\to \XX$ is a smooth morphism, then $X':=X\times_{\XX}\wt{X}$ has smooth projections to $X$ and $\wt{X}$, and the projection $X'\to \wt{X}$ is surjective on $K$-points.
Applying Lemma \ref{independence-lem} together with Lemma \ref{integration-lem}(ii), we see that $L^2$-property holds for $\wt{X}\to \XX$ and the pairings are compatible. 
This easily implies the remaining assertions.
\end{proof}

\begin{definition}
In the situation of Proposition \ref{L2-prop-def} (where $\XX$ is of finite type) we say that $\XX$ is an {\it $L^2$-stack}.
We say that a stack $\XX$ of possibly infinite type is $L^2$, if every open substack of finite type in $\XX$ is $L^2$ (or, equivalently, there exists
an exhaustive filtration of $\XX$ by open $L^2$ substacks of finite type).
\end{definition}

Let us rewrite our definition in the case of the global quotient $\XX=[X/G]$, where $G$ is a linear algebraic group such that every $G$-torsor over $\Spec(K)$ is trivial.
Let us fix a trivialization $\vol_{\fg}\in {\det}(\fg)^*$, where $\fg$ is the Lie algebra of $G$, and let $dg$ denote the corresponding right invariant Haar measure on $G$.
Let
$$\de(g)=\de_G(g):=|{\det}^{-1}(\ad(g))|$$
denote the modular character.

\begin{lemma}\label{quotient-lem}
In the above situation, for $\varphi_1,\varphi_2\in \SS(X(K),|\om_X|^{1/2})$, one has
$$\lan \varphi_1 |\vol_{\fg}|^{1/2}, \varphi_2 |\vol_{\fg}|^{1/2}\ran=
\int_{G(K)} (\varphi_1,g^*\varphi_2) \de(g)^{1/2}dg,$$
where $(\varphi,\psi)=\int_{X(K)}\varphi\cdot \ov{\psi}$.
\end{lemma}

\begin{proof}
For the presentation $\pi:X\to [X/G]$ we have $Y=G\times X$, such that the two projections to $X$ are $p_X:(g,x)\mapsto x$ and $a:(g,x)\mapsto gx$.

Let
$$\xi_x:\fg\to T_xX$$
denote the tangent map at $1$ to the map $a_x:g\mapsto gx$.
It is easy to see that $\xi_x$ is the fiber at $x$ of the map of $G$-equivariant bundles on $G$,
$$\xi:\fg\ot \OO_X\to T_X,$$
where we use the adjoint action of $G$ on $\fg$.
The complex of $G$-equivariant bundles $[\fg\ot \OO_X\to T_X]$ placed in degrees $[-1,0]$ corresponds to the
tangent complex of $[X/G]$. Hence, we get an identification of the relative tangent bundle (as a $G$-equivariant bundle),
$$T_\pi\simeq \fg\ot \OO_X.$$

Let $\fg\ot\OO_G\rTo{\sim} T_G$ be the trivialization of the tangent sheaf of $G$ using right translations.
Then the tangent map to $a:G\times X\to X$ at $(g,x)$ is given by
$$\fg\oplus T_xX\rTo{(\xi_{gx},g)}T_{gx}X.$$

To compute the isomorphisms $\a_1$ and $\a_2$ on $G\times X$, we need to know explicitly the isomorphisms
$a^*T_\pi\simeq T_{p_X}$ and $p_X^*T_\pi\simeq T_a$.
The former isomorphism comes from the identification of both sides with $\fg\ot \OO$. The latter isomorphism,
evaluated at $(g,x)\in G\times X$, fits into a commutative diagram
\begin{diagram}
(T_a)_{g,x}&\rTo{}&T_{(g,x)}G\times X\\
\dTo{\sim}&&\dTo{}\\
(T_\pi)_x&\rTo{}& T_x X
\end{diagram}
where the lower horizontal arrow can be identified with $\xi_x:\fg\to T_x X$.
From our description of $da$ above we see that
$$(T_a)_{g,x}=\ker(da)_{g,x)}=\{(v,-g^{-1}\xi_{gx}(v) \ |\ v\in \fg\}\sub \fg\oplus T_xX.$$
Note that $g^{-1}\xi_{gx}(v)=\xi_x\ad(g)^{-1}(v)$.
Hence, the isomorphism $\fg\simeq (T_\pi)_x\simeq (T_a)_{g,x}$ is given by
$$\fg\to (T_a)_{g,x}: v\mapsto (-\ad(g)(v),\xi_x(v)).$$

It follows that the isomorphism
$$\a_1:p_X^*\om_X\ot a^*\om_\pi\to \om_G\boxtimes \om_X$$
corresponds to the right invariant identification $\om_G\simeq \det^{-1}(\fg)$.
On the other hand, the isomorphism
$$\a_2:a^*\om_X\ot p_X^*\om_\pi\to \om_G\boxtimes \om_X$$
is given at $(g,x)$ by the map
$${\det}(da_g:T_xX\to T_{gx})\cdot {\det}^{-1}(\ad(g)): {\det}^{-1}(T_{gx})\ot {\det}^{-1}(\fg) \to {\det}^{-1}(T_xX)\ot {\det}^{-1}(\fg)$$

Hence,
$$|\a_1|^{1/2}|\a_2|^{1/2}(\varphi_1(x)|\vol_{\fg}|^{1/2}\cdot \ov{\varphi}_2(gx) |\vol_{\fg}|^{1/2})=|{\det}^{-1}(\ad(g))|^{1/2}\cdot \varphi_1(x)\cdot (g^*\varphi_2)(x)\cdot |\vol_{\fg}|.$$
This immediately leads to the claimed formula (integrating over $X(K)$ first we get $(\varphi_1,g^*\varphi_2)\cdot \de(g)^{1/2}dg$).
\end{proof}

\begin{example}\label{BG-ex}
Let $\XX=BG$, where $G$ is a linear algebraic group such that every $G$-torsor over $\Spec(K)$ is trivial.
This stack is $L^2$ if an only if $\int_{G(K)} \de(g)^{1/2}dg<\infty$, where $dg$ is the right invariant Haar measure on $G$.
Assume the characteristic is zero and $G$ is connected. 
Then this is equivalent to compactness of $G(K)$. Indeed, note that $\de$ is trivial on the unipotent radical $R_u$ of $G$.
Since $R_u$ is necessarily split over $K$, if it is nontrivial it contains an additive group over $K$, which would lead to the infinite integral.
Hence, $R_u=1$, so $G$ is reductive. If we have an embedding of $\G_m$ into $G$, we again would get the infinite integral, hence $G$ is anisotropic,
i.e., $G(K)$ is compact.
\end{example}

Recall that a smooth irreducible stack is called {\it very good} if for any $n>0$, the locus of points where the dimension of the automorphism group is equal to $n$, has
codimension $>n$.

\bigskip

\noindent
{\bf Conjecture B}. {\it If $\XX$ is very good then $\XX$ is an $L^2$-stack.}

\bigskip

Note that there exist  $L^2$-stacks which are not very good, for example $BE=[pt/E]$, where $E$ is an elliptic curve, or $[\A^1/\G_m]$, or $[(\P^1)^3/\SL_2]$.
We will show (see Proposition \ref{SL2-very-good}) that if $V$ is a finite-dimensional representation of $\SL_2$ then $[V/\SL_2]$ is $L^2$ if and only it is very good.

\subsection{Some general results}

First, we want to show that the $L^2$-property can be checked locally.

\begin{lemma}\label{local-lem}
(i) Let $\pi:X\to \XX$ be smooth, and assume that for every point $x\in X(K)$, there exists $\varphi\in \SS(X(K),|\om_X|^{1/2}\ot |\om_\pi|^{1/2})$, $\varphi\ge 0$,
such that $\varphi(x)>0$ and $\lan \varphi,\varphi\ran<+\infty$. Then every $\varphi\in \SS(X(K),|\om_X|^{1/2}\ot |\om_\pi|^{1/2})$ is $L^2$ relative to $\XX$.

\noindent
(ii) Assume there exists a Zariski open covering $\XX=\cup_i \UU_i$, such that each $\UU_i$ is $L^2$.
Then $\XX$ is $L^2$.
\end{lemma}

\begin{proof}
(i) Given $\varphi\in \SS(X(K),|\om_X|^{1/2}\ot |\om_\pi|^{1/2})$, $\varphi\ge 0$, for every point $x$ in the support $S$ of $\varphi$, we can find a compact neighborhood of $x$,
$U_x\sub X(K)$, and an element $\varphi_x\in \SS(X(K),|\om_X|^{1/2}\ot |\om_\pi|^{1/2})$ of class $L^2$ relative to $\XX$, such that $\varphi_x\ge 0$ and $\varphi_x=1$ on $U_x$.
Since $S$ is compact, it is covered by finitely many such neighborhoods $U_{x_i}$. Hence, there exists a constant $C>0$, such that $\varphi\le C \sum_i \varphi_{x_i}$,
which implies that $\varphi$ is $L^2$ relative to $\XX$.

\noindent
(ii) Let $U_i\to \wt{\UU}_i$ be admissible coverings. Then $X:=\sqcup U_i\to \XX$ satisfies the assumptions of (i), and the assertion follows from (i).
\end{proof}


Next, we relate our $L^2$-norm for $(1/2+is)$-densities on a stack $\XX$ 
with the $L^2$-norm on a nice open subscheme $U\sub \XX$. In the case when $\XX$ is the stack of bundles (or bundles with parabolic structure),
the assumptions of the Lemma below are satisfied for the open subscheme of very stable bundles (see Theorem \ref{vs-smooth-thm} below).

\begin{lemma}\label{very-stable-lem}
Let $\pi:X\to \XX$ be a smooth morphism, where $X$ is a smooth variety over $K$, and suppose $U\sub \XX$ is a dense open substack, which is a scheme.
Let $\wt{U}=\pi^{-1}(U)$, and let $\pi_U:\wt{U}\to U$ be the morphism induced by $\pi$. 
For
$\varphi\in \SS(X(K),|\om_X|^{1/2+is}\ot |\om_\pi|^{1/2-is})$, consider $\varphi|_{\wt{U}}\in C^\infty(U,|\om_{\wt{U}}|^{1/2+is}\ot |\om_{\pi_U}|^{1/2-is})$.
Assume that the push-forward $\psi:=(\pi_U)_*(\varphi|_{\wt{U}})$ converges as a distribution and belongs to $C^\infty(U,|\om_U|^{1/2+is})$, and the same is true
for $|\varphi|$. 
Then $\lan|\varphi|,|\varphi|\ran=(|\psi|,|\psi|)$, and if this is finite then $\lan\varphi,\varphi\ran=(\psi,\psi)$. In particular, $\varphi$ is $L^2$ if and only if $\psi$ is.
\end{lemma}

\begin{proof}
Let $Y=X\times_{\XX} X$, $V=Y\times_{\XX} U\simeq \wt{U}\times_U\wt{U}$. Then $V$ is a dense open in $Y$. Hence, the complement to $V(K)$ in $Y(K)$ has measure zero,
and so
$$\lan\varphi,\varphi\ran=\int_{Y(K)} p_1^*\varphi\cdot p_2^*\ov{\varphi}=\int_{V(K)} p_1^*(\varphi|_{\wt{U}})\cdot p_2^*(\ov{\varphi}|_{\wt{U}}),$$
(and similarly for $|\varphi|$ instead of $\varphi$).
Now the assertion follows from Lemma \ref{independence-lem} applied to the identity covering $U\to U$ and the smooth morphism $\wt{U}\to U$.
\end{proof}


Next, we have a general result that says that stacks smooth and representable over $L^2$-stacks are themselves $L^2$.

\begin{prop} Let $f:\XX\to \YY$ be a smooth representable morphism of algebraic stacks over $K$ such that $\YY$ is an admissible $L^2$-stacks, and $f(K)$ is essentially surjective.
Then $\XX$ is also an admissible $L^2$-stack.
\end{prop}

\begin{proof}
Let $Y\to \YY$ be an admissible smooth covering, and let $X:=\XX\times_{\YY} Y$. We denote by $f:X\to Y$ the smooth map induced by $f:\XX\to \YY$.
Then it is easy to see that $X\to \XX$ is also an admissible smooth covering.
Furthermore, for $X':=X\times_\XX X$ and $Y':=Y\times_\YY Y$ we have cartesian diagrams
\begin{equation}\label{X'Y'XY-cart-diag}
\begin{diagram}
X'&\rTo{f'}& Y'\\
\dTo{p_i}&&\dTo{p_i}\\
X&\rTo{f}&Y
\end{diagram}
\end{equation}
for $i=1,2$ (where $p_i$ are the natural projections), where $f'$ is induced by $f$ on both factors.

Now for a nonnegative $\varphi\in \SS(X(K),|\om_X|^{1/2}\ot |\om_{X/\XX}|^{1/2})$, we have to show convergence of $\int_{X'(K)} p_1^*\varphi\cdot p_2^*\varphi$.
Let us consider the smooth morphism $f':X'\to Y'$. It is enough to check that $f'_*(p_1^*\varphi\cdot p_2^*\varphi)$
converges as a distribution (which is automatically nonnegative), and that the value of this distribution on the function $1$ converges.
In other words, if $\ldots\sub C_n\sub C_{n+1}\sub\ldots \sub Y'(K)$ is an exhaustive sequence of open compacts, then we need to check that the monotone
sequence $f'_*(p_1^*\varphi\cdot p_2^*\varphi)(\de_{C_n})$ is bounded.

Note that $\varphi$ can be thought of as a section of $|\om_f|^{1/2}\ot f^*\LL$, where $\LL:=|\om_Y|^{1/2}\ot |\om_{Y/\YY}|^{1/2}$.
Hence, we can view $p_1^*\varphi$ and $p_2^*\varphi$ on $X'$ as sections of $|\om_{f'}|^{1/2}\ot (f')^*(p_1^*\LL)$ and $|\om_{f'}|^{1/2}\ot (f')^*(p_2^*\LL_2)$,
where we have a positive isomorphism $p_1^*\LL\ot p_2^*\LL\simeq |\om_Y'|$. 

Note that $\varphi^2$ is a nonnegative Schwartz section of $|\om_f|\ot f^*\LL^2$, hence, $f_*(\varphi^2)$ is a nonnegative Schwartz section of $\LL^2$, 
hence, $(f_*(\varphi^2))^{1/2}$ is a Schwartz section of $\LL$. By Lemma \ref{integration-lem}(iii), using the cartesian squares \eqref{X'Y'XY-cart-diag},
we deduce the equality of twisted nonnegative distributions
$$f'_*(p_i^*\varphi^2)=p_i^*f_*(\varphi^2)\in C^\infty.
$$

Applying the Cauchy-Schwarz inequality in fibers of $f'$ (see Lemma \ref{CS-lem}),
we deduce the inequality of nonnegative distributions on $Y'$,
$$f'_*(p_1^*\varphi\cdot p_2^*\varphi)\le p_1^*(f_*(\varphi^2))^{1/2}\cdot p_2^*(f_*(\varphi^2))^{1/2}.$$
Since the right-hand hand has finite value on $1$ (by the assumption that $\YY$ is $L^2$), so does the left-hand side.
\end{proof}

\begin{definition}
We say that a morphism of admissible stacks $\XX\to \YY$ is an {\it admissible $H$-gerbe}, where $H$ is a group scheme over $K$, if there exists an admissible covering $Y\to \YY$,
such that $\XX\times_{\YY} Y\simeq BH\times Y$.
\end{definition}

\begin{example}\label{adm-gerbe-ex}
Note that if $\XX\to \YY$ is an admissible $H$-gerbe then the map on $K$-points, $\XX(K)\to \YY(K)$ is essentially surjective. For example, the $\mu_2$-gerbe
$B\SL_2\to B\PGL_2$ is not admissible, since $B\PGL_2(K)$ has a nontrivial object corresponding to an order $2$ element of the Brauer group.
However, the natural morphism $[(\P^1)^r/\SL_2]\to [(\P^1)^r/\PGL_2]$ for $r\ge 1$ is an admissible $\mu_2$-gerbe: the equivalence $[(\P^1)^r/\PGL_2]\simeq [(\P^1)^{r-1}/B]$ shows that
$(\P^1)^r\to [(\P^1)^r/\PGL_2]$ is an admissible covering (and it splits the gerbe).
\end{example}

\begin{lemma}\label{gerbe-lem} 
Let $\XX\to \YY$ be an admissible $H$-gerbe, where $H$ is a finite discrete group such that $|H|$ is invertible in $K$. Then $\XX$ is $L^2$ if and only if $\YY$ is $L^2$.
Furthermore, in this case the natural maps of Schwartz spaces $\SS(\XX(K),|\om|^{\kappa})\to \SS(\YY(K),|\om|^{\kappa})$ are surjective.
\end{lemma}

\begin{proof}
First we observe that the diagonal map $\XX\to \XX\times_\YY \XX$ is a finite \'etale map of order $|H|$. Indeed, this is a local statement, so this follows from the similar statement about the 
diagonal map $BH\to BH\times BH$.

Let $Y\to \YY$ be an admissible smooth covering such that the $H$-gerbe becomes trivial over $Y$. Then the map $Y\to \YY$ factors through a smooth map $Y\to \XX$. This
immediately implies surjectivity of the maps of Schwartz spaces of $\kappa$-densities.

On the other hand, the cartesian diagram
\begin{diagram}
Y\times_\XX Y&\rTo{f}& Y\times_\YY Y\\
\dTo{}&&\dTo{}\\
\XX&\rTo{}& \XX\times_\YY \XX
\end{diagram}
shows that the map $f:Y\times_\XX Y\to Y\times_\YY Y$ is an \'etale covering of degree $|H|$.

Note also that triviality of $\om_{BH}$ implies that $\om_{\XX/YY}|_Y$ is trivial, hence, $\om_{Y/\XX}\simeq \om_{Y/\YY}$. Thus, 
$$\SS(Y(K),|\om_Y|^{1/2}\ot |\om_{Y/\XX}|^{1/2})\simeq \SS(Y(K),|\om_Y|^{1/2}\ot |\om_{Y/\YY}|^{1/2}).$$
The condition of being $L^2$ for $\YY$ (resp., for $\XX$) means that a certain nonnegative density $\eta$ on $(Y\times_\YY Y)(K)$ (resp., $\eta'$ on $(Y\times_\XX Y)(K)$)
associated with an element of the above space, has convergent integral.
It is easy to check that the the densities $\eta$ and $\eta'$ are related by $\eta'=f^*\eta$.
Since $f_*(f^*\eta)=|H|\cdot \eta$, convergence of the integral of $\eta$ is equivalent to that of $f^*\eta$. 
\end{proof}

\subsection{Examples for linear actions}



Below we denote by $O\sub K$ the ring of integers, and by $t\in O$ a generator of the maximal ideal.

Let $G$ be a split reductive group over $K$. Then $G$ is defined over $O$, so we have the maximal compact subgroup $G(O)\sub G(K)$. 
Let $V$ be an (algebraic) finite-dimensional linear representation of $G$ over $K$.

There always exists a $G(O)$-invariant lattice $\La_0$ in $V$.
We can present $V(K)$ as an increasing union of $G(O)$-invariant
$O$-lattices $t^{-n}\La_0\sub V$. Let $|dv|$ denote the positive measure on $V(K)$ associated with a trivialization of $\det(V)$. Then
for any $g\in G(K)$, we have $g^*|dv|=|{\det}_V(g)|\cdot |dv|$. Thus, for any $G(O)$-invariant lattice $\La\sub V$, we have 
$$(\de_{\La}\cdot |dv|^{1/2}, g^*(\de_{\La}\cdot |dv|^{1/2}))=|{\det}_V(g)|^{1/2}\cdot \vol(\La\cap g^{-1}\La),$$
where we take volume is computed with respect to $|dv|$.
Thus, $\de_{\La}\cdot |dv|^{1/2}$ is $L^2$ if and only if the integral 
\begin{equation}\label{La-integral-eq}
\int_{g\in G(K)} |{\det}_V(g)|^{1/2}\cdot \de(g)^{1/2}\cdot \vol(\La\cap g^{-1}\La) dg
\end{equation}
(where $dg$ is the right-invariant Haar measure) converges.

This leads to the following criterion for $[V/G]$ to be an $L^2$-stack for a split reductive group $G$ in terms of the character of $V$.
Let $T\sub G$ denote a split maximal torus, and let $X_\bullet(T)$ denote the coweight lattice, with $X_\bullet(T)_+\sub X_\bullet(T)$ the subset
of dominant coweights. We denote by $2\rho$ the sum of positive roots.

\begin{prop}\label{weight-crit-prop} 
In the above situation the stack $[V/G]$ is $L^2$ if and only if the series
\begin{align*}
&\sum_{\nu\in X_\bullet(T)_+}q^{\lan 2\rho,\nu\ran-\frac{1}{2}\sum_{\la}\lan \la,\nu\ran\dim(V_\la)+\sum_{\la: \lan \la,\nu\ran<0}\lan \la,\nu\ran \dim(V_\la)}\\
&=
\sum_{\nu\in X_\bullet(T)_+}q^{\lan 2\rho,\nu\ran-\frac{1}{2}\sum_{\la:\lan \la,\nu\ran>0}\lan \la,\nu\ran\dim(V_\la)+\frac{1}{2}\sum_{\la: \lan \la,\nu\ran<0}\lan \la,\nu\ran \dim(V_\la)}
\end{align*}
converges.
\end{prop}

\begin{proof}
As we have seen, $[V/G]$ is $L^2$ if and only if the integral \eqref{La-integral-eq} converges for every lattice $\La$ of the form $t^{-n}\La_0$, $n>0$.
Furthermore, since $\vol(t^{-n}\La_0\cap g^{-1}(t^{-n}\La_0))=q^n\vol(\La_0\cap g^{-1}\La_0)$, it is enough to consider one $G(O)$-invariant lattice $\La$.

We have the weight decomposition $V=\bigoplus_\la V_\la$ with respect to $T$, 
and we can find $T(\OO)$-invariant lattices
$$\La_i=\bigoplus \La_{i,\la}, \ \ i=1,2,$$
where $\La_{i,\la}$ are $O$-lattices in $V_\la$, such that 
$$\La_1\sub \La\sub \La_2.$$

By $G(O)$-invariance of $\La$, the volume of $\La\cap g^{-1}\La$ depends only on the double $G(O)$-coset of $g$. Hence,
using the Cartan decomposition 
$$G(K)=\sqcup_{\nu\in X_\bullet(T)_+} G(O)\nu(t)G(O),$$
we can rewrite the integral \eqref{La-integral-eq} as 
$$\sum_{\nu\in X_\bullet(T)_+}\vol(G(O)\nu(t)G(O))\cdot |{\det}_V(\nu(t))|^{1/2}\cdot \vol(\La\cap \nu(t)^{-1}\La).$$ 

Since $\nu(t)$ preserves the weight decomposition of $V$ and acts as $t^{\lan\la,\nu\ran}$ on $V_\la$, we have
$$|{\det}_V(\nu(t))|^{1/2}=|t^{\sum_\la \lan\la,\nu\ran\dim(V_\la)}|^{1/2}=q^{-\frac{1}{2}\sum_{\la}\lan \la,\nu\ran\dim(V_\la)},$$
and
$$\vol(\La_i\cap \nu(t)^{-1}\La_i)=\prod_\la \vol(\La_{i,\la}\cap t^{-\lan\la,\nu\ran}\La_{i,\la}),$$
for $i=1,2$. We have
$$\La_{i,\la}\cap t^{-\lan\la,\nu\ran}\La_{i,\la}=\begin{cases} t^{-\lan\la,\nu\ran}\La_{i,\la} &\lan\la,\nu\ran<0,\\ \La_{i,\la} &\lan\la,\nu\ran\ge 0.\end{cases}$$
Hence, 
$$\vol(\La_{i,\la}\cap t^{-\lan\la,\nu\ran}\La_{i,\la})/\vol(\La_{i,\la})=\begin{cases} q^{\lan\la,\nu\ran} &\lan\la,\nu\ran<0\\ 1 &\lan\la,\nu\ran\ge 0,\end{cases}$$
and so
$$\vol(\La_i\cap \nu(t)^{-1}\La_i)=\vol(\La_i)\cdot q^{\sum_{\lan\la,\nu\ran<0} \lan\la,\nu\ran},$$
for $i=1,2$. It follows that
$$\vol(\La_1)\cdot q^{\sum_{\lan\la,\nu\ran<0} \lan\la,\nu\ran}\le \vol(\La\cap \nu(t)^{-1}\La)\le \vol(\La_2)\cdot q^{\sum_{\lan\la,\nu\ran<0} \lan\la,\nu\ran}$$

It remains to use the fact that
$$C_1\cdot q^{\lan 2\rho,\nu\ran}\le \vol(G(O)\nu(t)G(O))\le C_2\cdot q^{\lan 2\rho,\nu\ran}$$
which follows from the formula
$$\vol(G(O)\nu(t)G(O))=\frac{|(G/P_\nu)(k)|}{q^{\dim(G/P_\nu)}}\cdot q^{\lan 2\rho,\nu\ran},$$
where $P_\nu\sub G$ is the standard parabolic associated with $\nu$ (see \cite[Prop.\ 7.4]{Gross}, \cite[3.3.1]{Tits}). 
\end{proof}

Let us consider several examples of the use of Proposition \ref{weight-crit-prop}.
First, we consider the case when $G=T$ is an algebraic torus.

\begin{prop} Let $\rho:T\to \GL(V)$ be a linear representation of a split algebraic torus $T$, such that $\ker(\rho)$ is finite.
Then the stack $[V/T]$ is $L^2$.
\end{prop}

\begin{proof}
In the case of tori, the weight criterion (see Proposition \ref{weight-crit-prop}) states that the $L^2$-property is equivalent to the convergence of the series
$$
\sum_{\nu\in X_\bullet(T)}q^{-\frac{1}{2}\sum_{\la:\lan \la,\nu\ran>0}\lan \la,\nu\ran\dim(V_\la)+
\frac{1}{2}\sum_{\la: \lan \la,\nu\ran<0}\lan \la,\nu\ran \dim(V_\la)}$$
If the representation is obtained by a restriction from a homomorphism of tori $T\to T'$ with the finite kernel, then the map
on cocharacters $X_\bullet(T)\to X_\bullet(T')$ is injective. Hence, convergence for $T'$ implies convergence fo $T$.
Thus, it is enough to check convergence for the action of the maximal torus of $\GL(V)$ on $V$. Then the sum becomes a power
of the corresponding sum for the $1$-dimensional case, 
$$1+2\sum_{n>0} q^{-n/2},$$
which converges.
\end{proof}

Next we consider the case when $V$ is the direct sum of several copies of the adjoint representation.

\begin{prop}\label{adjoint-prop} 
Let $G$ be a split semisimple group, $\fg$ its adjoint representation. Then $[\fg^{\oplus r}/G]$ is $L^2$ if and only if $r>1$.
\end{prop}

\begin{proof}
Let $\Phi$ be the set of roots, $\Phi_+$ the subset of positive roots.
By Proposition \ref{weight-crit-prop}, we have to study the convergence of the sum
$$\sum_{\nu\in X_\bullet(T)_+} q^{\lan 2\rho,\nu\ran+\sum_{\a\in \Phi:\lan \a,\nu\ran<0}\lan \a,\nu\ran\cdot r}.$$
But for a dominant coweight $\nu$ and a root $\a$, we have $\lan \a,\nu\ran<0$ only if $\a$ is a negative root.
Hence,
$$\sum_{\a\in \Phi:\lan \a,\nu\ran<0}\lan \a,\nu\ran=\sum_{\a\in \Phi_+}\lan -\a,\nu\ran=-\lan 2\rho,\nu\ran.$$
Thus, our sum is
$$\sum_{\nu\in X_\bullet(T)_+} q^{-(r-1)\lan 2\rho,\nu\ran}=\sum_{n_1\ge 0,\ldots,n_s\ge 0}q^{-2(r-1)\sum_i n_i}=\bigl(\sum_{n\ge 0} q^{-2(r-1)n}\bigr)^s,$$
where $s$ is the rank of $G$. This sum converges if and only if $r>1$.
\end{proof}

Next, let us consider the case $G=\SL_2$.

\begin{prop} Let $V$ be a finite dimensional representation of $\SL_2$. Then $[V/\SL_2]$ is $L^2$ if and only if $\sum_{n>0} n\dim V_n>2$,
where $V=\bigoplus_{n\in \Z} V_n$ is the weight decomposition.
\end{prop}

\begin{proof} Indeed, in this case by Proposition \ref{weight-crit-prop}, we should consider the series
$$\sum_{m\ge 0} q^{2m-\frac{1}{2}\sum_{n>0}mn\dim(V_n)+\frac{1}{2}\sum_{n<0}mn\dim(V_n)}=
\sum_{m\ge 0} q^{m(2-\sum_{n>0}n\dim(V_n))},$$
where we used the fact that $\dim(V_{-n})=\dim(V_n)$. This immediately gives the result.
\end{proof}


\subsection{Configuration stacks for $\P^1$}

\begin{prop}\label{conf-L2-prop}
The stack $[(\P^1)^r/\SL_2]$ is $L^2$ for any $r\ge 3$.
\end{prop}

\begin{proof}
Let $V$ be the standard $2$-dimensional representation of $\SL_2$.
It is enough to prove that $\XX_r:=[V^r/(\SL_2\times\G_m^r)]$ is $L^2$ for $r\ge 3$ (since $[(\P^1)^r/\SL_2]$ is an open substack in $\XX_r$).
Let us enumerate the dominant coweights of $\SL_2\times\G_m^r$ by $(n,n_1,\ldots,n_r)$, where $n\ge 0$ is the dominant coweight of $\SL_2$.
The natural basis of $V^r$ gives $2r$ functionals $n_i\pm n$ on these coweights. 
For $n\ge 0$, let $\ell_n(m)$ denote the sum of negative numbers among the pair $(m+n,m-n)$. Then
$$\ell_n(m)=\begin{cases} 2m, & m\le -n, \\ m-n, & -n\le m\le n, \\ 0, m\ge n.\end{cases}.$$

By the weight criterion (see Proposition \ref{weight-crit-prop}) we need to check the convergence of the series
$$\sum_{n\ge 0, n_1,\ldots,n_r} q^{2n-\sum_i n_i+\sum_i \ell_n(n_i)}=\sum_{n\ge 0}q^{2n}\cdot (\sum_m q^{\ell_n(m)-m})^r.$$
We have
\begin{align*}
&\sum_m q^{\ell_n(m)-m}=\sum_{m<-n} q^m + \sum_{-n\le m<n} q^{-n}+ \sum_{m\ge n} q^{-m}\\
&=(an+b)q^{-n},
\end{align*}
for some constants $a$ and $b$.
Hence, our original series is equal to
$$\sum_{n\ge 0}q^{(2-r)n}(an+b)^r,$$
which converges for $r\ge 3$.
\end{proof}

\begin{cor}
The stack $[(\P^1)^r/\PGL_2]$ is $L^2$ for any $r\ge 3$.
\end{cor}

\begin{proof} By Example \ref{adm-gerbe-ex}, the morphism $[(\P^1)^r/\SL_2]\to [(\P^1)^r/\PGL_2]$ is an admissible $\mu_2$-gerbe.
Hence, the assertion follows from Lemma \ref{gerbe-lem}.
\end{proof}

\subsection{Very good property}\label{vg-sec}

First, let us consider linear actions of $\SL_2$. 

\begin{prop}\label{SL2-very-good}
Let $V$ be a finite dimensional representation of $\SL_2$. Then $[V/\SL_2]$ is very good if and only if it is $L^2$.
\end{prop}

\begin{proof}
{\bf Step 1}. If $V$ and $W$ are representations of $G$, such that $[V/G]$ is very good then $[(V\oplus W)/G]$ is also very good. Indeed,
the locus of $(v,w)$ such that dimension of the stabilizer $G_{v,w}$ is $\ge m$ is contained in the locus of $(v,w)$ such that the dimension of $G_v$ is $\ge m$,
which has codimension $<m$. 

\noindent
{\bf Step 2}. Let $V(n)$ be the $(n+1)$-dimensional irreducible representation, where $n\ge 3$. Then $[V(n)/\SL_2]$ is very good.
Indeed, we can realize $V(n)$ as the space of binary forms of degree $n$. The forms with a stabilizer of positive dimension are equivalent to
$(xy)^{n/2}$ or to $x^n$. In the former case the stabilizer is $1$-dimensional, and the codimension of the corresponding locus is $n/2-1$. In the latter case
the stabilizer is also $1$-dimensional, and the codimension is $n-1$.

\noindent
{\bf Step 3}. The stacks $[V(1)^3/\SL_2]$, $[V(1)\oplus V(2)/\SL_2]$ and $[V(2)^2/\SL_2]$ are very good. 
Indeed, for $(v_1,v_2,v_3)\in V(1)^3$, the nontrivial stabilizer appears only when $v_1,v_2,v_3$ lie in a line $L\sub V(1)$. If they are not all zero then the stabilizer
is $1$-dimensional, but the locus of such $(v_1,v_2,v_3)$ has codiimension $2$. The stabilizer of $(0,0,0)$ is $\SL_2$ which is $3$-dimensional, whereas the codimension is $6$.

Let $(v,A)\in V(1)\oplus V(2)$ (where $V(2)$ is the adjoint representation). Assume that $v\neq 0$. Since the stabilizer of $v$ consists of unipotent elements, the stabilizer of $(v,A)$ is trivial,
unless $A$ is nilpotent. Furthermore, we have a nontrivial $1$-dimensional stabilizer only if $Av=v$. The codimension of this locus is $2$. On the other hand, if $v=0$ and $A\neq 0$ the stabilizer is also $1$-dimensional, and the codimension of this locus is $2$. 

For $(A,B)\in V(2)\oplus V(2)$, where $A\neq 0$, $B\neq 0$, the stabilizer is trivial, unless either both $A$ and $B$ belong to a Cartan subalgebra, or both belong to a unipotent radical of a Borel subalgebra. In both cases the centralizer is $1$-dimensional, and the codimension is $\ge 2$. In the cases when $A=0$ or $B=0$ the condition on codimension is easily seen to be satisfied.

\noindent
{\bf Step 4}. The stacks $[V(2)/\SL_2]$ and $[V(1)^2/\SL_2]$ are not very good.
Indeed, in the former case the generic stabilizer has positive dimension. In the latter case, the locus of $(v_1,v_2)$ that are linearly dependent, has codimension $1$ and generically
$1$-dimensional stabilizers.

The assertion easily follows from Steps 1--4.
\end{proof}

For a linear action of an algebraic torus $T$ on $V$, a very good property of $[V/T]$ implies that the generic stabilizer is finite, hence
the kernel of $\rho$ is finite. However, for a representation with finite kernel the stack $[V/T]$ is not necessarily very good, say, consider the standard action of $\G_m$ on $\A^1$.


Let us now consider the configuration stacks $[(\P^1)^r/\SL_2]$. 

\begin{lemma} The stack $[(\P^1)^r/\SL_2]$ is very good if and only if $r\ge 4$.
\end{lemma}

\begin{proof} The stabilizer of a configuration $(p_1,\ldots,p_r)$ is 
\begin{itemize}
\item finite if there are three distinct points among $p_i$'s;
\item $1$-dimensional (isomorphic to $\G_m$) if there are exactly two distinct points among the $p_i$'s;
\item $2$-dimensional (isomorphic to the Borel subgroup) if all $p_i$'s are the same.
\end{itemize}
The codimensions of the closures of the latter two loci are $r-2$ and $r-1$, respectively, which implies the assertion.
\end{proof}

Note that $[(\P^1)^3/\SL_2]$ is not very good but still $L^2$.

\section{Quasiparabolic bundles on $\P^1$}

\subsection{From Schwartz space to half-densities on the very stable locus}

Let $C$ be a curve of genus $g$ with parabolic points $p_1,\ldots,p_N$, where $N/2+g-1>0$, i.e., either $g\ge 2$ or $g=1$ and $N\ge 1$ or $g=0$ and $N\ge 3$.
For a fixed line bundle $L_0$ on $C$, let $\Bun_{L_0}(C,p_1,\ldots,p_N)$ denote the moduli stack of quasiparabolic bundles of rank $2$ with the
fixed determinant $L_0$. We denote by $\Bun^*_{L_0}(C,p_1,\ldots,p_N)$ where $*=ss$ (resp., $*=s$, resp., $*=vs$), the locus of semistable
(resp., stable, resp., very stable) quasiparabolic bundles (see appendix \ref{appendix-vs-sec} for details).

In the appendix \ref{appendix-vs-sec} we will prove the following parabolic analog of the main result of \cite{BKP}.

\begin{theorem}\label{vs-smooth-thm}
Assume that $N/2+g-1>0$ and $\Re(\kappa)\ge 1/2$.
For an admissible presentation $\pi:X\to [X/H]$ of an open substack of finite type in $\Bun_{L_0}(C,p_1,\ldots,p_N)$, let $X^{vs}\sub X$ denote the very stable locus.
Then for any $\varphi\in \SS(X(K),|\pi^*\om_{[X/H]}|^{\kappa}\ot |\om_\pi|)$, the push forward $\pi_*(\varphi|_{X^{vs}})$
of $\varphi|_{X^{vs}}$ to $(X^{vs}/H)(K)$ converges absolutely
as a (twisted) distribution, and defines a smooth $\kappa$-density on $(X^{vs}/G)(K)$. 
\end{theorem}


To apply the above Theorem in the case $C=\P^1$ we need to pass from rank $2$ bundles with fixed determinant to $\PGL_2$-bundles.
The stack $\Bun_{\PGL_2}(\P^1,p_1,\ldots,p_N)$ has two irreducible components, 
$\Bun^{ev}_{\PGL_2}(\P^1,p_1,\ldots,p_N)$ and $\Bun^{odd}_{\PGL_2}(\P^1,p_1,\ldots,p_N)$, corresponding to bundles of even and odd degree. 

\begin{lemma}
For $N\ge 1$, the natural maps
\begin{equation}\label{Bun-gerbe-maps}
\begin{array}{l}
\Bun_{\OO}(\P^1,p_1,\ldots,p_N)\to \Bun^{ev}_{\PGL_2}(\P^1,p_1,\ldots,p_N), \\ 
 \Bun_{\OO(1)}(\P^1,p_1,\ldots,p_N)\to \Bun^{odd}_{\PGL_2}(\P^1,p_1,\ldots,p_N),
 \end{array}
\end{equation}
are admissible $B\mu_2$-gerbes (in particular, they are essentially surjective on $K$-points). 
Hence, they induce isomorphisms of the coarse moduli spaces of stable bundles.
\end{lemma}

\begin{proof}
The morphisms \eqref{Bun-gerbe-maps} are $\mu_2$-gerbes, and any smooth covering of the source would provide a splitting of the $\mu_2$-gerbe. 
Hence, it suffices to show that the morphisms \eqref{Bun-gerbe-maps} are essentially surjective on $K$-points. 
For this it is enough to show that any $\PGL_2$-bundle on $\P^1_K$, equipped
with a $B$-structure at some $K$-point of $\P^1_K$, lifts to a rank $2$ vector bundle.

Let $K^s$ denote the separable closure of $K$. Given a $\PGL_2$-bundle $P$ on $\P^1_K$, let $P^s$ denote the corresponding $\PGL_2$-bundle on $\P^1_{K^s}$.
Since the Brauer group of $\P^1_{K^s}$ is trivial, $P^s$ lifts to a rank $2$ vector bundle $V^s$ on $\P^1_{K^s}$. 

Assume first that $V^s$ is trivial. Then $P$ is a form of the trivial $\PGL_2$-bundle, so it is given by a cohomology class $\a$ in $H^1(K,\PGL_2)$. The restriction of $P$ to any
$K$-point in $\P^1$ is still the $\PGL_2$-torsor over $\Spec(K)$ corresponding to the class $\a$. 
Since $P|_{p_1}$ admits a reduction of the structure group to the Borel subgroup $B\sub \PGL_2$, it follows that $\a$ is trivial. Hence, $P$ is trivial, so it admits a lifting to a rank $2$ vector bundle. 

Next, assume that $V^s\simeq \OO(n)\oplus \OO$, where $n>0$. Then $P^s$ has a $B$-structure, invariant under all automorphisms. Hence, $P$ also comes from a $B$-bundle $\FF$.
Let $\wt{B}\sub \GL_2$ be the Borel subgroup. Since the projection $\wt{B}\to B$ has a splitting, we can lift $\FF$ to $\wt{B}$-bundle, which provides a lifting of $P$ to a rank $2$ vector bundle.
\end{proof}

It follows that the morphisms \eqref{Bun-gerbe-maps} induce surjective maps on Schwartz spaces of $\kappa$-densities (see Lemma \ref{gerbe-lem}), and so Theorem \ref{vs-smooth-thm}
implies that a similar assertion holds for the stacks $\Bun_{\PGL_2}(\P^1,p_1,\ldots,p_N)$ (where $N\ge 3$).

Combining Theorem \ref{vs-smooth-thm} with Lemma \ref{very-stable-lem},
we deduce that our $L^2$-norm on the Schwartz space agrees with the standard $L^2$-norm on the $1/2$-densitiies over the very stable locus $\Bun^{vs}_{\PGL_2}(C,p_1,\ldots,p_N)$. 

\begin{cor}
For a presentation $\pi:X\to [X/H]$ of an open substack of finite type in $\Bun_{\PGL_2}(\P^1,p_1,\ldots,p_N)$, and for $\varphi\in \SS(X(K),|\om_X|^{1/2}\ot |\om_\pi|^{1/2})$, such that
$\varphi\ge 0$, one has
$$\lan \varphi,\varphi\ran=\int_{\Bun^{vs}_{\PGL_2}(\P^1,p_1,\ldots,p_N)(K)}(\pi_*(\varphi|_{X^{vs}}))^2$$
(where possibly both sides are $+\infty$).
\end{cor}

\subsection{Hecke operators}

For a point $x\in C(K)$, $x\neq p_i$, we consider the Hecke correspondence 
$$Z_x\rTo{\pi_1,\pi_2}\Bun_{\PGL_2}(C,p_1,\ldots,p_N)$$ 
which is the moduli stack of pairs of rank $2$ vector bundles  $(V_1\sub V_2)$,
such that $V_2/V_1\simeq \OO_x$ and the parabolic structures on $V_1$ and $V_2$ at $p_i$ are the same, up to tensoring with a line bundle. 
There is a natural isomorphism of line bundles on $Z_x$ (defined up to a rescaling),
\begin{equation}\label{main-lb-isom}
\a:\pi_1^*\om^{1/2}\rTo{\sim} \pi_2^*\om^{1/2}\ot \om_{\pi_2}
\end{equation}
(see \cite{EFK1}). Since $\pi_1$ and $\pi_2$ are smooth and proper (in fact, they are $\P^1$-bundles), this leads to a well defined operator
\begin{align}\label{S-Hecke-eq}
&H^S_x:\SS(\Bun_{\PGL_2}(C,p_1,\ldots,p_N)(K),|\om|^{1/2})\rTo{\pi_1^*}
\SS(Z_x(K),|\pi_1^*\om^{1/2}|)\rTo{|\a|} \nonumber \\
& \SS(Z_x(K),|\pi_2^*\om^{1/2}\ot \om_{\pi_2}|)\rTo{\pi_{2*}}\SS(\Bun_{\PGL_2}(C,p_1,\ldots,p_N)(K),|\om|^{1/2}).
\end{align}

Let $D\sub C$ be a positive divisor disjoint from $p_i$ and $p$, and let 
$$\pi:\Bun^D_{\PGL_2}(C,p_1,\ldots,p_N)\to \Bun_{\PGL_2}(C,p_1,\ldots,p_N)$$
denote the smooth covering corresponding to a choice of trivialization at $D$. 
As is explained in \cite{BKP2} there exists an exhaustive filtration $\UU_m\sub \UU_{m+1}\sub\ldots \Bun_{\PGL_2}(C,p_1,\ldots,p_N)$ by open substacks,
such that the Hecke correspondences induce correspondences of varieties,
$$\UU_m^D\lTo{\pi_1} Z_x^{m,D}\rTo{\pi_2} \UU_{m'}^D,$$
where $\UU_m^D=\UU_m\times_{\Bun_{\PGL_2}(C,p_1,\ldots,p_N)} \Bun^D_{\PGL_2}(C,p_1,\ldots,p_N)$,
with $\pi_1$ smooth and proper and $\pi_2$ smooth. Note that $\pi_1^*\om_{\pi}\simeq p_2^*\om_\pi$,
so we still have Hecke operators
$$H^S_x:\SS(\UU_m^D(K),|\om|^{1/2}\ot |\om_{\pi}|^{1/2})\to \SS(\UU_{m'}^D(K),|\om|^{1/2}\ot |\om_\pi|^{1/2})$$
compatible with the operators \eqref{S-Hecke-eq}.

Next, we specialize to the case $C=\P^1$ with $N\ge 4$ parabolic points $p_1,\ldots,p_N$.
Let $U=\Bun^{vs}_{\PGL_2}(\P^1,p_1,\ldots,p_N)$.
In \cite{EFK2}, for each point $x\in \P^1(K)$ ($x\neq p_i$) a bounded self-adjoint operator 
$$H_x:L^2(U(K),|\om|^{1/2})\to L^2(U(K),|\om|^{1/2})$$
is constructed, which agrees with $H^S_x$ on a certain dense subspace.
We need a slightly more precise result.


\begin{prop}\label{Hecke-comp-prop}
Let $\pi:X=\UU^D_m\to \Bun_{\PGL_2} (\P^1,p_1,\ldots,p_N)$ be the natural projection, and let $\pi:X^{vs}\to U$ be the induced map over the very stable locus. 
Let $\varphi\in \SS(X(K),|\om_X|^{1/2}\ot |\om_\pi|^{1/2})$ be such that $\pi_*(\varphi_{X^{vs}})$ is of class $L^2$.
Then for any $x\in \P^1(K)$, $x\neq p_i$, one has 
an equality of (twisted) distributions on $U(K)$,
$$H_x(\pi_*(\varphi|_{X^{vs}}))=\pi_*(H^S_x(\varphi)|_{X^{vs}}).$$
In particular, $\pi_*(H^S_x(\varphi)|_{X^{vs}})$ is also of class $L^2$.
\end{prop}

Before proving Proposition \ref{Hecke-comp-prop}, let us recall the definition of the operator $H_x$ on $L^2(U(K),|\om|^{1/2})$.

The correspondence $Z_x$ induces the correspondences
$$\Bun^{ev}_{\PGL_2}(\P^1,p_1,\ldots,p_N)\lTo{\pi_1}Z_x^{ev}\rTo{\pi_2}\Bun^{odd}_{\PGL_2}(\P^1,p_1,\ldots,p_N),$$
$$\Bun^{odd}_{\PGL_2}(\P^1,p_1,\ldots,p_N)\lTo{\pi_1}Z_x^{odd}\rTo{\pi_2}\Bun^{ev}_{\PGL_2}(\P^1,p_1,\ldots,p_N).$$
It is well known that $Z_x^{odd}$ is naturally identified with $Z_x^{ev}$ in such a way that the projections
$\pi_1$ and $\pi_2$ get swapped. Furthermore, isomorphism \eqref{main-lb-isom} induces an isomorphism
\begin{align*}
&\a':\pi_2^*\om^{1/2}\simeq (\pi_2^*\om\ot \om_{\pi_2})\ot (\pi_2^*\om^{1/2}\ot \om_{\pi_2})^{-1}\simeq 
(\pi_1^*\om\ot \om_{\pi_1})\ot (\pi_2^*\om^{1/2}\ot \om_{\pi_2})^{-1}\\
&\rTo{\a^{-1}} (\pi_1^*\om\ot \om_{\pi_1})\ot \pi_1^*\om^{-1/2}\simeq \pi_1^*\om^{1/2}\ot\om_{\pi_1},
\end{align*}
and under the isomorphism $Z_x^{odd}$ with $Z_x^{ev}$, $|\a|$ gets swapped with $|\a'|$.

Thus, it is enough to work with $Z_x^{ev}$ and study the corresponding two Hecke operators
$$\varphi\mapsto\pi_{2*}(|\a|\pi_1^*\varphi), \ \ \psi\mapsto \pi_{1*}(|\a'|\pi_2^*\psi)$$
between $\SS(\Bun^{ev}_{\PGL_2}(\P^1,p_1,\ldots,p_N)(K),|\om|^{1/2})$ and $\SS(\Bun^{odd}_{\PGL_2}(\P^1,p_1,\ldots,p_N)(K),|\om|^{1/2})$.

We can assume that $p_1=0$ and $p_N=\infty$.
We have an open embedding
$$[(\P^1)^N/\PGL_2]\sub \Bun^{ev}_{\PGL_2}(\P^1,p_1,\ldots,p_N)$$
corresponding to the locus of trivial bundles, and a further open embedding
$$X_{ev}=[\A^{N-2}/\G_m]\sub [(\P^1)^N/\PGL_2],$$
taking $(y_2,\ldots,y_{N-1})$ to the lines $(1,0),(1,y_2),\ldots,(1,y_{N-1}),(0,1)$ at $p_1,\ldots,p_N$
(see \cite[Sec.\ 3.1]{EFK2}).

Similarly, we have an open embedding
$$X_{odd}=[\A^{N-2}]/\G_m]\to \Bun^{odd}_{\PGL_2}(\P^1,p_1,\ldots,p_N)$$
sending $(z_2,\ldots,z_{N-1})$ to the bundle $\OO\oplus \OO(1)$ with the lines $(1,0),(1,z_2),\ldots,(1,z_{N-1}),(1,0)$ at $p_1,\ldots,p_N$
(see \cite[Sec.\ 3.2]{EFK2}).

The restriction of the $\P^1$-bundle $\pi_1:Z_x\to \Bun^{ev}_{\PGL_2}(\P^1,p_1,\ldots,p_N)$ over $X_{ev}$ is identified with
$$\pi_1^{-1}(X_{ev})\simeq [\A^{N-2}\times \P^1/\G_m]\to [\A^{N-2}/\G_m].$$
It is shown in \cite[Sec.\ 3.1]{EFK2} that over the open subset $(y_2,\ldots,y_{N-1};s)$ such that $s\neq y_i$, $s\neq 0,\infty$, one has
$$\pi_2(y_2,\ldots,y_{N-1},s)=
(z_2,\ldots,z_{N-1}), \ \ z_i=\frac{t_is-xy_i}{s-y_i},
$$
where $t_i$ is the coordinate of $p_i\in \P^1$.
Furthermore, the restriction of the isomorphism $|\a'|$ to the open subset $s\neq y_i$, $s\neq 0,\infty$ in $[\A^{N-2}\times \A^1/\G_m]\sub Z_x$ is given by
$$|\a'|(y_2,\ldots,y_{N-1};s)=|\prod_{i=1}^{N-1}(t_i-x)|^{1/2}\cdot \prod_{i=2}^{N-1}|s-y_i|^{-1}\cdot |s|^{N/2-2}\cdot |ds|.$$

Let $U_{ev}\sub X_{ev}$, and $U_{odd}\sub X_{odd}$ be open dense subschemes in the very stable loci, such that $\pi_1^{-1}(U_{ev})\simeq U_{ev}\times \P^1$.
The precise choice is not important since we will work up to measure zero subsets.
Then the operator 
$H_x:L^2(U_{odd}(K),|\om|^{1/2})\to L^2(U_{ev}(K),|\om|^{1/2})$ is given by
\begin{align*}
&H_xf(y_2,\ldots,y_{N-1})=\pi_{1*}(|\a'|(y_2,\ldots,y_{N-1};s)\cdot \pi_2^*\varphi)\\
&=\int_{s\in \P^1(K)}|\a'|(y_2,\ldots,y_{N-1};s)\cdot \varphi(\frac{t_2s-xy_2}{s-y_2},\ldots,\frac{t_{N-1}s-xy_{N-1}}{s-y_{N-1}}),
\end{align*}
where we identify half-densities on $X_{ev}$ (resp., $X_{odd}$) with functions of $y_2,\ldots,y_{N-1}$ of homogeneity degree $-(N-2)/2$.
(see \cite[Eq.\ (3.6)]{EFK2}). In \cite{EFK2}, $H_x$ is rewritten as an integral of unitary operators depending on $s\in\P^1$, with respect to a measure on $\P^1$
with finite total volume. This implies that $H_x$ is a bounded operator on $L^2$.

\begin{proof}[Proof of Proposition \ref{Hecke-comp-prop}]
We will give a proof for the operators from odd component to the even component. The proof of the other direction is similar.
 
Let us consider the diagram, in which two squares on the right are cartesian,
\begin{diagram}
\UU_{m'}^{D,ev}&\lTo{\pi_1}&Z_x^{m,D}&\rTo{\pi_2}&\UU_{m}^{D,odd}\\
\uTo{j}&&\uTo{j}&&\uTo{j}\\
U^{D,ev}&\lTo{\pi_1}&V^D&\rTo{\pi_2}&U^{D,odd}\\
\dTo{\pi}&&\dTo{\pi}&&\dTo{\pi}\\
U^{ev}&\lTo{\pi_1}&V&\rTo{\pi_2}&U^{odd}
\end{diagram}
where $V\sub U^{ev}\times\P^1$ is an open subset such that $\pi_2:V\to U^{odd}$ is regular, and
$V^D\sub U^{D,ev}\times\P^1$ is the preimage of $V$.

We start with $\varphi\in \SS(\UU_m^{D,odd}(K),|\om|^{1/2}\ot |\om_\pi|^{1/2})$. The argument below should be conducted in parallel for $\varphi$ and for $|\varphi|$
(to ensure absolute convergence where needed). 
Since $\pi_2$ is proper, $\pi_2^*\varphi\in \SS(Z_x^{m,D}(K),|\pi_2^*\om|^{1/2}\ot \pi_2^*|\om_\pi|^{1/2})$. 
We have 
$$H_x^S(\varphi)=\pi_{1*}(|\a'|\pi_2^*\varphi)$$
which is in $\SS(\UU^{D,ev}_{m'}(K),|\om|^{1/2}\ot |\om_\pi|^{1/2})$ since $\pi_1$ is smooth. Next, we restrict to the image of the open embedding $j:U^{D,ev}\hra \UU^{D,ev}_{m'}$.
Using the fact that $V^D$ is Zariski dense
in $\pi_1^{-1}(U^{D,ev})=U^{D,ev}\times \P^1$, we get
$$j^*H_x^S(\varphi)=\pi_{1*}(j^*|\a'|\pi_2^*\varphi)=\pi_{1*}(|\a'|\pi_2^*j^*\varphi)\in C^\infty(U^{D,ev}(K),|\om|^{1/2}\ot |\om_\pi|^{1/2}),$$
where the push-forward on the right converges.

Recall that by Theorem \ref{vs-smooth-thm}, $\pi_*(j^*\varphi)$ converges as a distribution and belongs to $C^\infty(U^{odd}(K),|\om|^{1/2})$. By assumption, $\pi_*(j^*\varphi)$ is of class $L^2$.
Next, we claim that 
\begin{equation}\label{H_x-interm-eq}
\pi_{1*}\pi_*(|\a'|\pi_2^*j^*\varphi)=H_x(\pi_*(j^*\varphi)),
\end{equation}
where the push-forward on the left converges as a distribution. 
Indeed, by the smooth base change (see Lemma \ref{integration-lem}(iii)), we have equality of twisted distributions on $V$,
$$\pi_*(\pi_2^*j^*\varphi)=\pi_2^*\pi_*(j^*\varphi).$$
Hence,
$$\pi_{1*}\pi_*(|\a'|\pi_2^*j^*\varphi)=\pi_{1*}(|\a'|\pi_*\pi_2^*j^*\varphi)=\pi_{1*}(|\a'|\pi_2^*\pi_*(j^*\varphi))=H_x(\pi_*(j^*\varphi)).$$

It follows that from \eqref{H_x-interm-eq} that we have an equality of (twisted) distributions on $U^{ev}$,
$$H_x(\pi_*(j^*\varphi))=\pi_{1*}\pi_*(|\a'|\pi_2^*j^*\varphi)=\pi_*\pi_{1*}(|\a'|\pi_2^*j^*\varphi)=\pi_*j^*H_x^S(\varphi),$$
where for the second equality we used Lemma \ref{integration-lem}(iv). This finishes the proof.

The last assertion follows from the boundedness of $H_x$ as an operator on the $L^2$ spaces.
\end{proof}

\subsection{The proof of the $L^2$ property}


\begin{proof}[Proof of Theorem A]
By Proposition \ref{L2-prop-def}, it is enough to consider the case $\kappa=1/2$.

Our starting point is that by Proposition \ref{conf-L2-prop}, the $L^2$ property holds for the open locus 
$$(\P^1)^N/\PGL_2\sub \Bun^{ev}_{\PGL_2}(\P^1,p_1,\ldots,p_N)$$ 
corresponding to the trivial bundle.

Next, we use the Hecke operators at a fixed point $x\neq p_i$. Let us define iteratively open substacks $\XX_i\sub\Bun_{\PGL_2}(\P^1,p_1,\ldots,p_N)$ as follows.
We set $\XX_0=(\P^1)^N/\PGL_2$, the locus where the bundle is trivial. If $\XX_i$ is already defined, we define $\XX_{i+1}$ as the image of $\XX_i$ under the correspondence $Z_x$
(i.e., $\XX_{i+1}$ consists of all elementary Hecke modifications of bundles in $\XX_i$). It is easy to see that the open substacks $(\XX_{2i})$ cover $\Bun^{ev}_{\PGL_2}(\P^1,p_1,\ldots,p_N)$
(resp., $(\XX_{2i+1})$ cover $\Bun^{odd}_{\PGL_2}(\P^1,p_1,\ldots,p_N)$).

For each $i$ and a sufficiently positive divisor $D\sub C$ (disjoint from $p_i$ and $x$), we have a presentation $\pi:\XX_i^D\to \XX_i$ corresponding to choices of trivialization at $D$.
By Lemma \ref{local-lem}, it is enough to prove that for all $x\in \XX_i^D(K)$ there exists a nonnegative $\varphi\in \SS(\XX_i^D(K),|\om|^{1/2}\ot |\om_\pi|^{1/2})$, such that $\varphi(x)>0$
and $\lan\varphi,\varphi\ran<+\infty$. We already know that this is so for $i=0$. Thus, it is enough to check that if this is true for $i$ then it is also true for $i+1$.

To this end we use Proposition \ref{Hecke-comp-prop} along with Lemma \ref{very-stable-lem}. Let $\varphi$ be a nonnegative element of $\SS(\XX_i^D(K),|\om|^{1/2}\ot |\om_\pi|^{1/2})$.
Then $H_x^S(\varphi)$ is a nonnegative element of $\SS(\XX_{i+1}^D(K),|\om|^{1/2}\ot |\om_\pi|^{1/2})$, positive on all Hecke modifications of the points in the support of $\varphi$.
Thus, it is enough that $\lan H_x^S(\varphi),H_x^S(\varphi)\ran<+\infty$. By Lemma \ref{very-stable-lem}, it is enough to check that $\pi_*(H_x^S(\varphi)|_{\XX_{i+1}^{D,vs}(K)})$ is
of class $L^2$ (where $\XX_{i+1}^{D,vs}\sub \XX_{i+1}^D$ is the very stable locus). By Proposition \ref{Hecke-comp-prop}, this follows from the fact that $\varphi$ is of class $L^2$.
\end{proof}

\appendix

\section{Schwartz spaces for quasiparabolic bundles and very stable locus}
\label{appendix-vs-sec}

We think of quasiparabolic structures at $p_1,\ldots,p_N\in C$ on a rank $2$ bundle $V$ over $C$, as parabolic structures corresponding to the weights $(0,1/2)$,
i.e., as collections of decreasing filtrations
$$V|_{p_i}=F^0V|_{p_i}\supset F^{1/2}V|_{p_i}\supset F^1V|_{p_i}=0,$$
where $F^{1/2}V|_{p_i}=\ell_i$ is a $1$-dimensional subspace in $V|_{p_i}$. 
So there is a $1$-dimensional weight $0$ subquotient $F^0V|_{p_i}/F^{1/2}V|_{p_i}$ and a $1$-dimensional weight $1/2$-subquotient $F^{1/2}V|_{p_i}/F^1V|_{p_i}$,
which makes the total weight of $V|_{p_i}$ equal to $1/2$.

Thus, the parabolic degree of $V$ is $\deg^{par}(V)=\deg(V)+\sum_i wt(V|_{p_i})=\deg(V)+N/2$. The parabolic slope is $\mu^{par}(V)=\deg^{par}(V)/2$.

Similarly, we consider parabolic structures on a line bundle $L$ at points $p_i$, given by the filtrations $L=F^0L|_{p_i}\supset F^{1/2}L|_{p_i}\supset F^1L|_{p_i}=0$,
so it is simply given by an assignment of a weight for each point $p_i$: we say that $L|_{p_i}$ has weight $0$ if $F^{1/2}L|_{p_i}=0$, and $L|_{p_i}$ has weight $1/2$ if $L=F^{1/2}L|_{p_i}$,
and $\deg^{par}(L)=\deg(L)+N_L/2$, where $N_L$ is the number of $p_i$ such that $wt(L|_{p_i})=1/2$.

We will only consider parabolic bundles of rank $\le 2$ below.
A parabolic morphism between parabolic bundles (of rank $\le 2$) is a morphism, compatible with filtrations.
For example, the map of line bundles $f:L\to M$ is parabolic if and only if for every $p_i$ such that $wt(L|_{p_i})=1/2$ and $wt(M|_{p_i})=0$, one has $f|_{p_i}=0$.
In other words,
$$\Hom_{par}(L,M)\simeq H^0(C,L^{-1}\ot M(-\sum_{i: wt(L|_{p_i})>wt(M|_{p_i})} p_i)).$$
In particular, this space can be nonzero only if $\deg^{par}(L)\le \deg^{par}(M)$.

Given a line subbundle $L\sub V$ in a parabolic rank $2$ bundle $(V,\ell_1,\ldots,\ell_N)$, there is a canonical parabolic structure on $L$, such that for any parabolic morphism
$f:V'\to V$ that factors through $L$ as a morphism of sheaves, $f$ is actually a composition of parabolic morphisms $V'\to L$ and $L\to V$.
Namely, for every $p_i$ such that $L|_{p_i}=\ell_i$, we set $wt(L|_{p_i})=1/2$, and we set $wt(L|_{p_j})=0$ for all other points $p_j$.

Dually, given a surjection $f:V\to L$, there is a canonical quotient parabolic structure on $L$, such that any parabolic morphism $V\to V'$ factors as a composition of
parabolic morphisms. Namely, for every $p_i$ such that $f|_{p_i}(\ell_i)=0$ we set $wt(L|_{p_i})=0$, and we set $wt(L|_{p_j})=1/2$ for all other points $p_j$.

For a pair of parabolic line bundles $L$ and $M$, we set
$$\Ext^1_{par}(M,L):=H^1(C,M^{-1}\ot L(-\sum_{wt(L|_{p_i})\le wt(M|_{p_i})}p_i)).$$
It is easy to see that we have the following version of Serre duality:
$$\Hom_{par}(L,M)^*\simeq \Ext^1_{par}(M,L\ot \om_C(\sum_{i=1}^N p_i)),$$
where the weights of $L\ot \om_C(\sum_{i=1}^N p_i)$ at $p_i$ are the same as those of $L$.

\begin{lemma} Let $(V,\ell_\bullet)$ be a rank $2$ parabolic bundle, and let 
$0\to L\to V\to M\to 0$ be an exact sequence, where $L$ and $M$ are line bundles. Then equipping $L$ and $M$ with induced parabolic structures,
one can associate naturally to this sequence an element $e\in \Ext^1_{par}(M,L)$, such that $e=0$ if and only if our exact sequence admits
 a parabolic splitting $V=L\oplus M$ (so for each $i$, either $\ell_i=L|_{p_i}$ or $\ell_i=M|_{p_i}$).
\end{lemma}

Stability (resp., semistability) of a parabolic bundle $V$ means that for any parabolic line subbundle $L\sub V$ one has $\deg^{par}(L)<\mu^{par}(V)$
(resp., $\deg^{par}(L)\le \mu^{par}(V)$).

Recall that a parabolic bundle $(V,\ell_1,\ldots,\ell_N)$ is called {\it very stable} if every nilpotent parabolic morphism 
$V\to V\ot\om_C(p_1+\ldots+p_N)$ is actually zero. Such a parabolic bundle is automatically stable.

\begin{lemma}\label{vs-lim-lem} 
Let $E$ be a very stable parabolic bundle of rank $2$. Assume that $\Hom_{par}(E,V)\neq 0$ for a parabolic bundle $V$ of rank $2$, such that
$\det(V)\simeq \det(E)$ and such that $V$ fits into
an exact sequence
$$0\to L\to V\to M\to 0,$$
where $\deg^{par}(M)\le \mu^{par}(V)$.
Then $\Ext^1_{par}(M,L)=0$, where $M$ and $L$ are equipped with induced parabolic structures.
\end{lemma}

\begin{proof}
Assume $\Ext^1_{par}(M,L)\neq 0$. Then $\Hom_{par}(L,M\ot \om(\sum_i p_i))\neq 0$.
Let us pick a nonzero parabolic morphism 
$$L\to M\ot \om(\sum_i p_i).$$
On the other hand, let us pick a nonzero parabolic morphism $E\to V$. The composition $E\to V\to M$ vanishes by stability of $E$ (note that $\mu^{par}(E)=\mu^{par}(V)$).
Hence, we get a nonzero parabolic morphism $E\to L$. Let $L'\sub L$ denote its image, which we equip with a parabolic structure as on a quotient of $E$.
Set $M':=\ker(E\to L')$, equipped with a canonical parabolic structure. We have
$$L\ot M\simeq \det(V)\simeq \det(E)\simeq L'\ot M'.$$
Hence, the embedding $L'\hra L$ induces an embedding $M\hra M'$, which is easily seen to be a parabolic morphism
(since $L'\hra L$ is a parabolic morphism and $wt(L|_{p_i})+wt(M|_{p_i})=1/2=wt(L'|_{p_i})+wt(M'|_{p_i})$).
Thus, we get a sequence of nonzero parabolic morphisms with the nonzero composition, 
$$\phi:E\to L'\to L\to M\ot \om(\sum_i p_i)\to M'\ot \om(\sum_i p_i)\to E\ot \om(\sum_i p_i).$$
Since $\phi^2=0$ this contradicts the assumption that $E$ is very stable. Hence, $\Ext^1_{par}(M,L)=0$.
\end{proof}

In the following analog of  \cite[Lemma 6.4]{BKP}) we consider parabolic bundles with a fixed determinant $L_0$.
For a point $x$ of a smooth stack $\XX$ we denote
by $\chi_{\om,x}$ the character of the group $\Aut(x)$ corresponding to the canonical line bundle on $\XX$. 
For an algebraic group $H$, we denote by
$\De^{alg}_H$ the algebraic modular character of $H$ (given by the inverse determinant of the adjoint action).

\begin{lemma}\label{weights-lem}
Let $V=A\oplus B$ be a parabolic sum of two line bundles, such that $\ell_i=A|_{p_i}$ for $i=1,\ldots,k$ and $\ell_i=B|_{p_i}$ for $i=k+1,\ldots,N$.
Let $x_V$ be the corresponding point of $\Bun_{L_0}(C,p_1,\ldots,p_N)$, where $L_0=A\ot B$.
Assume that $\deg^{par}(A)>\deg^{par}(B)$ and $\Ext^1_{par}(B,A)=0$. 
Let $\la^\vee:\G_m\to \Aut_{par}(V)$ denote the embedding such that $A$ has weight $1$ and $B$ has weight $-1$ with respect to $\la^\vee$.
Then
$$\lan\chi_{\om,x_V},\la^\vee\ran=2(\chi(AB^{-1}(-\sum_{i>k} p_i))-\chi(A^{-1}B(-\sum_{i\le k} p_i)))=4(\deg^{par}(A)-\deg^{par}(B)),$$
$$-\lan\De^{alg}_{Aut_{par}(V)},\la^\vee\ran=2\chi(AB^{-1}(-\sum_{i>k} p_i))=2(\deg^{par}(A)-\deg^{par}(B)-N/2-g+1).$$
In particular, for $\kappa\in\R$ such that $\kappa > \frac{1}{2}-\frac{N/2+g-1}{2(\deg^{par}(A)-\deg^{par}(B))}$, one has
$$\kappa\cdot \lan\chi_{\om,x_V},\la^\vee\ran\neq -\lan\De^{alg}_{Aut_{par}(V)},\la^\vee\ran.$$
\end{lemma}

\begin{proof}
Since $wt(A|_{p_i})=1/2$ and $wt(B|_{p_i})=0$ for $i\le k$ (resp., $wt(A|_{p_i})=0$ and $wt(B|_{p_i})=1/2$ for $i>k$),
we have
$$\und{\End}_{par,0}(V)\simeq \OO\oplus AB^{-1}(-\sum_{i>k} p_i)\oplus BA^{-1}(-\sum_{i\le k}p_i).$$
The $\la^\vee$-weights of this decomposition are $(0,2,-2)$, which leads to the formula for $\lan\chi_{\om,x_V},\la^\vee\ran$.

On the other hand, the Lie algebra of $\Aut_{par}(V)$ is
$$\End_{par,0}(V)=H^0(C,\OO)\oplus H^0(C,AB^{-1}(-\sum_{i>k}p_i)),$$
where we use the vanishing $\Hom_{par}(A,B)=H^0(BA^{-1}(-\sum_{i\ge k}p_i))=0$ due to the assumption on parabolic degrees of $A$ and $B$.
Since $H^1(C,AB^{-1}(-\sum_{i>k}p_i))=\Ext^1_{par}(B,A)=0$ by assumption, taking into account that
$$\deg^{par}(A)-\deg^{par}(B)=\deg(A)-\deg(B)+(k-(N-k))/2=\deg(A)-\deg(B)+k-N/2,$$
we get the claimed formula for $\lan\De^{alg}_{Aut_{par}(V)},\la^\vee\ran$.
\end{proof}

\bigskip

\begin{proof}[Proof of Theorem \ref{vs-smooth-thm}]
The proof is parallel to the proof of \cite[Thm.\ 6.8]{BKP}, using Lemmas \ref{vs-lim-lem} and \ref{weights-lem}.
Set $\Bun_{L_0}=\Bun_{L_0}(C,p_1,\ldots,p_N)$ for brevity. We will use the notion of a $|\om|^{\kappa}$-nice pair introduced in \cite{BKP}.

We consider the exhaustive filtration by open substacks, where $n$ is a half-integer,
$$\ldots \sub (\Bun_{L_0}^{\le n})^0\sub \Bun_{L_0}^{\le n}\sub (\Bun_{L_0}^{\le n+1/2})^0\sub\ldots \sub \Bun_{L_0},$$
where $\Bun_{L_0}^{\le n}$ consists of parabolic bundles $E$ such that $\deg^{par}(A)\le n$ for any line subbundle $A\sub V$ with the induced parabolic structure,
and $(\Bun_{L_0}^{\le n})^0\sub \Bun_{L_0}^{\le n}$ is the locus consisting of $E$ such that for any line subbundle $A\sub V$ with $\deg^{par}(A)=n$, one has 
$$\Ext^1_{par}(E/A,A)=H^1(L_0^{-1}A^2(-\sum_{i: wt(A|_{p_i})=0} p_i))=0$$
(where $A$ and $E/A$ are equipped with the induced parabolic structure).
The following properties are easily checked (similarly to \cite[Lem.\ 6.3]{BKP}):
\begin{itemize}
\item For $E\in \Bun_{L_0}^{\le n}$, one has $E\in (\Bun_{L_0}^{\le n})^0$ if and only if for any parabolic line bundle $A$ with $\deg^{par}(A)=n$, such that
$H^1(L_0^{-1}A^2(-\sum_{i: wt(A|_{p_i})=0} p_i))=0$, one has $\Hom_{par}(A,E)=0$.
\item For $E\in (\Bun_{L_0}^{\le n})^0\setminus \Bun_{L_0}^{\le n-1/2}$, one has a parabolic splitting $E\simeq A\oplus B$, where $\deg^{par}(A)=n$.
\item $(\Bun_{L_0}^{\le n})^0=\Bun_{L_0}^{\le n-1/2}$ unless $n\ge (g-1+N+\deg(L_0))/2$.
\item $(\Bun_{L_0}^{\le n})^0=\Bun_{L_0}^{\le n}$ for $n>g-1+(N+\deg(L_0))/2$.
\end{itemize}

Note that Lemma \ref{vs-lim-lem} implies that for $n\ge \deg(L_0)/2+N/4$, and for a very stable $E\in \Bun_{L_0}$, any $V$ in $\Bun_{L_0}^{\le n}$ with $\Hom_{par}(E,V)\neq 0$
is actually in $(\Bun_{L_0}^{\le n})^0$.

Now for an open substack of finite type $[X/H]\sub \Bun_{L_0}$, we have an induced filtration by $H$-invariant open subsets,
$$\ldots \sub (X^{\le n})^0\sub X^{\le n}\le (X^{\le n+1/2})^0\sub\ldots \sub X.$$
We want to show that the pair $(X^{vs},X^{\le n})$ is $|\om|^{\kappa}$-nice for every $n$, where $\Re(\kappa)\ge 1/2$,
where $X^{vs}\sub X$ is the very stable locus.

If $n<\deg(L_0)/2+N/4$ then this follows from the fact that $X^{\le n}$ is contained in the stable locus $X^s$.

Next, assume that $n<(g-1+N+\deg(L_0))/2$ and the assertion holds for $n-1/2$. Then closure of $G$-orbits of points in $X^{vs}$ are contained in $X^s$,
so the assertion follows from \cite[Lem.\ 5.4]{BKP}.

Next, assume that $n\ge (g-1+N+\deg(L_0))/2$ and the assertion holds for $n-1/2$, i.e., $(X^{vs},X^{\le n-1/2})$ is $|\om|^{\kappa}$-nice. 
We claim that the pair $(X^{vs},(X^{\le n})^0)$ is $|\om|^{\kappa}$-nice.
Indeed, since the bundles in $(X^{\le n})^0\setminus X^{\le n-1/2}$ are parabolically split, this follows from the weights computation in Lemma \ref{weights-lem} and
\cite[Lem.\ 5.2]{BKP}.

Finally, using Lemma \ref{vs-lim-lem} together with \cite[Lem.\ 5.4]{BKP}, we see that if the pair $(X^{vs},(X^{\le n})^0)$ is $|\om|^{\kappa}$-nice then so is
$(X^{vs},X^{\le n})$.
\end{proof}

\section{Configuration spaces for projective spaces}

Here we consider the $L^2$-property for stacks $[(\P^{n-1})^r/\SL_n]$.

\begin{prop} The stack $[(\P^{n-1})^r/\SL_n]$ is $L^2$ for $r>2(n-1)$.
\end{prop}

\begin{proof}
Similarly to the case of $\P^1$, we embed $[(\P^{n-1})^r/\SL_n]$ as an open substack into $[V^r/(\SL_n\times\G_m^r)]$,
where $V$ is the standard $n$-dimensional representation of $\SL_n$. We will prove that the latter linear quotient stack is $L^2$ using our criterion
Proposition \ref{weight-crit-prop}.
We denote the dominant coweights of $\SL_n\times\G_m^r$ as $(\nu,n_1,\ldots,n_r)$, where $\nu=(\nu_1\ge\ldots\ge\nu_n)$, $\sum \nu_i=0$.
The natural basis of $V^r$ gives $rn$ functionals on these coweights, given by 
$$(\nu,n_1,\ldots,n_r)\mapsto \nu_i+n_j, \ \ 1\le i\le n,\ 1\le j\le r.$$

For each dominant coweight $\nu$ of $\SL_n$, let us denote by $\ell_\nu$ the function on $\Z$, given by
$$\ell_\nu(m)=\sum_{i: \nu_i+m<0} (\nu_i+m).$$
By Proposition \ref{weight-crit-prop}, we need to check the convergence of
$$\sum_{\nu\ge 0, n_1,\ldots,n_r} q^{\lan 2\rho,\nu\ran-\frac{n}{2}\sum_j n_j+\sum_j \ell_\nu(n_j)}=
\sum_{\nu\ge 0}q^{\lan 2\rho,\nu\ran}\cdot (\sum_{m\in\Z} q^{\ell_\nu(m)-nm/2})^r.$$

It remains to prove that
$$\sum_{m\in\Z} q^{\ell_\nu(m)-nm/2}\le C\cdot \lan 2\rho,\nu\ran\cdot q^{-\lan \frac{\rho}{n-1},\nu\ran}$$
for some constant $C>0$ (independent of $\nu$).

We have 
$$\ell_\nu(m)=\begin{cases} nm, & m<-\nu_1,\\(m+\nu_k)+\ldots+(m+\nu_n), & -\nu_{k-1}\le m<-\nu_k,\\ 0,& m\ge -\nu_n.\end{cases}$$
Thus, it is enough to check a similar inequality for each of these ranges of $m$.
We have
$$\sum_{m<-\nu_1}q^{\ell_\nu(m)-nm/2}=\sum_{m<-\nu_1}q^{nm/2}=C\cdot q^{-\frac{n}{2}\nu_1}.$$
$$\sum_{m\ge -\nu_n}q^{\ell_\nu(m)-nm/2}=\sum_{m\ge -\nu_n}q^{-nm/2}=C\cdot q^{\frac{n}{2}\nu_n}.$$
$$\sum_{-\nu_{k-1}\le m<-\nu_k}q^{\ell_\nu(m)-nm/2}=\begin{cases}q^{\nu_k+\ldots+\nu_n}\cdot \frac{q^{-(n/2+1-k)\nu_k}-q^{-(n/2+1-k)\nu_{k-1}}}{q^{n/2+1-k}-1}, & k\neq n/2+1,\\
(\nu_{n/2}-\nu_{n/2+1})\cdot q^{\nu_{n/2+1}+\ldots+\nu_n}, & k=n/2+1.\end{cases}$$
Thus, for $k<n/2+1$, we have 
$$\sum_{-\nu_{k-1}\le m<-\nu_k}q^{\ell_\nu(m)-nm/2}\le C\cdot q^{-(\frac{n}{2}-k)\nu_k+\nu_{k+1}+\ldots+\nu_n}=C\cdot q^{-\nu_1-\ldots-\nu_{k-1}-(\frac{n}{2}+1-k)\nu_k},$$
while for $k>n/2+1$,
$$\sum_{-\nu_{k-1}\le m<-\nu_k}q^{\ell_\nu(m)-nm/2}\le C\cdot q^{(k-\frac{n}{2}-1)\nu_{k-1}+\nu_k+\ldots+\nu_n}.
$$

Now the assertion for odd $n$ follows from the inequalities
$$-\frac{n}{2}\nu_1\le -\nu_1-(\frac{n}{2}-1)\nu_2\le\ldots\le -\nu_1-\ldots-\nu_{(n-1)/2}-\frac{1}{2}\nu_{(n+1)/2}=
-\frac{1}{2}\lan \sum_{i=1}^{(n-1)/2} e_i-\sum_{i=(n+1)/2+1}^n e_i,\nu\ran,$$
$$\frac{n}{2}\nu_n\le (\frac{n}{2}-1)\nu_{n-1}+\nu_n\le \frac{1}{2}\nu_{(n+1)/2}+\nu_{(n+1)/2+1}+\ldots+\nu_n=
-\frac{1}{2}\lan \sum_{i=1}^{(n-1)/2} e_i-\sum_{i=(n+1)/2+1}^n e_i,\nu\ran,$$
$$-\frac{1}{2}\lan \sum_{i=1}^{(n-1)/2} e_i-\sum_{i=(n+1)/2+1}^n e_i,\nu\ran \le -\lan \frac{\rho}{n-1},\nu\ran.$$

Similarly, for even $n$ we use the inequalities
$$-\frac{n}{2}\nu_1\le -\nu_1-(\frac{n}{2}-1)\nu_2\le\ldots\le-\nu_1-\ldots-\nu_{n/2}=-\frac{1}{2}\lan \sum_{i=1}^{n/2} e_i-\sum_{i=n/2+1}^n e_i,\nu\ran,$$
$$\frac{n}{2}\nu_n\le (\frac{n}{2}-1)\nu_{n-1}+\nu_n\le \nu_{n/2+1}+\ldots+\nu_n=-\frac{1}{2}\lan \sum_{i=1}^{n/2} e_i-\sum_{i=n/2+1}^n e_i,\nu\ran,$$
$$-\frac{1}{2}\lan \sum_{i=1}^{n/2} e_i-\sum_{i=n/2+1}^n e_i,\nu\ran\le -\lan \frac{\rho}{n-1},\nu\ran.$$
\end{proof}


\begin{thebibliography}{9}
\bibitem{AT} F.~Ambrosino, J.~Teschner, {\it Analytic Langlands correspondence from SoV}, arXiv:2510.06991.
\bibitem{BD} A.~Beilinson, B.~Drinfeld, {\it Quantization of Hitchin's integrable system and Hecke eigensheaves},
preprint http://www.math.uchicago.edu/$\tilde{\phantom{x}}$drinfeld/langlands/QuantizationHitchin.pdf
\bibitem{BK} A.~Braverman, D.~Kazhdan, {\it Automorphic functions on moduli spaces of bundles on curves over local fields: a survey}, in {\it International Congress of Mathematicians. Vol. 2. Plenary lectures}, pp. 796--824, EMS Press, Berlin, 2023.
\bibitem{BKP} A.~Braverman, D.~Kazhdan, A.~Polishchuk, {\it Schwartz $\kappa$-densities for the moduli stack of rank $2$ bundles on a curve over a local field}, arXiv:2401.01037
\bibitem{BKP2} A.~Braverman, D.~Kazhdan, A.~Polishchuk, {\it Hecke operators for curves over non-archimedean local fields and related finite rings}, arXiv:2305.09595
(or IMRN, Vol. 2025, Issue 7 (April 2025) plus Erratum to appear).
\bibitem{EFK1} P.~Etingof, E.~Frenkel, D.~Kazhdan, {\it Hecke operators and analytic Langlands correspondence for curves over local fields},
Duke Math. J. 172 (2023), no. 11, 2015--2071.
\bibitem{EFK2} P.~Etingof, E.~Frenkel, D.~Kazhdan, {\it Analytic Langlands correspondence for $\PGL_2$ on $\P^1$ with parabolic structures over local fields},
Geom. Funct. Anal. 32 (2022), no. 4, 725--831.
\bibitem{GK} D.~Gaitsgory, D.~Kazhdan, {\it Algebraic groups over a $2$-dimensional local field: some further constructions}, in 
{\it Studies in Lie theory}, 97--130, Birkh\"auser Boston, MA, 2006.
\bibitem{Gross} B.~Gross, {\it On the Satake isomorphism}, in {\it Galois representations in arithmetic algebraic geometry (Durham, 1996)}, 223--237,
Cambridge Univ. Press, Cambridge, 1998.
\bibitem{KP} D.~Kazhdan, A.~Polishchuk,  {\it Schwartz $\kappa$-densities on the moduli stack of rank $2$ bundles near stable bundles}, arXiv:2408.10551.
\bibitem{Klyuev} D.~Klyuev, {\it Analytic Langlands correspondence for $\PGL_2(\C)$ on a genus one curve with parabolic structures}, arXiv:2311.13030.
\bibitem{Sak} Y.~Sakellaridis, {\it The Schwartz space of a smooth semi-algebraic stack}, Selecta Math. (N.S.) 22 (2016), no. 4, 2401--2490.
\bibitem{Tits} J.~Tits, {\it Reductive groups over local fields}, in {\it Automorphic forms, representations and L-functions (Corvallis, Ore., 1977)}, Part 1, pp. 29--69.
AMS, Providence, RI, 1979.
\end{thebibliography}
\end{document}